\documentclass[12pt]{article}

\usepackage{amsthm,amsmath,amssymb,bbm}
\usepackage{natbib}
\usepackage{multirow}
\usepackage[pdftex]{graphicx}
\usepackage{subfigure}
\usepackage{makecell}
\usepackage{booktabs}
\usepackage{array}
\usepackage{url}
\usepackage{algorithm}
\usepackage{algorithmic}
\usepackage{bm}
\usepackage{wrapfig}
\usepackage{lipsum}
\usepackage{mathrsfs}
\usepackage{dsfont}
\usepackage{titling}

\usepackage{setspace}

\usepackage[usenames,dvipsnames,svgnames,table]{xcolor}
\usepackage[colorlinks,
linkcolor=red,
anchorcolor=blue,
citecolor=blue
]{hyperref}

\newcommand{\sgn}{\mathop{\mathrm{sign}}}

\usepackage{smile}

\newcommand*{\supp}{\mathrm{supp}}
\newcommand*{\var}{\textnormal{var}}

\newcommand{\e}{\mathbb{E}}
\newcommand{\nn}{\nonumber}

\newcommand \btt{\bbeta}
\newcommand \hbt{\hat{\btt}}
\newcommand \bttc{\bbeta^*}

\def\##1\#{\begin{align}#1\end{align}}
\def\$#1\${\begin{align*}#1\end{align*}}

\def\T{\mathrm{\scriptstyle T}} %%%transpose operator
\def\sn{\sum_{i=1}^n}
\def\Sb{\mathbf{S}}

\newcommand{\red}[1]{\textcolor{red}{#1}}
\newcommand{\normal}[1]{\textnormal{#1}}

\newcommand{\wt}{\widetilde}
\newcommand{\bfsym}[1]{\ensuremath{\boldsymbol{#1}}}
       \def \bbeta    {\bfsym{\beta}}
       \def \bdelta   {\bfsym{\delta}}

\newcommand{\eff}{\textnormal{eff}}

\newcommand{\Rom}[1]{\text{\uppercase\expandafter{\romannumeral #1\relax}}}

\usepackage{geometry}
\usepackage{enumitem}
\usepackage{chngcntr}
\counterwithout{equation}{section}

\usepackage[flushleft]{threeparttable}
\begin{document}

\title{Adaptive Huber Regression%: Nonasymptotic Optimality and Phase Transition
\thanks{Qiang Sun is Assistant Professor, Department of Statistical Sciences, University of Toronto, Toronto, ON M5S 3G3, Canada (E-mail: \href{mailto:qsun@utstat.toronto.edu}{\textsf{qsun@utstat.toronto.edu}}).
Wen-Xin Zhou is Assistant Professor, Department of Mathematics, University
of California, San Diego, La Jolla, CA 92093 (E-mail: \href{mailto:wez243@ucsd.edu}{\textsf{wez243@ucsd.edu}}).
Jianqing Fan is Honorary Professor, School of Data Science, Fudan University, Shanghai, China and
Frederick L. Moore '18 Professor of Finance, Department of Operations Research and Financial Engineering,
Princeton University, NJ 08544 (E-mail:  \href{mailto:jqfan@princeton.edu}{\textsf{jqfan@princeton.edu}}).
}
\vspace{10pt}}

\author{Qiang Sun,~Wen-Xin Zhou,~and~Jianqing Fan}

\date{}
\maketitle

\vspace{-.5in}

\begin{abstract}
Big data can easily be contaminated by outliers or contain variables with heavy-tailed distributions, which makes many conventional methods inadequate. To address this challenge, we propose the adaptive Huber regression for robust estimation and inference. The key observation is that the robustification parameter should adapt to the sample size, dimension and moments for optimal tradeoff between bias and robustness. Our theoretical framework deals with heavy-tailed distributions with bounded $(1+\delta)$-th moment for any $\delta > 0$. We establish a sharp phase transition for robust estimation of regression parameters in both low and high dimensions: when $\delta  \geq 1$, the estimator admits a sub-Gaussian-type deviation bound without sub-Gaussian assumptions on the data, while only a slower rate is available in the regime  $0<\delta< 1$. Furthermore, this transition is smooth and optimal. We extend the methodology to allow both heavy-tailed predictors and observation noise. Simulation studies lend further support to the theory. In a genetic study of cancer cell lines that exhibit heavy-tailedness, the proposed methods are shown to be more robust and predictive.

%and more predictive and %Both simulation studies and a real data application to cancel cell lines lend further support to our obtained theory. %In a genetic study of cancel cell lines, we find our methods more robust for revealing   \red{We will use the NCI cell line data to validate the methodology.}  %Lastly, we extend our methodology to
\end{abstract}


\section{Introduction}\label{sec:1}

Modern data acquisitions have facilitated the collection of massive and high dimensional data with complex structures. Along with holding great promises for discovering subtle population patterns that are less achievable with small-scale data, big data have introduced a series of new challenges to data analysis both computationally and statistically \citep{loh2015regularized, fan2015tac}. During the last two decades, extensive progress has been made towards extracting useful information from massive data with high dimensional features and sub-Gaussian tails\footnote{A random variable $Z$ is said to have sub-Gaussian tails if there exists constants $c_1$ and $c_2$ such that $\PP(|Z|>t)\leq c_1\exp(-c_2t^2)$ for any $t \geq 0$. } \citep{tibs1996regression, fan2001variable, lars2004, bickel2009simultaneous}. We refer to the monographs, \cite{BvdG2011} and \cite{HTW2015}, for a systematic coverage of contemporary statistical methods  for high dimensional data.

The sub-Gaussian tails requirement, albeit being convenient for theoretical analysis, is not realistic in many practical applications since modern data are often collected with low quality.
For example, a recent study on functional magnetic resonance imaging  (fMRI) \citep{eklund2016cluster} shows that the principal cause of invalid fMRI inferences is that the data do not follow the assumed Gaussian shape, which speaks to the need of validating the statistical methods being used in the field of neuroimaging. In a microarray data example considered in \cite{Lan2015}, it is observed that some gene expression levels have heavy tails as their kurtosises are much larger than 3, despite of the normalization methods used.
In finance, the power-law nature of the distribution of returns has been validated as a stylized fact \citep{C2001}. \cite{fan2016robust} argued that heavy-tailed distribution is a stylized feature for high dimensional data and proposed a shrinkage principle to attenuate the influence of outliers.
Standard statistical procedures that are based on the method of least squares often behave poorly in the presence of heavy-tailed data\footnote{We say a random variable $X$ has heavy tails if $\PP(|X| > t)$ decays to zero polynomially in $1/t$ as $t \to \infty$.
} \citep{catoni2012challenging}. It is therefore of ever-increasing interest to  develop  new statistical methods that are robust against heavy-tailed errors and other potential forms of contamination.

In this paper, we first revisit  the robust regression that was initiated by Peter Huber in his seminal work \cite{Huber1973}.  Asymptotic properties of the Huber estimator have been well studied in the literature. We refer to \cite{Huber1973}, \cite{YohaiMaronna1979}, \cite{Portnoy1985}, \cite{Mammen1989} and \cite{HeShao1996, HeShao2000} for an unavoidably incomplete overview.  However, in all of the aforementioned papers, the robustification parameter is suggested to be set as fixed according to the 95\% asymptotic efficiency rule.  Thus, this procedure can not estimate the model-generating parameters consistently when the sample distribution is asymmetric.

From a nonasymptotic perspective (rather than an asymptotic efficiency rule),
we propose to use the Huber regression  with an adaptive {robustification parameter}, which is referred to as the {\it adaptive Huber regression}, for robust  estimation and inference. Our adaptive procedure achieves the nonasymptotic robustness in the sense that the resulting estimator admits exponential-type concentration bounds when only low-order moments exist.  Moreover, the resulting  estimator  is also an asymptotically  unbiased estimate for the parameters of interest.   In particular, we do not impose symmetry and homoscedasticity conditions on error distributions, so that our problem is intrinsically different from median/quantile regression models, which are also of independent interest and serve as important robust techniques \citep{K2005}.

We made several major contributions towards robust modeling in this paper. First and foremost, we establish nonasymptotic deviation bounds for adaptive Huber regression when the error variables have only finite $(1+\delta)$-th moments. By providing a matching lower bound, we observe a sharp phase transition phenomenon, which is in line with that discovered by \cite{devroye2015sub} for univariate mean estimation.
Second, a similar phase transition  for regularized adaptive  Huber regression is established in high dimensions. By defining the effective dimension and effective sample size, we present nonasymptotic results under the two different regimes in a unified form.  Last, by exploiting the localized analysis  developed in \cite{fan2015tac}, we remove the artificial bounded parameter constraint imposed in previous works; see \cite{loh2015regularized} and \cite{fan2016estimation}. In the supplementary material, we present a nonasymptotic  Bahadur representation for the adaptive Huber estimator when $\delta\geq 1$, which provides a theoretical foundation for robust finite-sample inference.

The rest of the paper proceeds as follows. The rest of this section is devoted to related literature. In Section \ref{sec:2}, we revisit the Huber loss and robustification parameter, followed by the proposal of adaptive Huber regression in both low and high dimensions. We sharply characterize the nonasymptotic performance  of the proposed estimators in Section \ref{sec:3}. We describe the algorithm and implementation in Section~\ref{sec:3plus}. Section \ref{sec:4} is devoted to simulation studies and a real data application. In Section \ref{sec:ext}, we extend the methodology to allow possibly heavy-tailed covariates/predictors.
All the proofs are collected  in the supplemental material.

\subsection{\bf Related Literature} The terminology ``robustness" used in this paper describes how stable the method performs with respect to the tail-behavior of the data, which can be either sub-Gaussian/sub-exponential or Pareto-like \citep{DHJ2011, catoni2012challenging, devroye2015sub}. This is different from the conventional perspective of robust statistics under Huber's $\epsilon$-contamination model \citep{Huber1964}, for which a number of depth-based procedures have been developed since the groundbreaking work of John Tukey \citep{T1975}. Significant contributions have also been made  in \cite{Liu1990}, \cite{LPS1999},  \cite{ZS2000}, \cite{M2002} and \cite{MM2004}. We refer to \cite{CGR2017} for the most recent result and a literature review concerning this problem.

Our main focus is on the conditional mean regression in the presence of heavy-tailed and asymmetric errors, which automatically distinguishes our method from quantile-based robust regressions \citep{K2005, BC2011, W2013, fan2014adaptive, ZPH2015}. In general, quantile regression is biased towards estimating the mean regression coefficient unless  the error distributions are symmetric around zero. Another recent work that is  related to ours is \cite{ACL2017}.
They studied a general class of regularized empirical risk minimization procedures with a particular focus on Lipschitz losses, %Their methods do not rely on any truncation techniques and therefore the knowledge of the variance of observation noise. 
%They imposed a so-called ``Bernstein condition'' on the loss function, 
which includes the quantile, hinge and logistic losses. 
%They also obtained minimax rates of convergence for a wide range of problems, while our goal is to design new robust estimators by exploring the interaction between their performance and the moment condition.
Different from all these work, our goal is to estimate the mean regression coefficients robustly.  The robustness is witnessed by a nonasymptotic analysis: the proposed estimators achieve sub-Gaussian deviation bounds when the regression errors have  only finite second moments. Asymptotically, our proposed estimators are fully efficient: they achieve the same efficiency as the ordinary least squares estimators.  %See Section~\ref{sec:5} for detailed comparisons.
%\cite{HI2017} investigated robust estimation from a computational standpoint. Specifically, the authors employed an approximate minimization technique to efficiently make use of computationally unwieldy robust losses. In general, the use of robust losses can bring substantial gains in stability at a tolerable price in terms of bias. 
%In this paper, to highlight the importance of tail-adaptivity, we focus on the Huber loss for which extreme observations are re-weighted by a slower-growing  function, the $\ell_1$-loss function. The intuition is that, among a class of widely used robust losses, the Huber loss is closest to the quadratic loss. The closeness is controlled by a robustification parameter, which plays a critical role in leveraging the gain in stability and cost in bias.

An important step towards estimation under heavy-tailedness has been made by \cite{catoni2012challenging}, whose focus is on estimating a univariate mean. Let $X$ be a real-valued random variable with mean $\mu = \EE(X)$ and variance $\sigma^2=\var (X)>0$, and assume that $X_1,\ldots, X_n$ are independent and identically distributed (i.i.d.) from $X$. For any prespecified exception probability $t>0$, Catoni constructs a robust mean estimator $\widehat\mu_{{\rm C}}(t)$ that deviates from the true mean $\mu$ logarithmically in $1/t$, that is,
\# \label{catoni.bound}
	 \PP\big[   |  \hat{\mu}_{{\rm C}}(t) - \mu | \leq t \sigma/  n^{1/2} \big] \geq 1- 2\exp(-c t^2),
\#
while the empirical mean deviates from the true mean only polynomially in $1/t^2$, namely subGaussian tails versus Cauchy tail in terms of $t$. Further, \cite{devroye2015sub} developed adaptive sub-Gaussian estimators that are independent of the prespecified exception probability.
Beyond mean estimation, \cite{brownlees2015empirical} extended Catoni's idea to study empirical risk minimization problems when the losses are unbounded. Generalizations of the univariate results to those for matrices, such as the covariance matrices, can be found in \cite{catoni2016pac}, \cite{minsker2016sub},  \cite{giulini2016robust} and \cite{fan2016estimation}.  \cite{fan2016estimation} modified Huber's procedure \citep{Huber1973} to obtain a robust estimator, which is concentrated around the true mean with exponentially high probability in the sense of \eqref{catoni.bound}, and also proposed a robust procedure for sparse linear regression with asymmetric and heavy-tailed errors.

\medskip
\noindent{\bf Notation}: We fix some notations that will be used throughout this paper. For any vector $\bu = ( u_1, \ldots, u_d)^\T \in \RR^d$ and $q \geq 1$, $\|\bu\|_q= (\sum_{j=1}^d | u_j|^q )^{1/q}$ is the $\ell_q$ norm. For any vectors $\bu , \bv \in \RR^d$, we write $\langle \bu, \bv \rangle = \bu^\T \bv$. Moreover, we let $\|\bu\|_0 = \sum_{j=1}^d 1( u_j \!\neq\! 0 )$ denote the number of nonzero entries of $\bu$, and set $\|\bu\|_\infty=\max_{1\leq j\leq d}| u_j |$. For two sequences of real numbers $\{ a_n \}_{n\geq 1}$ and $\{ b_n \}_{n\geq 1}$, $a_n \lesssim b_n$ denotes $a_n \leq C b_n$ for some constant $C>0$ independent of $n$, $a_n \gtrsim b_n$ if $b_n \lesssim a_n$, and $a_n \asymp b_n$ if $a_n \lesssim b_n$ and $b_n \lesssim a_n$. For two scalars, we use $a\wedge b=\min\{a, b\}$ to denote the minimum of $a$ and $b$. If $\Ab$ is an $m\times n$ matrix, we use $\| \Ab \|$ to denote its spectral norm, defined by $\| \Ab \| = \max_{ \bu \in \mathbb{S}^{n-1}} \| \Ab \bu \|_2$, where $\mathbb{S}^{n-1} = \{ \bu \in \RR^n : \|\bu \|_2 = 1 \}$ is the unit sphere in $\RR^n$. For an $n\times n$ matrix $\Ab$, we use $\lambda_{\max}(\Ab)$ and $\lambda_{\min}(\Ab)$ to denote the maximum and minimum eigenvalues of $\Ab$, respectively. For two $n\times n$ matrices $\Ab$ and $\Bb$, we write $\Ab \preceq \Bb$ if $\Bb-\Ab$ is positive semi-definite. For a function $f: \RR^d \to \RR $, we use $\nabla f \in \RR^d$ to denote its gradient vector  as long as it exists.

\section{Methodology}\label{sec:2}

We consider i.i.d. observations $(y_1 , \bx_1 ), \ldots, (y_n,\bx_n)$ that are generated from the following heteroscedastic regression model
\#
    y_i =  \langle \bx_i ,  \bbeta^* \rangle  + \varepsilon_i, ~\mbox{ with }~  \EE (\varepsilon_i | \bx_i)=0  ~\mbox{ and }~
    v_{i, \delta} = \EE \big( |\varepsilon_i|^{1+\delta} \big) < \infty.  \label{linear.model}
\#
Assuming that the second moments are bounded ($\delta=1$), the standard ordinary least squares (OLS) estimator, denoted by $\widehat\bbeta^{\textnormal{ols}}$,  admits a suboptimal polynomial-type deviation bound, and thus does not concentrate around $\bbeta^*$ tightly enough for large-scale simultaneous estimation and inference. The key observation that underpins this suboptimality of the OLS estimator is the sensitivity of quadratic loss to outliers \citep{Huber1973, catoni2012challenging}, while the Huber regression with a fixed tuning constant may lead to nonnegligible estimation bias. To overcome this drawback, we propose to employ the Huber loss with an adaptive robustification parameter to achieve robustness and (asymptotic) unbiasedness simultaneously.
We begin with the definitions of the Huber loss and the corresponding robustification parameter.

\begin{definition}[{\sf Huber Loss and Robustification Parameter}] \label{Huber.def}
The  Huber loss $\ell_\tau(\cdot)$ \citep{Huber1964} is defined as
$$
	\ell_\tau(x) =
	\left\{\begin{array}{ll}
	 x^2 /2 ,    & \mbox{if } |x | \leq \tau ,  \\
	\tau |x | -  \tau^2 /2 ,   &  \mbox{if }  |x | > \tau ,
	\end{array}  \right. 	
$$	
where $\tau>0$ is referred to as the {robustification parameter} that balances bias and robustness \citep{fan2016estimation}.
\end{definition}

  The loss function $\ell_\tau(x)$ is quadratic for small values of $x$, and becomes linear when $x$ exceeds $\tau$ in magnitude. The parameter $\tau$ therefore controls the blending of quadratic and $\ell_1$ losses, which can be regarded as two extremes of the Huber loss with $\tau=\infty$ and $\tau\rightarrow0$, respectively. Comparing with the least squares, outliers are down weighted  in the Huber loss. We will use the name, {\it adaptive Huber loss}, to emphasize the fact that the parameter $\tau$ should adapt to the sample size, dimension and moments  for a better tradeoff between bias and robustness. This distinguishes our framework from the classical setting. As $\tau \to \infty$ is needed to reduce the bias when the error distribution  is asymmetric, this loss is also called the RA-quadratic (robust approximation to quadratic) loss in \cite{fan2016estimation}.

Define the empirical loss function $\cL_\tau(\bbeta)= n^{-1}\sum_{i=1}^n \ell_{\tau}(y_i-  \langle \bx_i , \bbeta \rangle)$ for $\bbeta \in \RR^d$. The Huber estimator is defined through the following convex optimization problem:
\begin{align}
\hat{\bbeta}_\tau = \arg\min_{\bbeta\in \RR^d}   \cL_\tau(\bbeta).  \label{Huber.est}
\end{align}
In low dimensions, under the condition that $v_{\delta}= n^{-1} \sn \EE  (|\varepsilon_i |^{1+\delta} ) < \infty$ for some $\delta >0$, we will prove that $\hat{\bbeta}_\tau$ with $ \tau \asymp \min \{ v_\delta^{1/(1+\delta)}, v_1^{1/2}  \} \, n^{\max\{1/(1+\delta),1/2\}}$ (the first factor is kept in order to show its explicit dependence on the moment) achieves the  tight upper bound  $d^{1/2} \tau^{-(\delta\wedge 1)}\asymp  d^{1/2} n^{-\min\{\delta/(1+\delta), 1/2\}}$.  The phase transition at $\delta = 1$ can be easily observed (see Figure \ref{fig:1}).  %Proposition~\ref{prop:error} in Section \ref{sec:3.1} indicates that the bias introduced by the Huber loss scales at the order $\tau^{-  \delta  }$.  This suggests that when $0<\delta \leq 1$, the approximation bias matches the optimal rate of convergence, which is of order $n^{-  \delta /(1+\delta)  }$, when holding $d$ fixed.
When higher moments exist ($\delta\geq 1$), robustification leads to a sub-Gaussian-type deviation inequality in the sense of \eqref{catoni.bound}. %When higher moments exist, that is  $\delta>1$, the approximation bias decays at a faster rate and is negligible compared to the estimation error, which is in the order of $ n^{-1/2}$.

\begin{figure}[th!]
	\includegraphics[scale=.45]{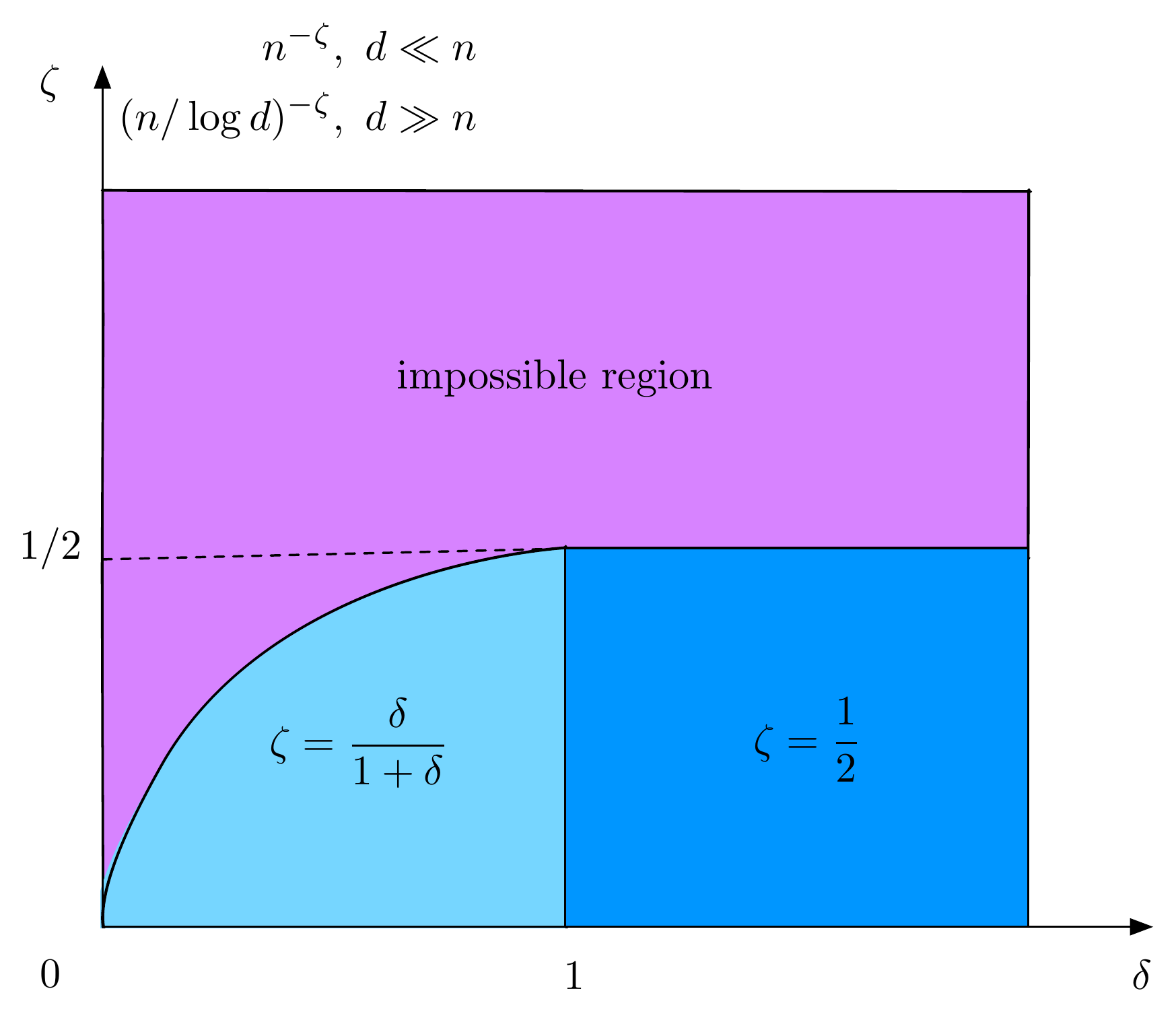}\centering\\
	\caption{\small Phase transition in terms of $\ell_2$-error for the adaptive Huber estimator. With fixed effective dimension,  $\|\widehat\bbeta_\tau-\bbeta^*\|_2\asymp n_{\textnormal{eff}}^{-\delta/(1+\delta)}$, when $0< \delta< 1$; $\|\widehat\bbeta_\tau-\bbeta^*\|_2\asymp n_{\textnormal{eff}}^{-1/2}$, when $\delta\geq 1$. Here $n_{\textnormal{eff}}$ is the effective sample size: $n_\textnormal{eff}=n$ in low dimensions while $n_\textnormal{eff}=n/\log d$ in high dimensions.  }
	\label{fig:1}
\end{figure}

In the high dimensional regime, we consider  the following regularized adaptive Huber regression with a different choice of the robustification parameter:
\begin{gather}
    \widehat\bbeta_{\tau ,\lambda} \in \arg\min_{\bbeta\in \RR^{d}} \big\{ \cL_\tau(\bbeta)\!+\!\lambda\|\bbeta\|_1  \big\}, \label{ssdr}
\end{gather}
where $ \tau \asymp \nu_\delta \{ n / (\log d) \}^{\max\{1/(1+\delta),1/2\}}$ and $\lambda\asymp \nu_\delta  \{ (\log d) /n\}^{\min\{\delta/(1+\delta),1/2\}}$ with $\nu_\delta= \min \{ v_\delta^{1/(1+\delta)} , v_1^{1/2}  \}.$
%With slight abuse of notation,
%When it is clear from the context, the dependence of $\widehat\bbeta_{\tau , \lambda}$ on $\lambda$ is sometimes suppressed.
Let $s$ be the size of the true support $\cS  = \supp(\bbeta^*)$. We will show that the regularized Huber estimator achieves an upper bound that is of the order $s^{1/2}  \{ (\log d )/ {n} \}^{\min\{\delta/(1+\delta), 1/2\}}$ for estimating $\bbeta^*$ in $\ell_2$-error with high probability.

To unify the nonasymptotic upper bounds  in the  two different regimes, we define the {\it effective dimension}, $d_{\textnormal{eff}}$, to be $d$ in low dimensions and $s$ in high dimensions. In other words, $d_{\textnormal{eff}}$ denotes the number of nonzero  parameters of the problem. The {\it effective sample size}, $n_{\textnormal{eff}}$, is defined as $n_\normal{eff}=n$ and   $n_\normal{eff}=n/\log d$ in low and high dimensions, respectively.  We will establish a phase transition: when $\delta\geq 1$, the proposed estimator enjoys a sub-Gaussian concentration, while it only achieves a slower concentration when $0<\delta<1$. Specifically, we show that, for any $\delta\in (0, \infty)$, the proposed estimators with $ \tau \asymp \min \{ v_\delta^{1/(1+\delta)}, v_1^{1/2}  \}\, n_\textnormal{eff}^{\max\{1/(1+\delta),1/2\}}$ achieve the following  tight upper bound,  up to logarithmic factors:
\#\label{eq:phase.uni}
\big\|\widehat\bbeta_\tau-\bbeta^*\big\|_2\lesssim  d_{\textnormal{eff}}^{{1}/{2}}  \, n_\textnormal{eff}^{-\min\{\delta/(1+\delta), 1/2\}} ~~\textnormal{with high probability. }
\#
This finding is  summarized in Figure \ref{fig:1}.

\section{Nonasymptotic Theory}\label{sec:3}

\subsection{Adaptive Huber Regression with Increasing Dimensions}
\label{sec:3.1}

We begin with the adaptive Huber regression in the low dimensional regime. First, we provide an upper bound for the estimation bias of Huber regression.
We then establish the phase transition by establishing matching upper and lower bounds on the $\ell_2$-error.
The analysis is carried out under  both fixed and random designs. The results  under random designs are provided in the supplementary material.
We start with the following regularity condition.
\begin{cond} \label{ass:3.0}
The empirical Gram matrix $\Sb_n :=  n^{-1} \sn \bx_i \bx_i^\T$ is nonsingular. Moreover, there exist constants $c_l$ and $c_u$ such that $c_l\leq \lambda_{\min}(\Sb_n)\leq \lambda_{\max}(\Sb_n)\leq c_u$.
\end{cond}

For any $\tau>0$, $\hat{\bbeta}_\tau$ given in \eqref{Huber.est} is natural $M$-estimator of
\#  \label{approxi.parameter}
	\bbeta^*_\tau :=  \arg\min_{\bbeta \in \RR^d}  \EE  \{ \cL_\tau(\bbeta) \} =  \arg\min_{\bbeta \in \RR^d}  \frac{1}{n}\sn  \EE \{ \ell_\tau( y_i  -  \langle \bx_i ,  \bbeta \rangle ) \},
\#
where the expectation is taken over the regression errors. We call $\bbeta^*_\tau$ the {\it Huber regression coefficient}, which is possibly different from the vector of true parameters $\bbeta^*$.
 The estimation bias, measured by $\| \bbeta^*_\tau - \bbeta^* \|_2$, is a direct consequence of robustification and asymmetric error distributions. Heuristically, choosing a sufficiently large $\tau$ reduces bias at the cost of losing robustness (the extreme case of $\tau = \infty$ corresponds to the least squares estimator). Our first result shows how the magnitude of $\tau$ affects the bias $\| \bbeta^*_\tau - \bbeta^* \|_2$. Recall that $v_\delta=n^{-1}\sum_{i=1}^n v_{i,\delta}$ with $v_{i, \delta} = \EE  ( |\varepsilon_i|^{1+\delta}  )$.

\begin{proposition} \label{prop:error}
Assume Condition~\ref{ass:3.0} holds and that $v_\delta $ is finite for some $\delta >0$. Then, the vector $\bbeta^*_\tau$ of Huber regression coefficients satisfies
\# \label{approxi.error}
 \| \bbeta^*_\tau -  \bbeta^*   \|_2 	 \leq  2 c_l^{-{1}/{2}}v_\delta\tau^{-\delta}
\#
provided $\tau  \geq   (4 v_\delta \wt M^2 )^{1/(1+\delta)}$ for $0 < \delta < 1$ or $\tau \geq  (2v_1)^{1/2} \wt M$ for  $\delta \geq 1$, where $\wt M = \max_{1\leq i\leq n} \| \Sb_n^{-1/2} \bx_i \|_2.$  %and $v_\delta = n^{-1} \sn v_{i,\delta}$.
\end{proposition}

The total estimation error $\|\widehat\bbeta_\tau-\bbeta^*\|_2$ can therefore be decomposed  into two parts
\$
\underbrace{\big\|\widehat\bbeta_\tau-\bbeta^*\big\|_2}_{\textnormal{total error}}\leq \underbrace{\big\|\widehat\bbeta_\tau-\bbeta_\tau^*\big\|_2}_{\textnormal{estimation error}}+\underbrace{\big\|\bbeta_\tau^*-\bbeta^*\big\|_2}_{\textnormal{approximation bias}},
\$
where the approximation bias is of order $\tau^{-\delta}$. A large $\tau$ reduces the bias but compromises the degree of robustness. Thus an optimal estimator is the one with  $\tau$ diverging at a certain rate to achieve the optimal tradeoff between  estimation error and approximation bias. Our next result presents nonasymptotic upper bounds on the $\ell_2$-error with an exponential-type exception probability, when $\tau$ is properly tuned. Recall that $\nu_\delta = \min\{ v_\delta^{1/(1+\delta)}, v_1^{1/2}\}$ for any  $\delta>0$.

\begin{theorem}[\sf Upper Bound]\label{thm:ld}
Assume Condition~\ref{ass:3.0} holds and $ v_{\delta} <\infty$ for some $\delta>0$. Let $L = \max_{1\leq i\leq n} \|  \bx_i \|_\infty$ and assume $n \geq C(L, c_l ) d^2 t$ for some $C(L, c_l)>0$ depending only on $L$ and $c_l$. Then, for any $t >0$ and $\tau_0 \geq  \nu_\delta$, the estimator $ \hat{\bbeta}_\tau$ with $\tau = \tau_0 ( n /t )^{\max\{1/(1+\delta),1/2\}}$ satisfies the bound
\# \label{L2.hatbeta}
	  \big\|  \hat\bbeta_\tau  -   \bbeta^* \big\|_2  \leq 4  c_l^{-1} L \tau_0  \, d^{1/2}  \bigg( \frac{t}{n} \bigg)^{\min\{\delta/(1+\delta),1/2\}}
\#
with probability at least $1- (2d+1)e^{-t}$.
\end{theorem}

\begin{remark}
It is worth mentioning that the proposed robust estimator depends on  the unknown parameter $v_\delta^{1/(1+\delta)}$.
Adaptation to the unknown moment is indeed another important problem. In Section~\ref{sec:4}, we suggest a simple cross-validation scheme for choosing $\tau$ with desirable numerical performance. A general adaptive construction of $\tau$ can be obtained via  Lepski's method \citep{Lepski1991}, which is more challenging due to  unspecified constants.
In the supplementary material, we discuss a variant of Lepski's method and establish its theoretical guarantee.
%We leave more discussions and the theoretical properties of Lepski's method to the supplementary material. %which relies on ``crude'' preliminary lower and upper bounds on $v_\delta^{1/(1+\delta)}$.
% We will focus on  cross validation in Section \ref{sec.implement}. For the Lepski's method, we refer to Section~\ref{sec.implement} for details.

%We leave the theoretical guarantee and numerical implementation of this procedure for future research.
\end{remark}

\begin{remark} We do not assume $\EE( | \varepsilon_i |^{1+\delta} | \bx_i)$ to be a constant, and hence the proposed method accommodates heteroscedastic regression models. For example, $\varepsilon_i$ can take the form of $\sigma(\bx_i)  v_i$, where $\sigma: \RR^d \to (0,\infty)$ is a positive function, and $v_i$ are random variables satisfying $\EE( v_i ) = 0$ and $\EE ( | v_i |^{1+\delta} ) < \infty$.
\end{remark}

\begin{remark}
We need the scaling condition to go roughly as $n\gtrsim d^2t$ under fixed designs. With random designs, we show that the scaling condition can be relaxed to $n\gtrsim d+t$. Details are given in the supplementary material.
\end{remark}

Theorem~\ref{ass:3.0} indicates that, with only bounded $(1+\delta)$-th moment,  the adaptive Huber estimator achieves the upper bound $d^{1/2}n^{-\min\{\delta/(1+\delta), 1/2\}}$, up to a logarithmic factor, by setting $t =\log(nd)$. A natural question is whether  the upper bound in \eqref{L2.hatbeta} is optimal. To address this,
we provide a matching lower bound up to a logarithmic factor. %by assuming identically and independently distributed error variables. With this setting, $v_\delta$ reduces to $\EE\big(|\varepsilon_i|^{1+\delta}\big)$.
Let $\cP_{\delta}^{v_\delta}$ be the class of all distributions on $\RR$ whose $(1+\delta)$-th absolute central moment equals $v_\delta.$ Let  $\Xb = (\bx_1, \ldots, \bx_n)^\T=(\bx^1,\ldots, \bx^d) \in \RR^{n\times d}$ be the design matrix and $\cU_n=\{\bu: \bu\in \{-1,1\}^n\}.$
%by holding $v_\delta$ fixed for each sample.
\begin{theorem}[{\sf Lower Bound}]\label{thm:ld:mini}
Assume that the regression errors $\varepsilon_i$ are i.i.d. from a distribution in $ \cP_{\delta}^{v_\delta}$ with $\delta>0$. Suppose there exists a $\bu\in \cU_n$ such that $\| n^{-1} \Xb^\T \bu \|_{\min}\geq \alpha$ for some $\alpha>0$. Then, for any $t\in [ 0, n/2]$ and  any estimator $\widehat\bbeta= \hat{\bbeta}(y_1,\ldots, y_n, t)$ possibly depending on $t$, we have
\$
\sup_{\PP \in\cP_{\delta}^{v_\delta} }\PP \Bigg[ \big\|\widehat\bbeta-\bbeta^*\big\|_2\geq     \alpha  c_u^{-1}   \nu_\delta \,d^{1/2}   \bigg( \frac{t}{n} \bigg)^{\min\{\delta/(1+\delta), 1/2\}}\Bigg] \geq  \frac{e^{-2t}}{2} ,
\$
where $c_u \geq   \lambda_{\max}(\Sb_n)$.
\end{theorem}

Theorem~\ref{thm:ld:mini} reveals that root-$n$ consistency with exponential concentration is impossible when $\delta \in (0,1)$.  It widens the phenomenon observed in Theorem~3.1 in \cite{devroye2015sub} for estimating a mean.  In addition to the eigenvalue assumption, we need to assume that there exists a $\bu\in \cU_n\subseteq \RR^n$ such that the minimum angle between $n^{-1}\bu$ and $\bx^j$ is non-vanishing.
This assumption comes from the intuition that the linear subspace spanned by $\bx^j$ is at most of rank $d$ and thus cannot span the whole space $\RR^n.$ This assumption naturally holds in the univariate case where $\Xb=(1, \ldots, 1)^\T$ and we can take $\bu=(1,\ldots, 1)^\T $ and $\alpha=1$.   More generally,   $\|\Xb^\T\bu/n\|_{\min}=\min \{ |\bu^\T \bx^1 |/n,\ldots,  |\bu^\T \bx^d |/n \}$.  Taking $|\bu^\T \bx^1|/n $ for an example, since $\bu\in \{-1,+1\}^n$, we can assume that each coordinate of $\bx^1$ is positive.  In this case, $\bu^\T\bx^1/n= \sn |x^1_i|/n\geq \min_i {|x_i^1|}$, which is strictly positive with probability one, assuming $\bx^1$ is drawn from a continuous distribution.  %

Together, the upper and lower bounds show that the adaptive Huber estimator achieves near-optimal  deviations. Moreover, it indicates that the Huber estimator with an adaptive $\tau$  exhibits a sharp phase transition: when $\delta\geq 1$, $\widehat\bbeta_\tau$ converges to $\bbeta^*$ at the parametric rate $n^{-1/2}$, while only a slower rate of order $n^{-\delta/(1+\delta)}$ is available when the second moment does not exist.

 \begin{remark}
We provide a parallel analysis under random designs in the supplementary material. Beyond the nonasymptotic  deviation bounds, we also prove a nonasymptotic Bahadur representation, which establishes a linear approximation of the nonlinear robust estimator. This result paves the way for future research on conducting statistical inference and constructing confidence sets under heavy-tailedness.  Additionally, the proposed estimator achieves full efficiency: it is as efficient as the ordinary least squares estimator asymptotically, while the robustness is characterized via nonasymptotic performance.
 \end{remark}

\subsection{Adaptive Huber Regression in High Dimensions}

In this section, we study the regularized adaptive Huber estimator in high dimensions where $d$ is allowed to grow with the sample size $n$ exponentially. The analysis is carried out under fixed designs, and results for random designs are again provided in the supplementary material. We start with  a modified version of the  localized restricted eigenvalue introduced by \cite{fan2015tac}.  Let $\Hb_\tau(\bbeta) = \nabla^2\cL_\tau(\bbeta)$ denote the Hessian matrix. Recall that $\cS = {\rm supp}(\bbeta^*)\subseteq \{ 1, \ldots, d\}$ is the true support set  with $|\cS|=s$.

\begin{definition}[\sf Localized Restricted Eigenvalue, LRE]\label{lre}
	The localized restricted eigenvalue of $\Hb_\tau$ is defined as
\$\kappa_+(m,\gamma,r)=\sup \Big\{{ \langle \bu,  \Hb_\tau(\bbeta)   \bu \rangle }:
     ( \bu ,\bbeta ) \in \mathcal{C}(m,\gamma,r) \Big\},\\
\kappa_-(m,\gamma,r)=\inf \Big\{{ \langle \bu,   \Hb_\tau(\bbeta) \bu \rangle }:
     ( \bu ,\bbeta )\in \mathcal{C}(m,\gamma,r)   \Big\},
\$
where $\mathcal{C}(m,\gamma,r) := \{  ( \bu ,\bbeta  )\in  \mathbb{S}^{d-1} \times \RR^d : \forall J \subseteq \{ 1, \ldots, d\}~{\textnormal{satisfying}}~
S\subseteq J, |J|\leq m, \| \bu_{J^c}\|_1\leq \gamma \| \bu_{J} \|_1, \|\bbeta-\bbeta^*\|_1\leq r \}$ is a local  $\ell_1$-cone.
\end{definition}

The LRE is defined in a local neighborhood of  $\bbeta^*$ under  $\ell_1$-norm. This facilitates our proof, while \cite{fan2015tac} use the $\ell_2$-norm.
\begin{cond}\label{con:lre}
 $\Hb_\tau$ satisfies the localized restricted eigenvalue condition $\textnormal{LRE}(k, \gamma, r)$, that is, $ \kappa_l \leq\kappa_-(k,\gamma, r)\leq \kappa_+(k,\gamma, r)\leq \kappa_u $ for some constants $\kappa_u, \kappa_l >0$.
\end{cond}

The condition above is  referred to as the {LRE condition} \citep{fan2015tac}. It is a unified condition for studying generalized loss functions, whose Hessians may possibly depend on $\bbeta$. For Huber loss,  Condition \ref{con:lre} also involves the observation noise. The following  definition concerns the restricted eigenvalues of $\Sb_n$ instead of $\Hb_\tau$. %However, the Hessian matrix $H_\tau(\bbeta)$ is not only depending on $\bbeta$,  but also a random matrix.

\begin{definition}[\sf Restricted Eigenvalue, RE]\label{re}
The restricted maximum and minimum eigenvalues of $\Sb_n$ are defined respectively as
\$
\rho_+(m,\gamma)&=\sup_{ \bu }\big\{{ \langle \bu,  \Sb_n  \bu \rangle }:
      \bu  \in \mathcal{C}(m,\gamma) \big\}, \\
  \rho_-(m,\gamma)&=\inf_{ \bu } \big\{{  \langle \bu,  \Sb_n  \bu \rangle }:
     \bu \in \mathcal{C}(m,\gamma)   \big\},
\$
where $\mathcal{C}(m,\gamma) := \{  \bu \in  \mathbb{S}^{d-1} : \forall J \subseteq \{ 1, \ldots, d\}~{\rm satisfying}~S\subseteq J, |J|\leq m, \| \bu_{J^c}\|_1\leq \gamma \| \bu_{J}\|_1 \}$.
\end{definition}

\begin{cond}\label{con:re}
$\Sb_n$ satisfies the restricted eigenvalue condition $\textnormal{RE}(k, \gamma)$, that is, $  \kappa_l \leq\rho_-(k,\gamma)\leq \rho_+(k,\gamma)\leq \kappa_u $ for some constants $\kappa_u, \kappa_l >0$.
\end{cond}

To make Condition \ref{con:lre} on $\Hb_\tau$ practically useful, in what follows,  we show that Condition \ref{con:re} implies Condition \ref{con:lre} with high probability. As before, we write $v_\delta=n^{-1}\sum_{i=1}^nv_{i,\delta}$ and $L=\max_{1\leq i\leq n} \| \bx_i \|_{\infty}$.
\begin{lemma}\label{lemma:hd:lm1}
Condition \ref{con:re} implies Condition \ref{con:lre} with high probability: if $0< \kappa_l \leq \rho_-(k,\gamma)\leq \rho_+(k, \gamma)\leq \kappa_u<\infty $ for some $k\geq 1$ and $\gamma>0$, then it holds with probability at least $1-e^{-t}$ that, $0<\kappa_l /2 \leq \kappa_-(k,\gamma , r)\leq \kappa_+(k, \gamma, r)\leq \kappa_u<\infty$ provided  $\tau \geq \max\{  8L r , c_1 (L^2 k v_\delta )^{1/(1+\delta)}\}$ and $n \geq c_2 L^4  k^2 t$, where $c_1, c_2>0$ are constants depending only on $(\gamma, \kappa_l )$.

\end{lemma}

With the above preparations in place, we are now ready to present the main results on the adaptive Huber estimator in high dimensions.

\begin{theorem}[\sf Upper Bound in High Dimensions]\label{thm:hd}
Assume Condition \ref{con:re} holds with $(k,\gamma)=(2s, 3)$, $v_\delta <\infty$ for some $0<\delta \leq 1$.   For any $t>0$ and $\tau_0\geq  \nu_{\delta}$, let $\tau =\tau_0 (n/t)^{\max\{ 1/(1+\delta) , 1/2 \}}$, $\lambda \geq 4L \tau_0  (t/n)^{ \min\{ \delta/(1+\delta) , 1/2 \}  }$, and $r>12\kappa_l^{-1}  s \lambda $.  Then with probability at least $1- (2s + 1) e^{-t}$,  the $\ell_1$-regularized Huber estimator $\hat{\bbeta}_{\tau, \lambda}$ defined in \eqref{ssdr} satisfies
\#
 \big\|\hat\bbeta_{\tau, \lambda} -\bbeta^*\big\|_2 \leq  3 \kappa_l^{-1}   s^{1/2} \lambda ,   \label{hd.bound.general}
\#
as long as $n \geq C(L,\kappa_l) s^2 t$ for some $C(L,\kappa_l)$ depending only on $(L,\kappa_l)$. In particular, with $t=(1+c) \log d$ for $c>0$ we have
\#
    \big\|\hat\bbeta_{\tau, \lambda} -\bbeta^*\big\|_2\lesssim   \kappa_l^{-1}L \tau_0 \, s^{1/2} \bigg\{ \frac{ (1+c) \log d}{n}\bigg\}^{\min\{\delta/(1+\delta),1/2\}}  \label{hd.bound.special}
\#
with probability at least $1- d^{-c}$.
\end{theorem}

The above result demonstrates that the regularized Huber estimator with an adaptive robustification parameter converges at the rate $s^{1/2}  \{ ( \log d ) / n \}^{\min\{\delta/(1+\delta),1/2\}}$ with overwhelming probability. Provided the observation noise has finite variance,  the proposed estimator performs as well as the Lasso with sub-Gaussian errors. We advocate the adaptive Huber regression method since sub-Gaussian condition often fails in practice \citep{Lan2015, eklund2016cluster}.

\begin{remark}
As pointed out by a reviewer, if one pursues a sparsity-adaptive approach, such as the SLOPE \citep{bogdan2015slope,bellec2016slope}, the upper bound on $\ell_2$-error can be improved from $\sqrt{s\log(d) /n}$ to $\sqrt{s \log(ed/s)/n}$.
With heavy-tailed observation noise, it is interesting to investigate whether this sharper bound can be achieved by Huber-type regularized estimator.
We leave this to future work as a significant amount of additional work is still needed.
On the other hand, since $\log(ed/s) = 1+ \log d - \log s$ and $s \leq n$, %under a standard constraint on $s$ that $s =O( d^{a} )$ for some $a < 1$,
$\log(ed/s)$ scales the same as $\log d$ so long as $\log d > a \log n$ for some $a > 1$.
\end{remark}

\begin{remark}
Analogously to the low dimensional case, here we impose the sample size scaling $n\gtrsim s^2\log d$ under fixed designs.
In the supplementary material, we obtain minimax optimal $\ell_1$-, $\ell_2$- and prediction error bounds for $\hat{\bbeta}_{\tau, \lambda}$ with random designs under the scaling $n\gtrsim s \log d$.
\end{remark}

Finally, we establish a matching  lower bound for estimating $\bbeta^*$. %in the $s$-sparse set $\BB_0(s) = \{\bbeta \in \RR^d: \|\bbeta\|_0=s \}$.
Recall the definition of $\cU_n$ in Theorem~\ref{thm:ld:mini}.

\begin{theorem}[\sf Lower Bound in High Dimensions]\label{thm:hd:mini}
Assume that $\varepsilon_i$ are independent from some distribution in $\cP_\delta^{v_\delta}$. Suppose that  Condition \ref{con:re} holds with $k=2s$ and  $\gamma=0$. Further assume that there exists a set $\cA$ with $|\cA|=s$ and $\ub\in \cU_n$ such that $\|\Xb_\cA^\T \ub/n\|_{\min}\geq \alpha$ for some $\alpha>0$. Then, for any $A>0$ and $s$-sparse estimator $\widehat\bbeta = \hat{\bbeta}(y_1, \ldots, y_n, A)$ possibly depending on $A$, we have
\$
\sup_{\PP \in \cP_{\delta}^{v_\delta}}\PP\Bigg[ \big\|\widehat\bbeta-\bbeta^*\big\|_2\geq     \nu_{\delta}  \frac{\alpha s^{1/2} }{\kappa_u}  \bigg(\frac{A\log d}{2n}\bigg)^{\min\{{\delta}/({1+\delta}), 1/2\}}\Bigg] \geq 2^{-1}d^{-A},
\$
as long as $n\geq 2(A\log d+\log 2)$.
\end{theorem}

Together,  Theorems~\ref{thm:hd} and \ref{thm:hd:mini} show that the regularized adaptive Huber estimator achieves the optimal rate of convergence  in $\ell_2$-error. The proof, which is given in the supplementary material, involves constructing a sub-class of binomial distributions for the regression errors.
Unifying the results in low and high dimensions, we arrive at the claim \eqref{eq:phase.uni} and thus the phase transition in Figure \ref{fig:1}.

\section{Extension to Heavy-tailed Designs} \label{sec:ext}

In this section, we extend the idea of adaptive Huber regression described in Section~\ref{sec:2} to the case where both the covariate vector  $\bx$ and the regression error $\varepsilon$ exhibit heavy tails.
%This also gives us a chance to present the results for random designs.
We focus on the high dimensional regime $d\gg n$, where $\bbeta^* \in \RR^d$ is sparse with $s = \| \bbeta^* \|_0 \ll n$.
%The low-dimensional setting can be dealt with in a similar way.
Observe that, for Huber regression, the linear part of the Huber loss penalizes the residuals, and therefore robustifies the quadratic loss in the sense that outliers in the response space (caused by heavy-tailed observation noise) are down weighted or removed. Since no robustification is imposed on the covariates, intuitively, the adaptive Huber estimator may not be robust against heavy-tailed covariates. In what follows, we modify the adaptive Huber regression  to robustify both the covariates and regression errors.

To begin with, suppose we observe  independent data $\{( y_i, \bx_i) \}_{i=1}^n$ from $(  y, \bx)$, which follows the linear model $y=\langle \bx, \bbeta^*\rangle + \varepsilon$. To robustify $\bx_i$, we define truncated covariates $ \bx^{\varpi}_i = ( \psi_{\varpi}(x_{i1}), \ldots,   \psi_{\varpi}(x_{id}) )^\T$, where $\psi_\varpi(x) :=   \min\{  \max (-\varpi, x), \varpi \}  $ and $\varpi>0$ is a tuning parameter. Then we consider the  modified adaptive Huber estimator (see \cite{fan2016robust} for a general robustification principle)
\#
	 \hat{\bbeta}_{\tau, \varpi, \lambda}  \in  \arg\min_{\bbeta\in \RR^d} \big\{   \cL^{\varpi}_\tau(\bbeta ) +  \lambda \| \bbeta \|_1  \big\} ,  \label{pure.robust.est}
\#
where $ \cL^{\varpi}_\tau(\bbeta ) = n^{-1} \sum_{i=1}^n\ell_\tau (y_i- \langle  \bx^{\varpi}_i , \bbeta  \rangle )$   and $\lambda>0$ is a regularization parameter.

Let $\cS$ be the true support of $\bbeta^*$ with sparsity $|\cS|=s$, and denote by $  \Hb^{\varpi}_\tau(\bbeta)=\nabla^2  \cL^{\varpi}_\tau(\bbeta)$ the Hessian matrix of the modified Huber loss. To investigate the deviation property of $ \hat{\bbeta}_{\tau, \varpi, \lambda} $, we impose the following mild moment assumptions.

\begin{cond} \label{moment.condition.hd}
(i) $\EE(\varepsilon) = 0$, $\sigma^2 = \EE(\varepsilon^2)>0$ and $v_3 := \EE(\varepsilon^4) <\infty$; (ii) The covariate vector $\bx= (x_{1},\ldots, x_{d})^\T \in \RR^d$ is independent of $\varepsilon$ and satisfies $M_4 := \max_{1\leq j\leq d} \EE(x_j^4) <\infty $.
\end{cond}

We are now in place to state the main result of this section. Theorem~\ref{robust.hdreg.dev.ineq} below demonstrates that the modified adaptive Huber estimator admits exponentially fast concentration when the convariates only have finite fourth moments, although at the cost of stronger scaling conditions.

\begin{theorem}  \label{robust.hdreg.dev.ineq}
Assume Condition~\ref{moment.condition.hd} holds and let $  \Hb^{\varpi}_\tau(\cdot)$ satisfy Condition~\ref{con:lre} with $k=2s$, $\gamma=3$ and $r>  12 \kappa_l^{-1}   \lambda s$. Then, the modified adaptive Huber estimator $ \hat{\bbeta}_{\tau, \varpi, \lambda}$ given in \eqref{pure.robust.est} satisfies, on the event $ \mathcal{E}(\tau, \varpi, \lambda) = \big\{ \| ( \nabla   \cL^{\varpi}_\tau(\bbeta^*) )_{\cS} \|_\infty\leq \lambda/2 \big\}$, that
\#
	\big\| \hat{\bbeta}_{\tau, \varpi, \lambda}  - \bbeta^* \big\|_2 \leq 3 \kappa_l^{-1} s^{1/2} \lambda .  \nn
\#
For any $t>0$, let the triplet $( \tau, \varpi, \lambda)$ satisfy
\#
	\lambda & \geq  2 M_4 \| \bbeta^* \|_2 \, s^{1/2}   \varpi^{-2} + 8  \big\{ v_2  M_2^{1/2}  +  M_4 \| \bbeta^* \|_2^3 \,  s^{3/2} \big\}  \tau^{-2}  \nn \\
 & \quad + 2\big(2\sigma^2 M_2  +  2M_4 \|\bbeta^* \|_2^2 \, s   \big)^{1/2} \sqrt{\frac{t}{n}}  +  \varpi \tau\frac{ t }{n},\label{scaling.tau}
\#
where $v_2 = \EE(|\varepsilon|^3)$ and $M_2 = \max_{1\leq j\leq d} \EE(x_j^2)$. Then $\PP\{ \mathcal{E}(\tau, \varpi, \lambda ) \} \geq 1- 2s e^{-t}$.
\end{theorem}

\begin{remark}
Assume that the quantities $v_3$, $M_4$ and $\| \bbeta^*\|_2$ are all bounded. Taking $t\asymp \log d$ in  \eqref{scaling.tau}, we see that $\hat{\bbeta}_{\tau, \varpi, \lambda}$ achieves a near-optimal convergence rate of order $s\sqrt{ (\log d) /n }$ when the parameters $(\tau, \varpi, \lambda)$ scale as
$$
 \tau \asymp  s^{1/2} \bigg( \frac{n}{\log d} \bigg)^{1/4} , \ \ \varpi \asymp  \bigg( \frac{n}{\log d} \bigg)^{1/4} ~\mbox{ and }~ \lambda \asymp \sqrt{\frac{ s \log d}{n}}.
$$
We remark here that  the theoretically optimal $\tau$ is different from that in the sub-Gaussian design case. See Theorem~B.2 in the supplementary material.
\end{remark}

\section{Algorithm and Implementation}\label{sec:3plus}

This section is devoted to computational algorithm and numerical implementation. We focus on the regularized adaptive Huber regression  in \eqref{ssdr}, as \eqref{Huber.est} can be easily solved via the iteratively reweighted least squares method. To solve the convex optimization problem in \eqref{ssdr},  standard optimization algorithms, such as the cutting-plane or interior point method, are not scalable to large-scale problems.

In what follows, we describe a fast and easily implementable method using the local adaptive majorize-minimization  (LAMM) principle \citep{fan2015tac}.  We say that  a function $g(\btt|\btt^{(k)})$ majorizes $f(\bbeta)$ at the point $\bbeta^{(k)}$ if
\$
g(\btt|\btt^{(k)}) \geq f(\btt) \quad \mbox{and} \quad
g(\btt^{(k)}|\btt^{(k)}) = f(\btt^{(k)}).
\$
To minimize a general function $f(\btt)$, a majorize-minimization (MM) algorithm  initializes  at $\bbeta^{(0)}$, and then iteratively computes $\bbeta^{(k+1)}=\arg\min_{\bbeta\in \RR^d} g(\bbeta|\bbeta^{(k)})$ for $k=0,1,\ldots$. The objective value of such an algorithm decreases in each step, since
\#\label{0311.3}
%f(\btt^{(k+1)})\leq g(\btt^{(k+1)}|\btt^{(k)})\leq g(\btt^{(k)}|\btt^{(k)}) =  f(\btt^{(k)}).
   f(\bbeta^{(k+1)}) \stackrel{\mbox{\tiny major.}} {\leq}  g(\bbeta^{(k+1)} \,|\, \bbeta^{(k)}) \stackrel{\mbox{\tiny min.}}{\leq}
    g(\bbeta^{(k)}\,|\,\bbeta^{(k)}) \stackrel{\mbox{\tiny init.}} {=} f(\bbeta^{(k)}).
\#
%An inspection of the above arguments shows that the majorization requirement is not necessary.  It requires only the local property:
%\begin{equation} \label{0311.b}
%   f(\bbeta^{(k+1)})\leq g(\btt^{(k+1)}|\btt^{(k)})~\text{and}~ g(\btt^{(k)}|\btt^{(k)})=f(\bbeta^{(k)})
%\end{equation}
%for the inequalities in \eqref{0311.3} to hold.
%Inspired by the above observation,
As pointed out by \cite{fan2015tac}, the majorization requirement only needs to hold locally at $\bbeta^{(k+1)}$ when starting from $\bbeta^{(k)}$.  We therefore  locally majorize $\cL_\tau(\bbeta)$ in \eqref{ssdr} at $\bbeta^{(k)}$   by an isotropic quadratic function
$$
 g_k(\bbeta|\bbeta^{(k)})= \cL_\tau(\btt^{(k)})+\big\langle \nabla\cL_\tau(\btt^{(k)}),\, \btt-\btt^{(k)}\big\rangle +\frac{\phi_k}{2}\big\|\btt-\btt^{(k)}\big\|_2^2,
$$
where $\phi_k$ is a quadratic parameter such that  $g_k(\btt^{(k+1)}|\btt^{(k)}) \geq \cL_\tau(\btt^{(k+1)})$.
The isotropic form also allows a simple analytic solution to the subsequent majorized optimization problem:
\#\label{pg_form}
 \min_{\btt\in \RR^d} \biggl\{ \big\langle \nabla\cL_\tau(\btt^{(k)}),\btt-\btt^{(k)}\big\rangle +\frac{\phi_k}{2}\big\|\btt-\btt^{(k)}\big\|_2^2+\lambda\big\|\bbeta\big\|_1   \biggr\}.
\#
It can be shown that \eqref{pg_form} is minimized at
$$
 \btt^{(k+1)}=T_{\lambda,\phi_k}(\btt^{(k)})= S\Big(\btt^{(k)}-{\phi_k^{-1}}{\nabla\cL_\tau(\btt^{(k)})}, {\phi_k^{-1}}\lambda\Big),
$$
where $S(\xb,\lambda)$ is the  soft-thresholding operator defined by
$S(\xb,\lambda)=  \text{sign}(x_j) \max (|x_j|-\lambda , 0  )$. The simplicity of this updating rule is due to the fact that  \eqref{pg_form} is an unconstrained optimization problem.

To find the smallest $\phi_k$ such that $g_k(\bbeta^{(k+1)}|\bbeta^{(k)})\geq \cL_\tau(\bbeta^{(k+1)})$,   the basic idea of LAMM is to start from a relatively small isotropic parameter $\phi_k=\phi_k^0$ and then successfully inflate $\phi_k$ by a factor $\gamma_u > 1$, say $\gamma_u=2$.  If the solution satisfies $g_k(\bbeta^{(k+1)}|\bbeta^{(k)})\geq \cL_\tau(\bbeta^{(k+1)})$, we stop and obtain $\bbeta^{(k+1)}$, which makes the target value non-increasing.  We then continue with the iteration to produce next solution until the solution sequence $\{\bbeta^{(k)} \}_{k=1}^\infty$ converges. A simple stopping criterion is $\|\bbeta^{(k+1)}-\bbeta^{(k)}\|_2\leq \epsilon$ for a sufficiently small $\epsilon$, say $10^{-4}$. We refer to \cite{fan2015tac} for a detailed complexity analysis of the LAMM algorithm.

\begin{algorithm}[!t]\label{alg:1}
	\caption{LAMM algorithm for regularized adaptive  Huber regression. }\label{alg:ls}
	\begin{algorithmic}[1]
		\STATE{{\bf Algorithm}: $\{\btt^{(k)},\phi_k\}_{k=1}^\infty \leftarrow \mbox{LAMM}(\lambda,\btt^{(0)}, \phi_0, \epsilon$ ) }
		\STATE{\textbf{Input}: $\lambda, \btt^{(0)}, \phi_0, \epsilon$   }
		\STATE{\textbf{Initialize}: $\phi^{(\ell,k)}\leftarrow \max\{\phi_{0},\gamma_u^{-1}\phi^{(\ell,k-1)}\}$  }
		\STATE{\textbf{for} $k=0,1,\ldots$ until $\|\bbeta^{(k+1)}-\bbeta^{(k)}\|_2\leq \epsilon$  \textbf{do}}
		\STATE{~~~~ \textbf{Repeat}}
		\STATE{~~~~~~~~ $\btt^{(k+1)}\leftarrow T_{\lambda,\phi_k}(\btt^{(k)})$ }
		\STATE{~~~~~~~~ \textbf{If} $g_k(\btt^{(k+1)}|\bbeta^{(k)})< \cL_\tau(\bbeta^{(k+1)})$ \textbf{ then } $\phi_k\leftarrow \gamma_u\phi_k$ }
		\STATE{~~~~ \textbf{Until}  $g_k(\btt^{(k+1)}|\bbeta^{(k)})\geq \cL_\tau(\bbeta^{(k+1)})$ }
		\STATE{~~~~ \textbf{Return} $\{\btt^{(k+1)},\phi_k\}$}
		\STATE{\textbf{end for}}
		\STATE{\textbf{Output}: $\widehat\bbeta=\bbeta^{(k+1)}$}
	\end{algorithmic}
\end{algorithm}

\section{Numerical Studies}\label{sec:4}

%\subsection{Simulation Examples}

\subsection{ Tuning Parameter and Finite Sample Performance}
\label{sec.implement}

For numerical studies and real data analysis, in the case where the actual order of moments is unspecified, we  presume the variance is finite and therefore choose robustification and regularization parameters as follows:
\$
\tau=c_\tau \times \widehat \sigma\, \left(\frac{n_\eff}{t}\right)^{1/2}~~\textnormal{and}~~\lambda=c_\lambda\times \widehat \sigma \,\left(\frac{n_{\eff}}{t}\right)^{ 1/2},
\$
where $\widehat\sigma^2 =n^{-1}\sum_{i=1}^n(y_i-\bar y)^2$ with $\bar y = n^{-1} \sn y_i$ serves as a crude preliminary estimate of $\sigma^2$, and the parameter $t$ controls the confidence level. We set $t = \log n$ for simplicity except for the phase transition plot.  The constant $c_\tau$ and  $c_\lambda$ are chosen via  3-fold cross-validation from a small set of constants, say $\{0.5, 1, 1.5\}$.

%%%%%%%%%%%%%%%%%%%%%%%%%%%%
% Table 1
%%%%%%%%%%%%%%%%%%%%%%%%%%%%
\begin{table}[!t]
\begin{center}
\caption{Results for adaptive Huber regression (AHR) and ordinary least squares (OLS) when $n=100$ and $d=5$. The mean and standard  deviation (std) of $\ell_2$-error based on 100 simulations are reported. }
\vspace{0.2cm}
\begin{tabular}{l   ccc c}
  \hline
 Noise & \multicolumn{2}{c}{AHR}& \multicolumn{2}{c}{OLS} \\
&mean & std & mean & std  \\ \hline
Normal   & 0.566 &0.189& 0.567& 0.191  \\
Student's $t$  & 0.806 &0.651& 1.355 &2.306\\
Log-normal &  3.917 &3.740& 8.529 &13.679\\
\hline
\end{tabular}
\label{table1}
\end{center}
\end{table}

We generate data from the linear model
\#\label{ex:1:eq:1}
y_i= \langle  \bx_i , \bbeta^* \rangle +\varepsilon_i, ~~~i\!=\!1,\ldots, n,
\#	
where $\varepsilon_i$ are i.i.d. regression errors and
$
\bbeta^* = (5,-2,0,0, 3,\underbrace{0,\ldots, 0}_{d-5})^\T  \in \RR^d.
$
Independent of  $\varepsilon_i$, we generate $\bx_i$ from standard multivariate normal distribution $\cN({\bf 0}, \Ib_d )$. %$\sim t _\textnormal{df}$, a student's  $t$-distribution with $\textnormal{df}$ being the degree of freedom, and $\bbeta^*$ is a $d$-dimensional vector of coefficients such that
In this section, we set $(n,d)=(100,5)$, and generate regression errors from three different distributions: the normal distribution $\cN(0,4)$, the $t$-distribution with degrees of freedom 1.5, and the log-normal distribution $\log \cN(0, 4)$.
Both $t$ and log-normal distributions are heavy-tailed, and produce outliers with high chance. %correspond to outliers in the model.

The results on $\ell_2$-error for adaptive Huber regression and the least squares estimator, averaged over 100 simulations, are summarized in Table~\ref{table1}.
In the case of normally distributed noise, the adaptive Huber estimator performs as well as the least squares.
With heavy-tailed regression errors following Student's $t$ or log-normal distribution,  the adaptive Huber regression significantly outperforms the least squares.
These empirical results reveal that adaptive Huber regression prevails across various scenarios: not only it provides more reliable estimators in the presence of heavy-tailed and/or asymmetric errors, but also loses almost no efficiency at the normal model.

\subsection{Phase Transition}
In this section,  we validate the phase transition behavior of $\|\widehat\bbeta_\tau-\bbeta^*\|_2$ empirically.  We generate continuous responses according to  \eqref{ex:1:eq:1},	
where $\bbeta^*$ and $\bx_i$ are set the same way as before. We sample independent errors as $\varepsilon_i\sim t _\textnormal{df}$, Student's  $t$-distribution with $\textnormal{df}$ degrees of freedom.  Note that $t_{\textnormal{df}}$ has finite $(1+\delta)$-th moments provided $\delta<\textnormal{df}-1$ and infinite $\textnormal{df}$-th moment. Therefore, we take $\delta=\textnormal{df}-1-0.05$ throughout.

%Our second example is designed  to demonstrate the phase transition phenomenon of the Huber estimator in both low  and high dimensions. %see Theorems \ref{thm:ld}, \ref{thm:ld:mini},  \ref{thm:hd} and \ref{thm:hd:mini}.
In low dimensions, we take $(n,d)=(500,5)$ and a sequence of degrees of freedoms (df's):  $\textnormal{df}\!\in\!\{1.1, 1.2, \ldots, 3.0\}$; in high dimensions, we take $(n, d)=(500,1000)$, with the same choice of df's. Tuning parameters $(\tau, \lambda)$ are calibrated similarly as before. Indicated by the main theorems, %Theorems \ref{thm:ld}, \ref{thm:ld:mini}, \ref{thm:hd} and \ref{thm:hd:mini},
it holds
\begin{enumerate}
\item{}(Low dimension):
$$
	-\log\big(\|\widehat\bbeta_\tau -\bbeta^*\|_2 \big) \asymp \frac{\delta}{1+\delta}\log (n)-\frac{1}{1+\delta}\log (v_\delta), \ \  0< \delta \leq 1,
$$
\item{}(High dimension):
$$
-\log\big( \|\widehat\bbeta_{\tau} -\bbeta^*\|_2 \big) \asymp \frac{\delta}{1+\delta}\log\Big(\frac{n}{\log d}\Big)-\frac{1}{1+\delta}\log (v_\delta),  \ \   0< \delta \leq 1 ,
$$
\end{enumerate}
which are approximately $\log(n)\times\delta/(1+\delta)$ and $\log(n/\log d)\times \delta/(1+\delta)$, respectively, when  $n$ is sufficiently large.

\begin{figure}[t!]
	\includegraphics[scale=.6]{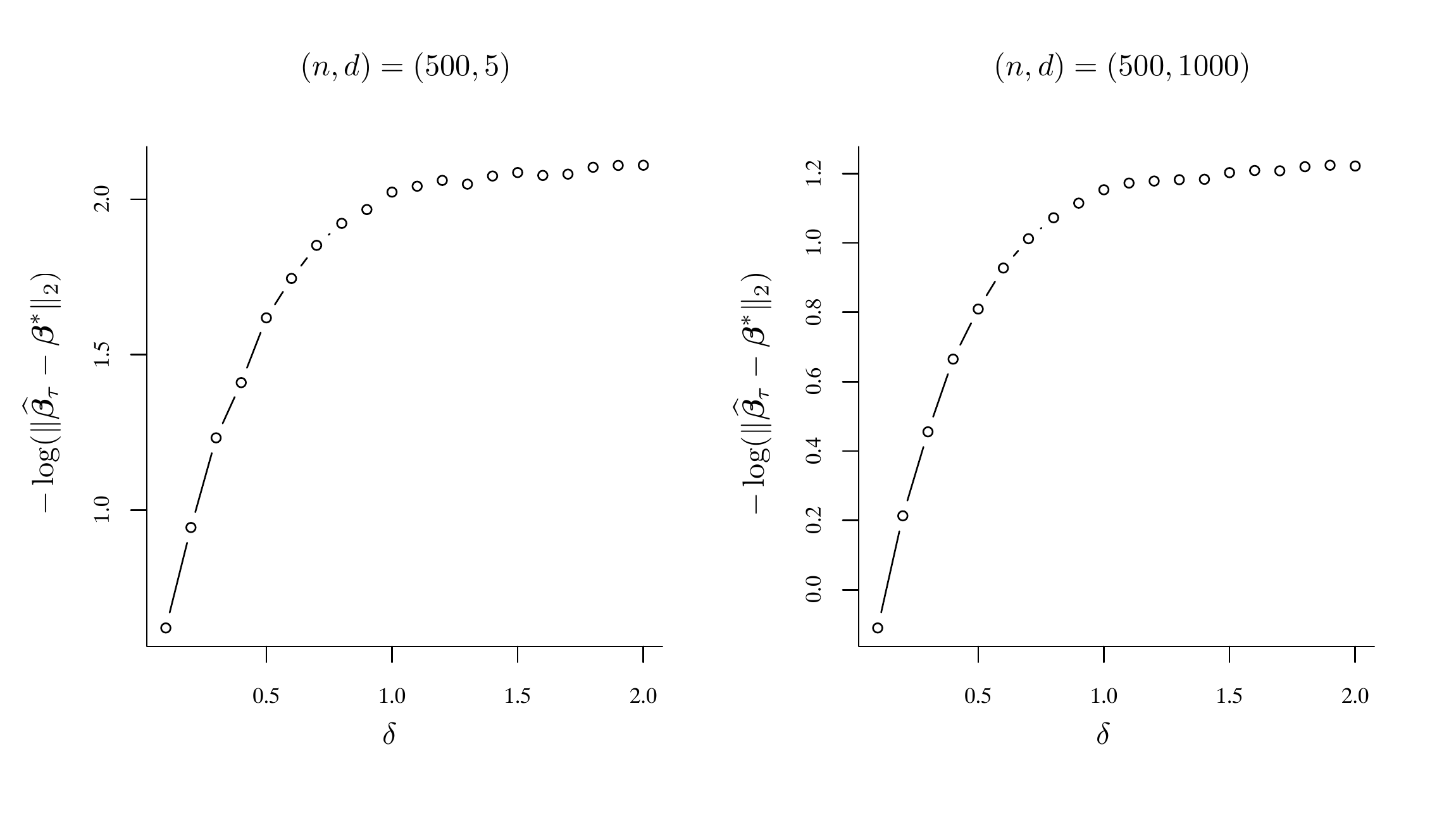}\centering\\
	\caption{Negative $\log$ $\ell_2$-error  versus $\delta$ in low (left panel) and high (right panel) dimensions.}  \label{fig:3}
\end{figure}

\begin{figure}[t!]
	\includegraphics[scale=.6]{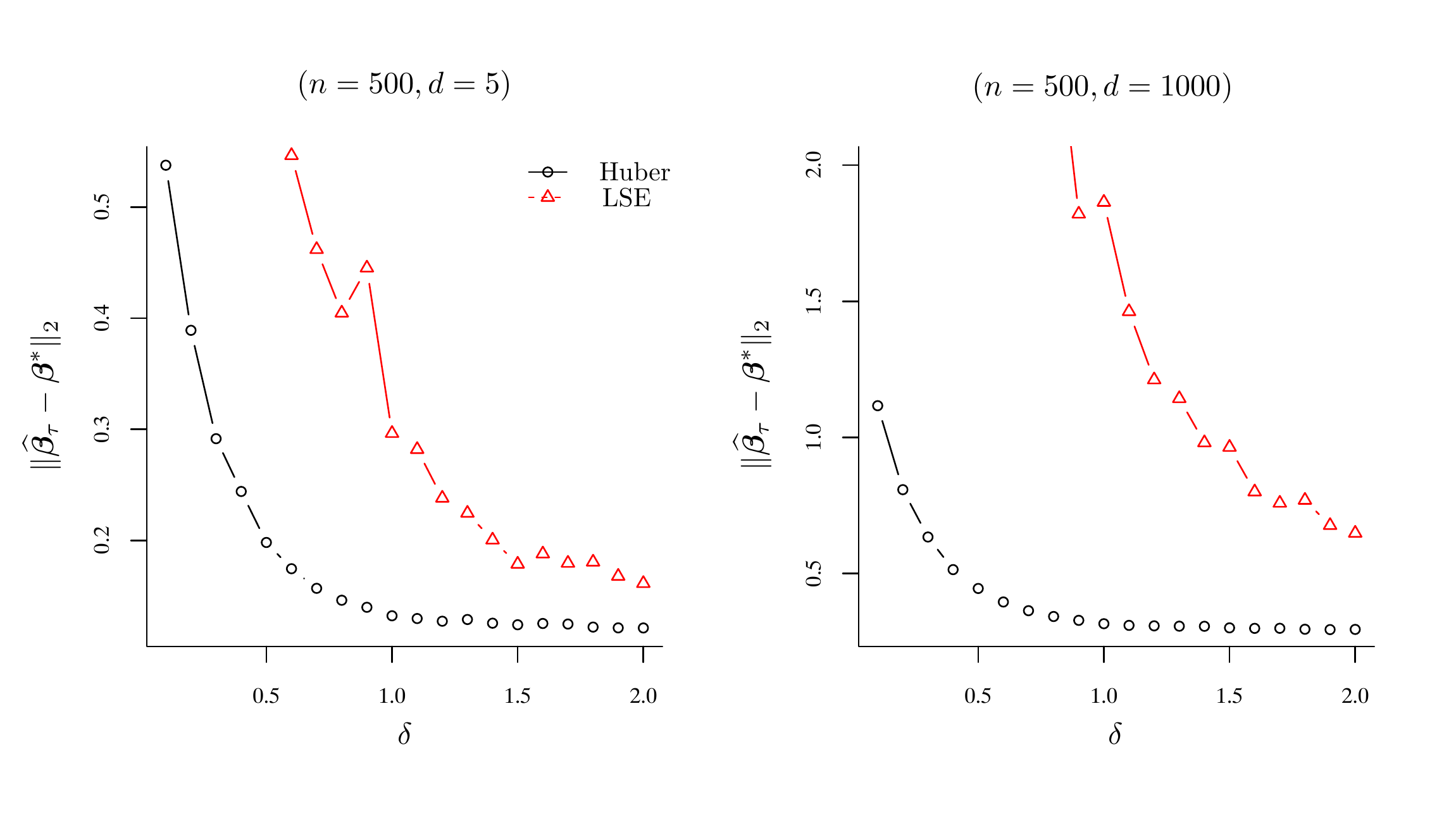}\centering\\
	\caption{Comparison between the  (regularized) adaptive Huber estimator and the (regularized) least squares estimator under $\ell_2$-error.}  \label{fig:4}
\end{figure}

Figure \ref{fig:3} displays the negative $\log$ $\ell_2$-error versus $\delta$ in both low and high dimensions over 200 repetitions for each $(n,d)$ combination. The empirically fitted curve closely resembles the theoretical curve displayed in Figure \ref{fig:1}. These numerical results are in line with the theoretical findings, and empirically validate the phase transition of the adaptive Huber estimator.

We also compared the $\ell_2$-error of the adaptive Huber estimator  with that of the OLS estimator for $t$-distributed errors with varying degrees of freedoms. As shown in Figure \ref{fig:4}, adaptive Huber exhibits a significant advantage especially when $\delta$ is small. The OLS slowly catches up as $\delta$ increases.

\subsection{Effective Sample Size}

\begin{figure}[t!]
	\includegraphics[scale=.6]{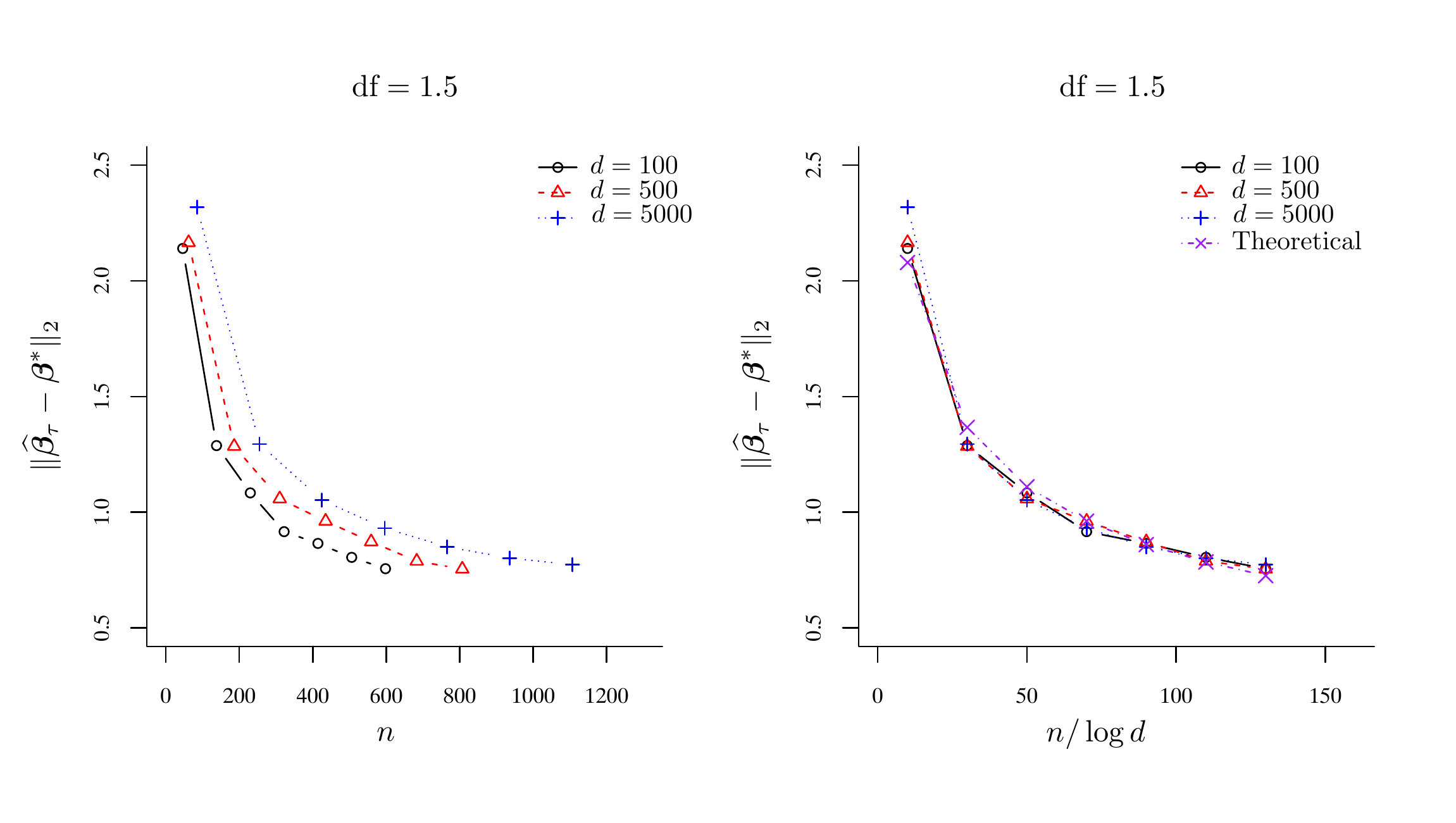}\centering\\
	\caption{The $\ell_2$-error versus sample size $n$ (left panel) and the $\ell_2$-error versus {effective sample size} $n_\textnormal{eff}=n/\log d$ (right panel).
%The scaling behavior of the $\ell_2$-error $\|\widehat\bbeta_\tau-\bbeta\|_2$ in high dimensions. Left panels shows the $\ell_2$-error  versus sample size, while the right panel shows the  $\ell_2$-error versus the {effective sample size} $n_\textnormal{eff}=n/\log d$.
}  \label{fig:2}
\end{figure}

In this section,  we verify the scaling behavior of $\|\widehat\bbeta_\tau-\bbeta^*\|_2$ with respect to the effective sample size.
The data are generated in the same way as before except that the errors are drawn from $t_{1.5}$.
As discussed in the previous subsection, we take $\delta=0.45$ and then choose the robustification parameter as
$%\label{eq:tau}
\tau=c_\tau \widehat v_\delta ({n}/{\log d})^{ 1/(1+\delta) },%~\textnormal{where}~\delta=\textnormal{df-1-0.05}.
$ where $\widehat v_\delta$ is the $(1+\delta)$-th sample absolute central moment. For simplicity, we take  $c_\tau=0.5$ here since our goal is to demonstrate the scaling behavior as $n$ grows, instead of to achieve the best finite-sample performance.

The left panel of Figure \ref{fig:2} plots the $\ell_2$-error $\|\widehat\bbeta_{\tau, \lambda }-\bbeta^*\|_2$ versus sample size  over 200 repetitions when the dimension $d \in \{100, 500, 5000\}$. In all three settings, the $\ell_2$-error decays as the sample size grows. As expected, the curves shift to the right when the dimension increases. Theorem~\ref{thm:hd} provides a specific prediction about this scaling behavior: if we plot the $\ell_2$-error versus {effective sample size} ($n/\log d$), the curves should align roughly with the theoretical curve
\$
\|\widehat\bbeta_{\tau, \lambda } -\bbeta^*\|_2\asymp \bigg(\frac{n}{\log d}\bigg)^{-\delta/(1+\delta)}
\$
 for different values of $d$. This is validated empirically by the right panel of Figure~\ref{fig:2}. This near-perfect alignment in Figure \ref{fig:2} is also observed by  \cite{wainwright2009ieee} for  Lasso with sub-Gaussian errors.

\subsection{A Real Data Example: NCI-60 Cancer Cell Lines}
We apply  the proposed methodologies to the NCI-60, a panel of 60 diverse human cancel cell lines.
The NCI-60 consists of data on 60 human cancer cell lines and can be downloaded from \url{http://discover.nci.nih.gov/cellminer/}. More details on data acquisition can be found in \cite{shankavaram2007transcript}. Our aim is to investigate the effects of genes on protein expressions.
The gene expression data were obtained with an Affymetrix HG-U133A/B chip, $\log_2$ transformed and normalized with the guanine dytosine robust multi-array analysis.
 We then combined the same gene expression variables measured by multiple different probes into one   by taking  their median,  resulting in a set of $p=17,924$ predictors. The protein expressions based on 162 antibodies were acquired via reverse-phase protein lysate arrays in their original scale.
One observation had to be removed since all values were missing in the gene expression data, reducing the number of observations to $n = 59$.

We first center all the protein and gene expression variables to have mean zero, and then plot the histograms of the kurtosises of all expressions in Figure \ref{fig:hist}. The left panel in the figure shows that, 145 out of 162 protein expressions have  kurtosises larger than 3; and 49 larger than 9. In other words, more than 89.5\% of the protein expression variables have tails heavier than the normal distribution, and about 30.2\% are severely heavy-tailed with tails flatter than $t_5$, the $t$-distribution with 5 degrees of freedom. Similarly, about 36.5\% of the gene expression variables, even after the $\log_2$-transformation, still exhibit empirical kurtosises larger than that of $t_5$.  This suggests that, regardless of the normalization methods used, genomic data can still exhibit heavy-tailedness, which was also pointed out by \cite{PH2005}.

\begin{figure}[t!]
	\includegraphics[scale=.6]{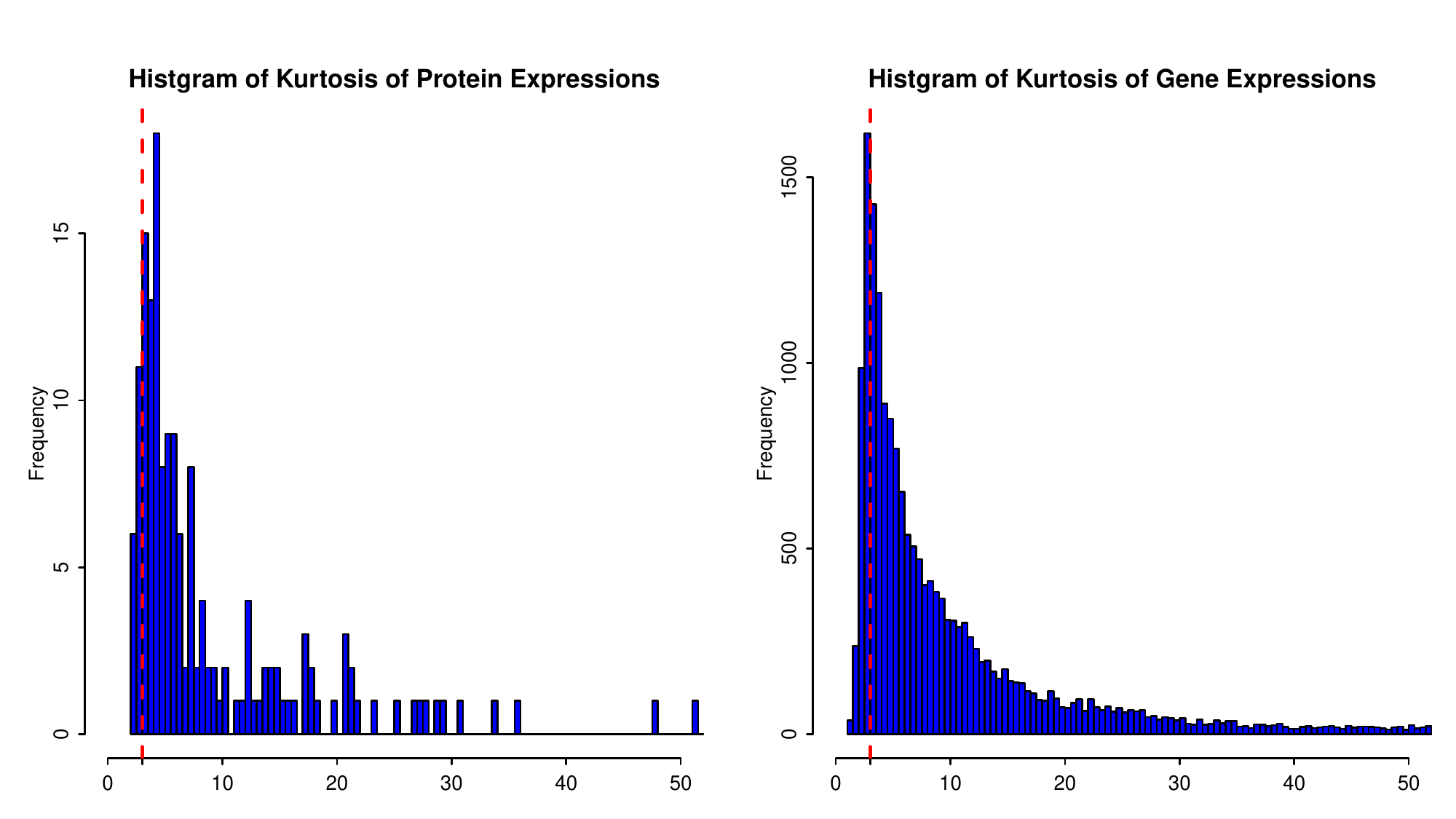}\centering\\
	\caption{Histogram of kurtosises for the protein and gene expressions. The dashed red line at 3 is the kurtosis of a normal distribution.}  \label{fig:hist}
\end{figure}

We  order the protein expression variables according to their scales, measured by the standard deviation. We show the results for the protein expressions based on the KRT19 antibody, the protein {keratin 19}, which constitutes the variable with the largest standard deviation, serving as one dependent variable. KRT19,  a type I keratin, also known as Cyfra 21-1, is encoded by the {\it KRT19} gene. %It  is a type I keratin, a member of the keratin family.
  Due to its high sensitivity, the KRT19 antibody is the most used marker for  the  tumor cells disseminated in lymph nodes, peripheral blood, and bone marrow of breast cancer patients \citep{nakata2004serum}.   We denote the adaptive Huber regression as AHuber, and that with truncated covariates as TAHuber.  We then compare {AHuber} and {TAHuber} with Lasso. Both regularization and robustification parameters are chosen by the ten-fold cross-validation.

To measure the predictive performance, we consider a robust prediction loss:  the mean absolute error (MAE) defined as
  \$
 \textnormal{MAE}\big(\widehat\bbeta\big)=\frac{1}{n_{\textnormal{test}}}\sum_{i=1}^{n_{\textnormal{test}}}\big| y_{i}^{\textnormal {test}}- \langle \bx_{i}^{\textnormal{test}} , \widehat\bbeta \rangle \big|,
 \$
 where $y^{\textnormal{test}}_i$ and $\bx^{\textnormal{test}}_i$, $i=1,\ldots,n_{\textnormal{test}}$, denote the observations of the response and predictor variables in the test data, respectively. We report the MAE via the leave-one-out cross-validation. Table \ref{tab:1}  reports the MAE, model size and selected genes for the considered methods. TAHuber clearly shows the smallest MAE, followed by AHuber and Lasso. The Lasso produces a fairly large model despite the small sample. Now it has been recognized that Lasso tends to select many noise variables along with the significant ones, especially when data exhibit heavy tails.

\begin{table}[t]
\setlength{\belowcaptionskip}{5pt}	
	\begin{center}
	\caption{We report the mean absolute error (MAE) for protein expressions based on the KRT19 antibody from the NCI-60 cancer cell lines, computed from leave-one-out cross-validation. We also report the model size and selected genes for each method.} 	\label{tab:1}%
\begin{tabular}{lccp{10.3cm}}
		\toprule
		Method & MAE   & Size  & Selected Genes  \\
		%\cmidrule(r){2-5} \cmidrule(r){6-9}
%		& \multicolumn{4}{c|}{Linear\textbackslash Case 1} & \multicolumn{4}{c}{Linear\textbackslash Case 2} \\
		\midrule
		Lasso &7.64 & 42 & {\it FBLIM1,   MT1E,     EDN2,     F3,       FAM102B,  S100A14, LAMB3,    EPCAM,    FN1,      TM4SF1,   UCHL1,  NMU,      ANXA3,    PLAC8,    SPP1,     TGFBI,    CD74,     GPX3,     EDN1,     CPVL,     NPTX2,    TES,     AKR1B10,  CA2,      TSPYL5,   MAL2,     GDA,      BAMBI,    CST6,     ADAMTS15, DUSP6,    BTG1,     LGALS3,  IFI27,    MEIS2,   TOX3,     KRT23,    BST2,     SLPI,     PLTP,     XIST,     NGFRAP1}\\
		AHuber &6.74 &11 & {\it MT1E, ARHGAP29, CPCAM, VAMP8, MALL, ANXA3, MAL2, BAMBI, LGALS3, KRT19, TFF3}\\
                 TAHuber & 5.76 & 7& {\it MT1E, ARHGAP29, MALL, ANXA3, MAL2, BAMBI, KRT19}\\
                 \bottomrule
\end{tabular}%
\end{center}
\end{table}%

  The Lasso selects a model with 42 genes but excludes the {\it KRT19} gene, which encodes the protein {keratin 19}. {AHuber} finds  11 genes including {\it KRT19}. %including the  {\it keratin 19} gene.
   {TAHuber}  results in a model with 7 genes:  {\it KRT19, MT1E, ARHGAP29, MALL, ANXA3, MAL2, BAMBI}. First, {\it KRT19} encodes the {keratin 19} protein. It has been reported in \cite{wu2008overlapping} that the {\it MT1E} expression is positively correlated with cancer cell migration and tumor stage, and the {\it MT1E}  isoform was found to be present in estrogen receptor-negative breast cancer cell lines \citep{friedline1998differential}.   {\it ANXA3} is highly expressed in all colon  cell lines and all breast-derived cell lines positive for the oestrogen receptor \citep{ross2000systematic}. A very recent study in \cite{zhou2017silencing} suggested that  silencing the {\it ANXA3} expression by RNA interference inhibits the proliferation and invasion of breast cancer cells.  Moreover, studies in \cite{shangguan2012inhibition} and \cite{kretzschmar2000transforming} showed that the {\it BAMBI} transduction significantly inhibited TGF-$\beta$/Smad signaling and expression of carcinoma-associated fibroblasts in human bone marrow mesenchymal stem cells (BM-MSCs), %treated with TGF-$\beta$1 or tumor-conditioned medium or cocultured with cancer cells,
 and disrupted the cytokine network mediating the interaction between MSCs and breast cancer cells. Consequently, the {\it BAMBI} transduction abolished protumor effects of BM-MSCs in vitro and in an orthotopic breast cancer xenograft model, and instead significantly inhibited growth and metastasis of coinoculated cancer. {\it MAL2} expressions were shown to be elevated at both RNA and protein levels in breast cancer \citep{shehata2008nonredundant}. It has also been shown that {\it MALL} is associated with various forms of cancer \citep{oh2005transcriptome, landi2014genome}. However, the effect of {\it ARHGAP29} and {\it MALL} on breast cancer remains unclear and is worth further investigation.

\section*{Supplementary Materials}

%Supplemental materials are submitted together with this manuscript and are intended to be for online publication only. 
In the supplementary materials, we provide theoretical analysis under random designs, and proofs of all the theoretical results in this paper. %Proposition~\ref{prop:error} and Theorems~\ref{thm:ld}--\ref{robust.hdreg.dev.ineq}, are provided.

\section*{Acknowledgments}
The authors  thank  the Editor, Associate Editor, and two anonymous referees for their valuable comments. This work is supported by a Connaught Award,   NSERC Grant RGPIN-2018-06484, NSF Grants DMS-1662139, DMS-1712591, and DMS-1811376, NIH Grant 2R01-GM072611-14, and NSFC Grant 11690014.

\section*{Appendix} 
\appendix

\section{A Lepski-type method}

Adapting the unknown robustification parameter depends on the value of the variance provided it exists. Through Lepski's renowned adaptation method \citep{Lepski1991}, this can be done without actually knowing the variance in advance.
Assume that $v_1 =  n^{-1} \sn \EE(\varepsilon_i^2) <\infty$ and let $\sigma_{\max}, \sigma_{\min}>0$ be such that
$
	\sigma_{\min} \leq v_1^{1/2} \leq \sigma_{\max}.
$
Here, parameters $\sigma_{\max}$ and $\sigma_{\min}$ serve as crude preliminary upper and lower bounds for $v_1^{1/2}$, respectively.

For a prespecified $a>1$, let $\sigma_j = \sigma_{\min} a^j$ and define the set
$$
	\cJ = \cJ_a  = \big\{ j= 0,1,2,\ldots : \sigma_{\min} \leq \sigma_j  < a \sigma_{\max} \big\}
$$
with its cardinality satisfying ${\rm card}(\cJ) \leq 1+\log_a(\sigma_{\max} / \sigma_{\min})$. For every predetermined $t>0$, compute a collection of Huber estimators $\{\hat{\bbeta}_{\tau_j} \}_{j\in \cJ}$, where $\tau_j = \sigma_j (n/t)^{1/2}$ for $j\in \cJ$. Set
\#
	\hat j = \min\Bigg\{  j \in \cJ : \big\| \Sb_n^{1/2}(\hat{\bbeta}_{\tau_k} - \hat{\bbeta}_{\tau_j} ) \big\|_2 \leq 8 \wt L \sigma_j \,d^{1/2} \sqrt{\frac{t}{n}} ~\mbox{ for all } k>j , k\in \cJ \Bigg\} , \nn
\#
where $\wt L : = \max_{1\leq i\leq n} \| \Sb_n^{-1/2} \bx_i \|_\infty$ assuming $\Sb_n = n^{-1} \sn \bx_i \bx_i^\T$ is positive definite.
The final data-driven estimator is then defined as $\hat{\bbeta} = \hat{\bbeta}_{\tau_{\hat j}}$.

\begin{theorem} \label{thm:lepski}
For any $t>0$, the data-dependent estimator $\hat \bbeta$ satisfies the bound
\#
	   \big\| \Sb_n^{1/2} ( \hat \bbeta  -\bbeta^* ) \big\|_2 \leq 12 a \wt L v_1^{1/2} d^{1/2} \sqrt{\frac{t}{n}}  \label{lepski.bound}
\#
with probability at least $1- (2d+1) \log_a(a \sigma_{\max} / \sigma_{\min}) e^{-t}$, provided the sample size satisfies $n\geq  8 \max  (4 \wt L^2 d , \wt L^4 d^2 ) t$.
\end{theorem}

Lepski-type construction relies on preliminary crude upper and lower bounds for $v_1^{1/2}$, which are usually unknown in advance.
In practice, one can take $\sigma_{\min} = \hat \sigma / K$ and $\sigma_{\max} = K \hat \sigma$ for some $K > 1$, where $\hat \sigma^2 := (n-d)^{-1} \sn (y_i - \langle \bx_i, \hat \bbeta^{{\rm ols}} \rangle )^2$ and $ \hat \bbeta^{{\rm ols}}$ is the least squares estimator.
Moreover, one may choose $a =1.5$ and $t=\log n$ or $\log(nd)$.
However, the effectiveness of this method depends on how sharp the  constants are in the theoretical bounds.
We note that all constants in Theorems 1 and \ref{thm:lepski} are explicit, although they might not be sharp.
Finding sharp constants remains open.
Since the current content already consists of long and technical arguments, we will not pursue this particular goal in this paper.
%In this case, we are not sure how sharp the obtained constants are, and finding sharp constants remains a challenging problem.
%In practice, we found that the cross-validation with a pilot estimator works well.

%The idea of Lepski's method is powerful and simple: consider a sequence of confidence intervals obtained by assuming that the variance is bounded by a sequence of bounds vk and pick up as an estimator the middle of the smallest interval intersecting all the larger ones. For this to be legitimate, we need all the confidence regions for which the variance bound is valid to hold together, which is performed using a union bound.

\begin{proof}[Proof of Theorem~\ref{thm:lepski}]
Following the proof of Theorem~\ref{thm:ld} which is given in Appendix~\ref{sec:1}, it can be similarly proved that, for any $\tau = \tau_0 (n/t)^{1/2}$ with $\tau_0\geq v_1^{1/2}$,
\#
	\big\| \Sb_n^{1/2} (\hat \bbeta_\tau - \bbeta^* ) \big\|_2 \leq 4  \wt L \tau_0 \, d^{1/2} \sqrt{\frac{t}{n}}  \label{new.fix.l2bound}
\#
with probability at least $1- (2d+1)e^{-t}$ as long as $n \geq  8 \max  (4 \wt L^2 d , \wt L^4 d^2 ) t$.

Let $j^* = \min\{ j \in \cJ: \sigma_j \geq v_1^{1/2}\}$ and note that $v_1^{1/2} \leq \sigma_{j^*} \leq a v_1^{1/2}$. By the definition of $\hat j$,
\#
	\{ \hat j > j^* \}  & \subseteq \bigcup_{j \in \cJ: j > j^*}  \Bigg\{  \big\|   \Sb_n^{1/2} ( \hat \bbeta_{\tau_j} - \hat \bbeta_{\tau_{j^*}}) \big\|_2 > 8 \wt L \sigma_j \, d^{1/2} \sqrt{\frac{t}{n}} \Bigg\}  \nn \\
& \subseteq \bigcup_{j \in \cJ: j \geq  j^*}  \Bigg\{  \big\|   \Sb_n^{1/2} ( \hat \bbeta_{\tau_j} -  \bbeta^* ) \big\|_2 > 4 \wt L \sigma_j \, d^{1/2} \sqrt{\frac{t}{n}} \Bigg\} . \nn
\#
Define the event
$$
	\cE = \bigcap_{j \in \cJ: j \geq  j^*}  \Bigg\{  \big\|   \Sb_n^{1/2} ( \hat \bbeta_{\tau_j} -  \bbeta^* ) \big\|_2 \leq  4 \wt L \sigma_j \, d^{1/2} \sqrt{\frac{t}{n}} \Bigg\}
$$
such that $\cE \subseteq \{ \hat j \leq j^* \}$. From \eqref{new.fix.l2bound} we see that for each $j\geq j^*$,
$$
\big\|   \Sb_n^{1/2} ( \hat \bbeta_{\tau_j} -  \bbeta^* ) \big\|_2 \leq  4 \wt L \sigma_j \, d^{1/2} \sqrt{\frac{t}{n}}
$$
with probability at least $1- (2d+1)e^{-t}$ under the prescribed sample size scaling.  By the union bound, we obtain that
\#
	\PP(\cE^{{\rm c}})& \leq \sum_{j\in \cJ: j\geq j^*}   \PP\Bigg\{  \big\|   \Sb_n^{1/2} ( \hat \bbeta_{\tau_j} -  \bbeta^* ) \big\|_2 > 4 \wt L \sigma_j \, d^{1/2} \sqrt{\frac{t}{n}} \Bigg\} \nn \\
&  \leq (2d+1) | \cJ | e^{-t}  \leq (2d+1) \{ 1 + \log_a(\sigma_{\max}/\sigma_{\min}) \} e^{-t} . \nn
\#
On the event $\cE$, $\hat j \leq j^*$ and thus
\#
		 \big\| \Sb_n^{1/2} ( \hat \bbeta  -\bbeta^* ) \big\|_2  & \leq  \big\| \Sb_n^{1/2} ( \hat \bbeta_{\tau_{\hat j}}  - \hat \bbeta_{\tau_{j^*}} ) \big\|_2 +  \big\| \Sb_n^{1/2} ( \hat \bbeta_{\tau_{j^*}}  -\bbeta^* ) \big\|_2 \nn \\
& \leq    8 \wt L \sigma_{j^*}\,d^{1/2} \sqrt{\frac{t}{n}} + 4 \wt L \sigma_{j^*} \,d^{1/2} \sqrt{\frac{t}{n}} \leq 12a \wt L v_1^{1/2} d^{1/2} \sqrt{\frac{t}{n}} . \nn
\#
Together, the last two displays yield \eqref{lepski.bound}.
\end{proof}

\section{Random Design Analysis}\label{sec:0}

%\subsubsection{Random Design}
%\label{sec:RDA}

In this section, we derive counterparts of the results in Section~3 under random designs.
First we impose the following moment conditions on the covariates and regression errors.

\begin{cond}\label{ass:3.1}
In linear model~\eqref{linear.model}, the covariate vectors $\bx_i \in \RR^d$ are i.i.d. from a sub-Gaussian random vector $\bx$, i.e. $\PP ( | \langle \bu, \wt \bx \rangle | \geq  y )  \leq 2 \exp (- y^2 \| \bu \|_2^2 / A_0^2  )$ for all $y\in \RR$ and $\bu \in \RR^d$, where $\wt \bx = \bSigma^{-1/2}\bx$ with $\bSigma = (\sigma_{jk})_{1\leq j,k\leq d} = \EE(\bx  \bx^\T)$ being positive definite and $A_0>0$ is a constant.
The regression errors $\varepsilon_i$ are independent and satisfy $\EE(\varepsilon_i |\bx_i)=0$ and $v_{i, \delta} = \EE(|\varepsilon_i|^{1+\delta} | \bx_i) <\infty$ almost surely for some $\delta>0$.
\end{cond}

Throughout this section, for simplicity, we assume the independent regression errors $\varepsilon_i$ in model \eqref{linear.model} are homoscedastic in the sense that $v_{i,\delta}$ does not depend on $\bx_i$.
The conditional heteroscedastic model can be allowed with slight modifications as before.
With this setup, we write
\#
 v_\delta= \frac{1}{n} \sn v_{i,\delta} ~~\mbox{ and }~~  \nu_\delta = \min \{v_\delta^{1/(1+\delta)}, v_1^{1/2} \} , \ \  \delta >0 . \nn
\#
Assuming the $d\times d$ matrix $\bSigma = \EE(\bx \bx^\T)$ is positive definite, we use $\| \cdot \|_{\bSigma,2}$ to denote the rescaled $\ell_2$-norm on $\RR^d$:
$$
	\| \bu \|_{\bSigma, 2} = \| \bSigma^{1/2} \bu \|_2 , \quad  \bu \in \RR^d.
$$
Moreover, we use $\psi_\tau$ to denote the derivative of Huber loss, that is,
\#
	\psi_\tau(x) = \ell_\tau'(x) =  \sgn(x) \min (|x| , \tau ) , \ \ x\in \RR. \label{psi.def}
\#

\subsection{Huber regression in low dimensions}

In the low dimensional regime ``$d\ll n$",  we consider the Huber estimator
\#
	\hat \bbeta_\tau = \arg\min_{\bbeta \in \RR^d} \cL_\tau(\bbeta ) ,   \nn
\#
where $\cL_\tau(\bbeta) = n^{-1} \sn \ell_\tau( y_i - \langle \bx_i, \bbeta \rangle )$ is the empirical Huber loss function and $\tau>0$ is the robustification parameter. Under Condition~\ref{ass:3.1}, the following theorem provides (i) exponential-type concentration inequalities for $\hat \bbeta_\tau$ when $\tau$ is properly calibrated, and (ii) a nonasymptotic Bahadur representation result under the finite variance condition on regression errors, i.e. $\delta=1$.

\begin{theorem}\label{thm:A1}
Suppose Condition~\ref{ass:3.1} holds.
\begin{itemize}
\item[(\Rom{1})] For any $t >0$ and $\tau_0 \geq \nu_\delta$, the estimator $\hat{\bbeta}_\tau$ with $\tau =\tau_0  \{ n/(d+t) \}^{\max\{1/(1+\delta),1/2\}}$ satisfies
\#
  \PP\Bigg\{ \big\|   \hat\bbeta_\tau - \bbeta^*   \big\|_{\bSigma, 2}  \geq  C_1   \tau_0 \bigg( \frac{d+t}{n} \bigg)^{  \min\{ \delta/(1+\delta) , 1/2 \} } \Bigg\} \leq 2 e^{-t} \label{CI.type2}
\#
as long as $n \geq C_2 (d+t)$, where $C_1 , C_2>0$ depend only on $A_0$.

\item[(\Rom{2})] Assume that $v_1 <\infty$. For any $t >0$ and $\tau_0 \geq v_1^{1/2}$, the estimator $\hat{\bbeta}_\tau$ with $\tau  = \tau_0  \sqrt{n/(d+t)}$ satisfies
\begin{align}
	\PP \Bigg\{ \bigg\|   \bSigma^{1/2} ( \hat\bbeta_\tau - \bbeta^*  ) -   \frac{1}{n}\sn \psi_\tau(\varepsilon_i)  \wt \bx_i  \bigg\|_2  \geq  C_3 \tau_0  \frac{d+t}{n}     \Bigg\} \leq  3 e^{-t}   \label{BR}
\end{align}
provided $n\geq C_2  (d+t)$, where $C_3>0$ depends only on $A_0$.
\end{itemize}
\end{theorem}

With random designs, the first part of  Theorem~\ref{thm:A1} provides concentration inequalities for the $\ell_2$-error under finite $(1+\delta)$-th moment conditions with $\delta>0$; when the second moments are finite, the second part gives a finite-sample approximation of $\hat{\bbeta}_\tau - \bbeta^*$ by a sum of independent random vectors. The remainder of such an approximation exhibits sub-exponential tails.  Unlike the least squares estimator, the adaptive Huber estimator does not admit an explicit closed-form
representation, which causes the main difficulty for analyzing its asymptotic and nonasymptotic properties.
Theorem~\ref{thm:A1} reveals that, up to a higher-order remainder, the distributional property of $\hat \bbeta_\tau$ mainly depends on a linear stochastic term that is much easier to deal with.

Regarding the truncated random variable $\psi_\tau(\varepsilon_i)$, the following result shows that the differences between the first two moments of $\psi_\tau(\varepsilon_i)$ and $\varepsilon_i$ depend on both $\tau$ and the moments of $\varepsilon_i$. The higher moment $\varepsilon_i$ has, the faster these differences decay as a function of $\tau$. We summarize this observation in the following  proposition.   We drop $i$ for ease of presentation.

\begin{proposition}  \label{Prop2}
Assume that $\EE(\varepsilon)=0$, $\sigma^2 = \EE(\varepsilon^2)>0$ and $\EE  (|\varepsilon |^{2+\kappa }  )<\infty$ from some $\kappa \geq 0$. Then we have
\$
	| \EE \psi_\tau(\varepsilon) | \leq  \min\big\{  \tau^{-1} \sigma^2,  \tau^{-1- \kappa } \EE\big( | \varepsilon |^{2+ \kappa}\big) \big\}. \nn
\$
Moreover, if $\kappa >0$,
\$
	\sigma^2 -  2 \kappa^{-1} \tau^{- \kappa} \EE\big(|\varepsilon|^{2+ \kappa } \big)  \leq \EE \{ \psi_\tau^2(\varepsilon) \}  \leq \sigma^2. \nn
\$
\end{proposition}

Proposition~\ref{Prop2}, along with Theorem~\ref{thm:A1}, shows that the adaptive Huber estimator achieves nonasymptotic robustness against heavy-tailed errors, while enjoying  high efficiency when $\tau$ diverges to $\infty$. In particular, taking $t =   \log n$, we see that under the scaling $n  \gtrsim d$, the robust estimator $\hat{\bbeta}_\tau$ with $\tau \asymp \sqrt{n/(d  +   \log n)}$ satisfies
\#
	\bigg\| \hat{\bbeta}_\tau -  {\bbeta}^* -   \frac{1}{n} \sn \psi_\tau(\varepsilon_i)   \bSigma^{-1}  \bx_i  \bigg\|_2 =    O\bigg( \frac{d+\log n}{n} \bigg) \nn
\#
with probability at least $1- O(n^{-1})$. From an asymptotic point of view, this implies that if the dimension $d$, as  a function of $n$, satisfies
\#
	d =  o(n) ~\mbox{ as } n\to \infty , \nn
\#
then for any deterministic vector $\ba \in \RR^d$, the distribution of $\langle \ba , \hat{\bbeta}_\tau - {\bbeta}^* \rangle$ is close to that of $n^{-1} \sn \psi_\tau(\varepsilon_i) \langle \ba,  \bSigma^{-1}\bx_i \rangle $. If $\varepsilon_1,\ldots, \varepsilon_n$ are independent from $\varepsilon$ with variance $\sigma^2$ and $\EE (|\varepsilon|^{2+\kappa}  )<\infty$ for some $\kappa > 0$, taking  $\tau \asymp \sqrt{n/(d+\log n)}$ in Proposition~\ref{Prop2} implies that  $n^{-1/2} \sn \psi_\tau(\varepsilon_i) \langle \ba,   \bSigma^{-1}\bx_i  \rangle$ follows  a normal distribution with mean zero and variance $ \sigma^2 \| \bSigma^{-1/2} \ba \|_2^2$ asymptotically.

 \subsection{Huber regression in high dimensions}

In the high dimensional setting where $d\gg n$ and $s= \| \bbeta^* \|_0 \ll n$, we investigate the $\ell_1$-regularized Huber estimator
\#
\hat \bbeta_{\tau, \lambda } \in \arg\min_{\bbeta \in \RR^d}  \big\{ \cL_\tau(\bbeta ) + \lambda \| \bbeta \|_1 \big\}   \label{regularized.huber}
\#
under Condition~\ref{ass:3.1}, where $\tau$ and $\lambda$ represent, respectively, the robustification and regularization parameters.

\begin{theorem} \label{hd.huber}
Assume Condition~\ref{ass:3.1} holds and that the unknown $\bbeta^*$ is sparse with $s=\| \bbeta^* \|_0$.
Then any optimal solution $\hat \bbeta_{\tau, \lambda}$ to the convex program \eqref{regularized.huber} with
\#
	\tau = \tau_0 \bigg( \frac{n}{\log d} \bigg)^{ \max\{ 1/(1+\delta) , 1/2 \}}   ~ ( \tau_0 \geq \nu_\delta )
\#
and  $\lambda$ scaling as $A_0 \sigma_{\max}   \tau_0 \{ (\log d) / n \}^{ \min \{ \delta/(1+\delta) , 1/2\} }$ satisfies the bounds
\#
	 \big\|   \hat \bbeta_{\tau, \lambda} - \bbeta^* \big\|_{\bSigma, 2} \lesssim     \kappa_l^{-1/2} A_0  \sigma_{\max} \tau_0 \, s^{1/2} \bigg( \frac{\log d}{n} \bigg)^{ \min\{ \delta/(1+\delta) , 1/2 \} }  \nn  \\
	 \mbox{ and }~~  \big\| \hat \bbeta_{\tau, \lambda} - \bbeta^*  \big\|_1 \lesssim   \kappa_l^{-1}  A_0 \sigma_{\max} \tau_0 \, s \bigg( \frac{\log d}{n} \bigg)^{ \min\{ \delta/(1+\delta) , 1/2 \} }    \label{l1huber.bounds}
\#
with probability at least $1-3d^{-1}$ as long as $n\geq C   \kappa_l^{-1} \sigma_{\max}^2 s \log d$, where $C>0$ is a constant only depending on $A_0$, $\sigma_{\max} = \max_{1\leq j\leq d} \sigma_{jj}^{1/2}$ and $\kappa_l = \lambda_{\min}(\bSigma)$.
\end{theorem}

Provided the distribution of $\varepsilon_i$ has finite variance, i.e. $\delta=1$, Theorem~\ref{hd.huber} asserts that the $\ell_1$-regularized Huber regression with properly tuned $(\tau, \lambda)$ gives rise to statistically consistent estimators with $\ell_1$- and $\ell_2$-errors scaling as $s \sqrt{(\log d )/n}$ and $\sqrt{s ( \log d ) /n}$, respectively, under the sample size scaling $n \gtrsim s \log d$. These rates are the minimax rates enjoyed by the standard Lasso with Gaussian/sub-Gaussian errors \citep{BRT2009, W2009}.

 The results of Theorem~\ref{hd.huber} are useful complements to those in Theorem~4 under fixed designs. Taking $t=\log d$ therein, we see that the $\ell_2$-error bound in (10) almost coincides with that in \eqref{l1huber.bounds} up to constant factors. The sample size scaling under random designs is optimal and better than the scaling under fixed designs: the former is of order $O(s \log d)$, while the latter is of order $O(s^2 \log d)$. Technically, the sample size scaling is required to ensure the restricted strong convexity of Huber loss in a neighborhood of $\bbeta^*$; see Lemma~1 in the main text and Lemma~\ref{RSC.prop} below.
Since most existing works on analyzing high dimensional $M$-estimators beyond the least squares have focused on random designs (see, e.g. \cite{BC2011}, \cite{NRWY2012} and the references therein), it is not clear what the optimal sample size scaling is under fixed designs, although it is possible that the additional $s$ factor in Theorem~4 is purely an artifact of the proof technique.
We refer to \cite{vdG2008} for a study of generalized linear models in high dimensions. To achieve the oracle rate for the excess risk, the sparsity $s$ is required to be of order $O(\sqrt{n/\log n})$, or equivalently, the required sample size scales as $s^2 \log n$.

We complete this section by a prediction error bound for $\hat \bbeta_{\tau, \lambda}$, which is a direct consequence of Theorem~\ref{hd.huber}.

\begin{corollary} \label{hd.huber.corr}
Under the conditions of Theorem~\ref{hd.huber}, it holds
\# \label{prediction.error}
	 \frac{1}{\sqrt{n}}  \big\|  \Xb ( \hat \bbeta_{\tau, \lambda} - \bbeta^* )   \big\|_2  \lesssim     \kappa_l^{-1/2} A_0  \sigma_{\max} \tau_0 \, s^{1/2} \bigg( \frac{\log d}{n} \bigg)^{ \min\{ \delta/(1+\delta) , 1/2 \} }
\#
with probability at least $1-5d^{-1}$, where $\Xb = (\bx_1,\ldots, \bx_n)^\T $ is the $n\times d$ design matrix.
\end{corollary}

\section{Proofs of Main Theorems} \label{sec:1}

Throughout the proofs, we use $\psi_\tau = \ell'_\tau$ as in definition \eqref{psi.def} and let $\| \cdot \|_{\bSigma,2}$ be the rescaled $\ell_2$-norm on $\RR^d$ given by $\| \bu \|_{\bSigma, 2} = \| \bSigma^{1/2} \bu \|_2$ for $\bu \in \RR^d$.

\subsection{Auxiliary Lemmas}

First we collect several auxiliary lemmas. Our first lemma concerns the localized analysis that can be utilized to remove the parameter constraint in previous works. It is established in \cite{fan2015tac} and we reproduce it here for completeness.

\begin{lemma}\label{lm01}
Let  $D_\cL(\bbeta_1,\bbeta_2)=\cL(\bbeta_1)-\cL(\bbeta_2)- \langle \nabla \cL(\bbeta_2),\bbeta_1-\bbeta_2  \rangle$ and $D_\cL^s(\bbeta_1,\bbeta_2)=D_\cL(\bbeta_1,\bbeta_2)+D_\cL(\bbeta_2,\bbeta_1)$. For $\bbeta_\eta  =\bbeta^*+ \eta (\bbeta-\bbeta^*)$ with $\eta \in (0,1]$ and any convex loss functions $\cL$, we have
	\$
	D_\cL^s(\bbeta_{\eta} ,\bbeta^*)\leq \eta D_\cL^s(\bbeta,\bbeta^*).
	\$
\end{lemma}

\begin{proof}[Proof of Lemma~\ref{lm01}]
Let  $Q(\eta)=D_{\cL}(\bbeta_\eta,\bbeta^*)=\cL(\bbeta_\eta)-\cL(\bbeta^*)- \langle\nabla\cL(\bbeta^*),\bbeta_\eta-\bbeta^* \rangle$. Noting that the derivative of $\cL(\bbeta_\eta)$ with respect to $\eta$ is $ \frac{d}{d \eta} \cL(\bbeta_\eta) = \langle \nabla \cL(\bbeta_\eta), \bbeta-\bbeta^*\rangle$, we have
	\$
	Q'(\eta) = \langle\nabla\cL(\bbeta_\eta)-\nabla\cL(\bbeta^*),\bbeta-\bbeta^* \rangle.
	\$
Then, the symmetric Bregman divergence $D_{\cL}^s(\bbeta_\eta-\bbeta^*)$ can be written as
\$
D_{\cL}^s(\bbeta_\eta, \bbeta^*)= \langle\nabla\cL(\bbeta_\eta)-\nabla\cL(\bbeta^*), \eta(\bbeta-\bbeta^*) \rangle=\eta Q'(\eta) , \ \ 0<\eta\leq 1.
\$
Taking $\eta =1$ in the above equation, we have
$
Q'(1)=D_{\cL}^s(\bbeta,\bbeta^*)
$
as a special case. If $Q(\eta)$ is convex, then $Q'(\eta)$ is non-decreasing and thus
\$
D_{\cL}^s(\bbeta_\eta,\bbeta^*)= \eta Q'(\eta)\leq \eta Q'(1)=\eta D_{\cL}^s(\bbeta,\bbeta^*).
\$

It remains to show the convexity of  $\eta \in [0,1] \mapsto  Q(\eta)$; or equivalently, the convexity of $\cL (\bbeta_\eta )$ and $\langle \nabla\cL(\bbeta^*), \bbeta^*-\bbeta_\eta\rangle$, respectively.
First, note that $\bbeta_\eta$, as a function of $\eta$, is linear in $\eta$, that is, $\bbeta_{\alpha_1 \eta_1+\alpha_2 \eta_2}=\alpha_1 \bbeta_{\eta_1}+\alpha_2 \bbeta_{\eta_2}$  for all $\eta_1, \eta_2\in [0,1]$ and $\alpha_1, \alpha_2\geq 0$ satisfying $\alpha_1+\alpha_2=1$. Then, the convexity of $\eta \mapsto \cL(\bbeta_\eta)$ follows from this linearity and the convexity of the Huber loss. The convexity of the second term follows directly from the bi-linearity of the inner product.
\end{proof}

The following two lemmas provide restricted strong convexity properties for the Huber loss in a local vicinity of the true parameter under both fixed and random designs.

\begin{lemma}\label{lm02.fix}
Assume that Condition~\ref{ass:3.0} holds and that $v_\delta = n^{-1} \sn \e ( |\varepsilon_i |^{1+\delta} ) <\infty$ for some $0<\delta \leq 1$. Then for any $t, r >0$, the Hessian matrix $ \nabla^2 \cL_\tau(\bbeta) $ with $\tau > 2 M r$ satisfies that, with probability greater than $1- e^{-t}$,
\begin{align}
 & \min_{\bbeta\in \RR^d: \|   \bbeta - \bbeta^*   \|_2\leq r} \lambda_{\min} \big(  \nabla^2 \cL_\tau(\bbeta)  \big)  \nn \\
 & \quad \quad \quad \quad \quad \geq   \big\{ 1 -  (2 M r/\tau)^2 \big\}  c_l -  M^2  \big\{   (2/\tau)^{1+\delta} v_\delta + (2n)^{-1/2} t^{1/2}  \big\} ,   \label{Hessian.lbd}
\end{align}
where $M = \max_{1\leq i\leq n} \| \bx_i \|_2$.
\end{lemma}

\begin{proof}[Proof of Lemma~\ref{lm02.fix}]
To begin with, note that
$$
	\Hb_n (\bbeta) =  \nabla^2 \cL_\tau(\bbeta) = \frac{1}{n} \sum_{i=1}^n  \bx_i  \bx_i^\T 1\big( |y_i - \bx_i^\T \bbeta | \leq \tau \big) ,
$$
where $\mathbf{S}_n$ is given in Condition~\ref{ass:3.0}. %and $\wt \bx_i = \Sb_n^{-1/2} \bx_i$'s are such that $n^{-1}\sn \wt \bx_i \wt \bx_i^\T = \Ib_d$.
For each $\bbeta \in \RR^d$, define its centered and rescaled version $\bbeta_0 = \bbeta - \bbeta^*$ such that $y_i - \langle \bx_i , \bbeta \rangle = \varepsilon_i - \langle  \bx_i ,  \bbeta_0 \rangle$. Using the inequality that
$$
	1\big(  |y_i - \langle \bx_i , \bbeta \rangle | >  \tau \big) \leq 1\big( | \varepsilon_i | >  \tau/2 \big) +  1\big( |  \langle  \bx_i , \bbeta_0 \rangle |  >  \tau/2 \big),
$$
we have, for any $\bu  \in \mathbb{S}^{d-1}$ and $\bbeta \in \RR^{d}$ satisfying $\| \bbeta_0 \|_2  \leq r$,
\begin{align}
	& \langle \bu,  \Hb_n (\bbeta) \bu  \rangle  \nn \\
	& \geq  \| \Sb_n^{1/2} \bu   \|_2^2 - \frac{1}{n} \sn \langle  \bx_i , \bu \rangle^2 1 \big(  | \varepsilon_i | > \tau/2 \big)   -  \frac{1}{n} \sum_{i=1}^n \langle \bx_i , \bu \rangle^2  1 \big( | \langle  \bx_i , \bbeta_0 \rangle  | > \tau/2 \big)  \nn \\
	& \geq \|   \Sb_n^{1/2}  \bu   \|_2^2 - \max_{1\leq i\leq n} \| \bx_i \|_2^2 \, \bigg\{ \frac{1}{n} \sn  1\big( | \varepsilon_i | >  \tau/2 \big)  +  \frac{4}{\tau^2}\| \bbeta_0  \|_2^2 \,  \|  \Sb_n^{1/2} \bu   \|_2^2 \bigg\}  \nn \\
	  & \geq    c_l \big\{ 1 -  (2 M r/\tau)^2 \big\} -     \frac{M^2}{n}  \sum_{i=1}^n  1\big( | \varepsilon_i | >  \tau/2 \big) , \nn
\end{align}
provided that $\tau > 2 M r$.
For any $z\geq 0$, it follows from Hoeffding's inequality that, with probability at least $1- e^{-2n z^2}$,
\begin{align}
   \frac{1}{n} \sum_{i=1}^n   1 \big( | \varepsilon_i | >  \tau/2 \big)  \leq \frac{1}{n} \sn  \PP\big( | \varepsilon_i  | >\tau/2 \big)   +   z . \nn
\end{align}
This, together with the inequality $\PP( | \varepsilon_i  | > \tau/2 ) \leq (2/\tau)^{1+\delta} v_{i,\delta}$ and Condition~\ref{ass:3.0}, implies that, with probability at least $1- e^{ -2n z^2}$,
\begin{align}
	\big\langle \bu, \Hb_n (\bbeta)\bu \big\rangle
	  \geq   \big\{ 1 -  (2 M r/\tau)^2 \big\}    c_l -   M^2   \big\{  (2/\tau)^{1+\delta} v_\delta + z   \big\}  . \nn
\end{align}
This proves \eqref{Hessian.lbd} immediately by taking $z=\sqrt{t/(2n)}$.
\end{proof}

\medskip
\begin{lemma}   \label{lm02}
Assume $v_\delta < \infty$ for some $0<\delta \leq 1$ and $( \EE \langle \bu , \wt \bx \rangle^4 )^{1/4} \leq  A_1 \| \bu \|_2 $ for all $\bu \in \RR^d$ and some constant $A_1>0$.
Moreover, let $\tau, r>0$ satisfy
\#
 \tau \geq  2 \max\big\{  (4v_\delta)^{1/(1+\delta)}  , 4 A_1^2 r   \big\} ~~\mbox{ and }~~ n \gtrsim (\tau/r)^2(d+t). \label{RSC.scaling}
\#
Then with probability at least $1-e^{-t}$,
\#
	\big\langle \nabla \cL_\tau(\bbeta) - \nabla \cL_\tau(\bbeta^*), \bbeta - \bbeta^*  \big\rangle \geq \frac{1}{4} \big\| \bbeta - \bbeta^*  \big\|_{\bSigma, 2}^2  \label{RSC.bound}
\#
uniformly over $\bbeta \in \Theta_0(r) = \{ \bbeta \in \RR^d : \|  \bbeta -\bbeta^*  \|_{\bSigma, 2} \leq r \}$.
\end{lemma}

\begin{proof}[Proof of Lemma~\ref{lm02}]
To begin with, note that
\begin{align}
	\cT(\bbeta) & := \langle \nabla \cL_\tau(\bbeta) - \nabla \cL_\tau(\bbeta^*) ,  \bbeta - \bbeta^* \rangle \nn \\
	&   = \frac{1}{n} \sn  \{  \psi_\tau(y_i - \langle \bx_i, \bbeta^* \rangle ) - \psi_\tau(y_i - \langle \bx_i , \bbeta \rangle )  \}
	\langle \bx_i , \bbeta-\bbeta^* \rangle \nn \\
	& \geq   \frac{1}{n} \sn  \{   \psi_\tau( \varepsilon_i ) - \psi_\tau(y_i - \langle \bx_i , \bbeta \rangle ) \}
	\langle \bx_i , \bbeta-\bbeta^* \rangle  1\{ \cE_i \}, \label{T.def}
\end{align}
where $1\{ \cE_i \}$ denotes the indication function of the event
$$
	\cE_i = \big\{ | \varepsilon_i | \leq \tau/2 \big\} \cap  \big\{ | \langle \bx_i, \bbeta -\bbeta^* \rangle | \leq \tau \|   \bbeta -\bbeta^*  \|_{\bSigma, 2} / (2r) \big\} .
$$
On $\cE_i$, it holds $|y_i - \langle \bx_i, \bbeta \rangle | \leq |\varepsilon_i | + | \langle \bx_i, \bbeta -\bbeta^* \rangle | \leq \tau/2 + \tau/2 = \tau $ for all $\bbeta \in  \Theta_0(r)$. Since $\psi_\tau'(x) = 1$ for $|x| \leq \tau$, the right-hand of \eqref{T.def} can be bounded from below by
\begin{align}
 \frac{1}{n}  \sn \langle \bx_i , \bbeta -\bbeta^* \rangle^2 1  \big\{ | \langle \bx_i, \bbeta -\bbeta^* \rangle | \leq \tau \|   \bbeta -\bbeta^*  \|_{\bSigma, 2} / (2r) \big\} 1 \big\{ |\varepsilon_i | \leq \tau/2 \big\}. \label{T.lbd1}
\end{align}

To bound the right-hand of \eqref{T.lbd1}, the main difficulty is that the indicator function is non-smooth. To deal with this issue, we define the following ``smoothed" functions: for any $R>0$, write
$$
	\phi_R(x) = \begin{cases}
	 x^2   & \mbox{ if } |x| \leq R/2 , \\
    (x-R)^2 & \mbox{ if } R/2 < x \leq R , \\
    (x+R)^2 & \mbox{ if } -R \leq x \leq -R/2, \\
    0 & \mbox{ if } |x| > R,
\end{cases}    ~\mbox{ and }~   \varphi_R(y) = 1( |y| \leq R ).
$$
It is easy to see that the function $\phi_R$ is $R$-Lipschitz and satisfies
\#
	 x^2 1(|x| \leq R/2) \leq \phi_R(x) \leq x^2 1(|x| \leq R) . \label{phi.bound}
\#
Together, \eqref{T.def}, \eqref{T.lbd1}  and \eqref{phi.bound} imply
\#
	 \cT(\bbeta) \geq g(\bbeta ) : = \frac{1}{n } \sn \phi_{ \tau \|   \bbeta -\bbeta^*  \|_{\bSigma, 2} / (2r)} (\langle \bx_i, \bbeta - \bbeta^* \rangle) \varphi_{\tau/2}(\varepsilon_i).\label{def.g}
\#
For $r>0$, define  $\Delta(r) = \sup_{\bbeta \in \Theta_0(r)} |g(\bbeta) - \EE g(\bbeta)|/  \| \bbeta - \bbeta^* \|_{\bSigma, 2}^2$, such that
\#
	 \frac{\cT(\bbeta)}{\| \bbeta - \bbeta^* \|_{\bSigma, 2}^2} \geq \frac{\EE g(\bbeta)}{\| \bbeta - \bbeta^* \|_{\bSigma, 2}^2} - \Delta(r)
\#
for all $\bbeta \in \Theta_0(r)$. In the following, we establish lower and upper bounds for $\EE g(\bbeta)$ and $\Delta(r)$, respectively, starting with the former.

For $\bbeta \in \RR^d$, write $\bdelta = \bbeta - \bbeta^*$. By \eqref{def.g} and Markov's inequality,
\#
	\EE g(\bbeta)  & \geq \frac{1}{n} \sn  \EE \langle \bx_i, \bdelta \rangle^2 - \frac{1}{n} \sn \EE \langle \bx_i , \bdelta \rangle^2 1\big\{   |\langle \bx_i, \bdelta \rangle | \geq \tau \| \bdelta \|_{\bSigma,2}/(4r) \big\} \nn \\
& \quad~ -  \frac{1}{n} \sn \EE \langle \bx_i, \bdelta \rangle^2 1(|\varepsilon_i | > \tau/2) \nn \\
& \geq \bdelta^\T \bSigma \bdelta  -  v_\delta (2/\tau)^{1+\delta} \bdelta^\T \bSigma \bdelta  -(4r/\tau)^2 \| \bdelta \|_{\bSigma, 2}^{-2} \frac{1}{n} \sn \EE \langle \bx_i, \bdelta \rangle^4   \nn \\
& \geq    \| \bdelta \|_{\bSigma, 2}^2 \big\{ 1   -  v_\delta (2/\tau)^{1+\delta}   -    (4A_1^2 r/\tau)^2  \big\} . \nn
\#
Provided $\tau\geq 2\max\{   (4v_\delta )^{1/(1+\delta)}, 4A_1^2 r \}$,
\#
	 \EE g(\bbeta)  \geq \frac{1}{2}  \| \bbeta - \bbeta^* \|_{\bSigma, 2}^2 ~\mbox{ for all } \bbeta \in \RR^d. \label{Eg.lbd}
\#

Next we bound the supremum $\Delta(r)$. Write $g(\bbeta) = n^{-1} \sn g_i(\bbeta)$. Noting that $0\leq \phi_R(x) \leq R^2/4$ and $0\leq \varphi(y) \leq 1$, we have
$$
		0 \leq  g_i(\bbeta) \leq  (\tau/4r)^2 \| \bbeta -\bbeta^* \|_{\bSigma,2}^2.
$$
By Theorem~7.3 in \cite{B2003}, for any $x>0$, $\Delta(r)$ satisfies the bound
\# \label{Delta.concentration}
	 \Delta(r) \leq  \EE  \Delta(r) + \{ \EE \Delta(r) \}^{1/2}  (\tau /2r) \sqrt{\frac{x}{n}} + \sigma_n \sqrt{\frac{2x}{n}} + (\tau / 4r)^2 \frac{x}{3n}
\#
with probability at least $1-e^{-x}$, where by \eqref{phi.bound},
$$
	 \sigma_n^2  := \frac{1}{n} \sn \sup_{\bbeta \in \Theta_0(r)}   \frac{\EE g^2_i(\bbeta)}{\| \bbeta - \bbeta^* \|_{\bSigma, 2}^4} \leq  A_1^4.
$$
For the expected value $\EE \Delta(r)$, using the symmetrization inequality and the connection between Gaussian complexity and Rademacher complexity, we obtain that $\EE \Delta(r) \leq \sqrt{2 \pi} \, \EE \{ \sup_{ \bbeta \in \Theta_0(r)}  | \mathbb{G}_{\bbeta } | \} $, where
$$
	\mathbb{G}_{\bbeta} = \frac{1}{n} \sn \frac{G_i}{\| \bbeta - \bbeta^* \|_{\bSigma, 2}^2} \phi_{\tau \| \bbeta - \bbeta^* \|_{\bSigma, 2}/(2r)} (\langle\bx_i, \bbeta -\bbeta^* \rangle ) \varphi_{\tau/2}(\varepsilon_i)
$$
and $G_i$ are i.i.d. standard normal random variables that are independent of $\{(y_i, \bx_i)\}_{i=1}^n$.
Let $\EE^*$ be the conditional expectation given $\{(y_i, \bx_i)\}_{i=1}^n$. Since $\{ \mathbb{G}_{\bbeta} : \bbeta \in \Theta_0(r) \}$ is a conditional Gaussian process, for any $\bbeta_0\in \Theta_0(r)$ we have
\#
	 \EE^*\bigg\{ \sup_{\bbeta \in \Theta_0(r)} | \mathbb{G}_{\bbeta} | \bigg\} \leq \EE^* |\mathbb{G}_{\bbeta_0}| + 2 \EE^* \bigg\{ \sup_{\bbeta\in \Theta_0(r)} \mathbb{G}_{\bbeta} \bigg\} . \label{gaussian.sup.bound}
\#
Further, taking the expectation with respect to $\{(y_i, \bx_i)\}_{i=1}^n$ on both sides, \eqref{gaussian.sup.bound} remains valid with $\EE^*$ replaced by $\EE$. We write $\bbeta^*$ as $(\beta^*_1, \wt \bbeta^{* \T})^\T$ with $\beta^*_1$ denoting the first coordinate of $\bbeta^*$ and $\wt \bbeta^* \in \RR^{d-1}$.
Recalling $\phi_R(u) \leq \min (u^2 , R^2/4)$, we take $\bbeta_0 = ( \beta^*_1 + ( \EE x_1^2)^{-1/2}  r  , \wt \bbeta^{* \T})^\T$ so that $\| \bbeta_0 - \bbeta^* \|_{\bSigma, 2} =r $ and
$\EE | \mathbb{G}_{\bbeta_0} | \leq (  \EE \mathbb{G}_{\bbeta_0}^2 )^{1/2}  \leq  (4r)^{-1}\tau  n^{-1/2}$.
To bound the conditional expectation $\EE^* \{ \sup_{\bbeta\in \Theta_0(r)} \mathbb{G}_{\bbeta} \}$ in \eqref{gaussian.sup.bound}, we employ the Gaussian comparison theorem  as in the proof of Lemma~11 in \cite{LW2015}

Denote by $\var^*$ the conditional variance given $\{(y_i, \bx_i)\}_{i=1}^n$. For $\bbeta, \bbeta' \in \Theta_0(r)$, write $\bdelta = \bbeta -\bbeta^*$ and $\bdelta' = \bbeta'-\bbeta^*$. By conditional normality, we quickly compute and bound the variance of $\mathbb{G}_{\bbeta} -\mathbb{G}_{\bbeta'}$:
$$
	\var^*(\mathbb{G}_{\bbeta} -\mathbb{G}_{\bbeta'}) \leq \frac{1}{n^2} \sn \varphi_{\tau/2}^2(\varepsilon_i)
\Bigg\{   \frac{ \phi_{\tau \| \bdelta \|_{\bSigma,2}/(2r)}(\langle\bx_i, \bdelta \rangle) } {\| \bdelta \|_{\bSigma, 2}^2}  -  \frac{\phi_{\tau \| \bdelta' \|_{\bSigma,2}/(2r)}(\langle\bx_i, \bdelta' \rangle) } {\| \bdelta' \|_{\bSigma, 2}^2} \Bigg\}^2.
$$
Using the property $\phi_{cR}(cx) = c^2 \phi_{R}(x)$ for any $c>0$, we find that
$$
	\phi_{ \tau \| \bdelta' \|_{\bSigma,2}/(2r) } (\langle \bx_i ,\bdelta' \rangle ) = \frac{\| \bdelta' \|_{\bSigma, 2}^2}{ \| \bdelta \|_{\bSigma, 2}^2 } \phi_{\tau \| \bdelta \|_{\bSigma, 2}/ (2r)} \bigg( \frac{\| \bdelta' \|_{\bSigma, 2} }{ \| \bdelta \|_{\bSigma, 2} } \langle \bx_i, \bdelta \rangle  \bigg).
$$
It follows from the above calculations and the Lipschitz property of $\phi_R$ that
\#
	& 	\var^*(\mathbb{G}_{\bbeta} -\mathbb{G}_{\bbeta'})  \nn \\
& \leq \frac{1}{n^2} \sn \frac{1}{\| \bdelta  \|_{\bSigma, 2}^4} \bigg\{   \phi_{\tau \| \bdelta \|_{\bSigma,2}/(2r)}(\langle\bx_i, \bdelta \rangle) -
\phi_{\tau \| \bdelta \|_{\bSigma, 2}/ (2r)} \bigg( \frac{\| \bdelta' \|_{\bSigma, 2} }{ \| \bdelta \|_{\bSigma, 2} } \langle \bx_i, \bdelta \rangle  \bigg) \bigg\}^2 \nn \\
& \leq \frac{1}{n^2} \sn \frac{\tau^2}{4r^2 }  \bigg(  \frac{\langle\bx_i, \bdelta \rangle }{ \| \bdelta  \|_{\bSigma, 2}} - \frac{\langle  \bx_i, \bdelta' \rangle }{ \| \bdelta'  \|_{\bSigma, 2}} \bigg)^2 . \label{diff.var.bound}
\#
Let $G_1',\ldots, G_n'$ be i.i.d. standard normal random variables that are independent of all the previous variables, and define a new process
$$
	\ZZ_{\bbeta} = \frac{\tau}{2r n} \sn G_i' \frac{\langle  \bx_i, \bbeta - \bbeta^* \rangle }{ \| \bbeta -\bbeta^* \|_{\bSigma, 2}} .
$$
As an immediate consequence of \eqref{diff.var.bound}, we have $\var^*(\mathbb{G}_{\bbeta} -\mathbb{G}_{\bbeta'})  \leq \var^*(\mathbb{Z}_{\bbeta} -\mathbb{Z}_{\bbeta'}) $. Therefore, by the Gaussian comparison inequality \citep{LT1991},
$$
	\EE^* \bigg\{ \sup_{\bbeta\in \Theta_0(r)} \mathbb{G}_{\bbeta} \bigg\} \leq 2 \EE^* \bigg\{ \sup_{\bbeta\in \Theta_0(r)} \mathbb{Z}_{\bbeta} \bigg\} \leq \frac{\tau}{r}  \EE^* \bigg\|  \frac{1}{n} \sn G_i' \wt \bx_i \bigg\|_2 ,
$$
where $\wt \bx_i = \bSigma^{-1/2} \bx_i$.
Taking the expectation with respect to $\{(y_i, \bx_i)\}_{i=1}^n$ on both sides gives $\EE \{ \sup_{\bbeta\in \Theta_0(r)} \mathbb{G}_{\bbeta} \} \leq (\tau / r) \EE \| n^{-1} \sn G_i' \wt \bx_i \|_2 \leq (\tau/r)\sqrt{d/n}$. From this and the unconditional version of \eqref{gaussian.sup.bound}, we obtain
\#
	 \EE \Delta(r) \leq \sqrt{2\pi}  \bigg(  \frac{2\tau}{r} \sqrt{\frac{d}{n}} + \frac{\tau}{4r\sqrt{n}} \bigg) .\label{mean.Delta.ubd}
\#
Together, \eqref{Delta.concentration} with $x=t$  and \eqref{mean.Delta.ubd} imply that as long as $n\gtrsim (\tau/r)^2(d+t)$, $\Delta(r) \leq 1/4$ with probability at least $1- e^{-t}$. Combining this with \eqref{T.lbd1} and \eqref{Eg.lbd} proves the stated result.
\end{proof}

Recall that $\Theta_0(r)= \{ \bbeta \in \RR^d : \| \bbeta - \bbeta^* \|_{\bSigma, 2} \leq r\}$. Let $ \cC =  \{ \bbeta  \in \RR^d  :  \| (\bbeta - \bbeta^*)_{\cS^{{\rm c}}} \|_1 \leq 3\| (\bbeta - \bbeta^*)_{\cS} \|_1 \}$ be an $\ell_1$-cone in $\RR^d$, where $\cS \subseteq \{ 1,\ldots, d\}$ denotes the support of $\bbeta^*$.
As a counterpart of Lemma~\ref{lm02} in high dimensions, Lemma~\ref{RSC.prop} below shows that the adaptive Huber loss satisfies the restricted strong convexity condition over $\Theta_0(r)  \cap \cC$ with high probability.

\begin{lemma} \label{RSC.prop}
Assume $v_\delta < \infty$ for some $0<\delta \leq 1$ and $( \EE \langle \bu , \wt \bx \rangle^4 )^{1/4} \leq  A_1 \| \bu \|_2 $ for all $\bu \in \RR^d$ and some constant $A_1>0$.
Let $(n,d,\tau, r)$ satisfy
\#   \label{sample.size.scaling.1}
  \tau \geq  2 \max\big\{  (4v_\delta)^{1/(1+\delta)}  , 4 A_1^2 r   \big\}  ~\mbox{ and }~ n\gtrsim   \kappa_l^{-1}    (A_0 \tau/r)^2 \max_{1\leq j\leq d} \sigma_{jj}  \,  s \log d ,
\#
Then with probability at least $1-d^{-1}$,
\# \label{RSC.bound0}
\big\langle \nabla \cL_\tau(\bbeta) - \nabla \cL_\tau(\bbeta^*) , \bbeta  - \bbeta^* \big\rangle  \geq \frac{1}{4} \big\|    \bbeta  - \bbeta^*  \big\|_{\bSigma, 2}^2\#
{uniformly over} $\bbeta  \in  \Theta_0(r) \cap \cC$.
\end{lemma}

\begin{proof}[Proof of Lemma~\ref{RSC.prop}]

The proof is based on an argument similar to that in the proof of Lemma~\ref{lm02}. With slight abuse of notation, we keep using $\Delta(r)$ as the supremum of a random process:% indexed by $\Theta_0(r) \cap \cC$:
\#
	\Delta(r)  =   \sup_{ \bbeta  \in \Theta_0(r) \cap \cC }  \frac{ | g(\bbeta   ) - \EE  g(\bbeta  ) | }{\| \bbeta -\bbeta^* \|_{\bSigma ,2}^2 }. \nn
\#
Provided $\tau \geq 2 \max\{  (4 v_\delta)^{1/(1+\delta)} , 4 A_1^2 r\}$, it can be shown that
\#
	\frac{\cT(\bbeta )}{\| \bbeta  -  \bbeta^* \|_{\bSigma ,2}^2 }  \geq \frac{1}{2}   - \Delta(r) ~\mbox{ for all }~   \bbeta  \in \Theta_0(r) \cap \cC . \label{Lower.bound.1'}
\#
According to \eqref{Delta.concentration}, it remains to bound $\EE \Delta(r)$. Following the proof of Lemma~\ref{lm02}, it suffices to focus on the (conditional) Gaussian process
\#
	\ZZ_{\bbeta  }  =  \frac{\tau }{2r n}  \sn G_i' \frac{ \langle \bx_i, \bbeta  - \bbeta^* \rangle }{\| \bbeta  - \bbeta^*  \|_{\bSigma,2} } ,  \ \ \bbeta \in \Theta_0(r) \cap \cC , \nn
\#
where $G'_i $ are i.i.d. standard normal random variables that are independent of all other random variables.
For every $\bbeta \in  \Theta_0(r) \cap \cC$, it is easy to see that
$$
	 \| \bbeta - \bbeta^* \|_1  \leq 4\sqrt{s} \, \| \bbeta - \bbeta^* \|_2 \leq 4 \kappa_l^{-1/2} \sqrt{s} \, \| \bbeta - \bbeta^* \|_{\bSigma , 2} ,
$$
implying
$$
  \sup_{ \bbeta  \in  \Theta_0(r) \cap \cC  } \ZZ_{\bbeta } \leq   2 \kappa_l^{-1/2} \sqrt{s} \, \frac{\tau}{r}   \bigg\| \frac{1}{n} \sn G'_i   \bx_i \bigg\|_{\infty} .
$$
Keep all other statements the same, we obtain
\#
	\EE  \Delta(r)  \leq \sqrt{2\pi} \Bigg(   8  \kappa_l^{-1/2} \sqrt{s} \,  \frac{  \tau}{r}  \EE \bigg\| \frac{1}{n} \sn G_i'  \bx_i \bigg\|_{\infty} + \frac{ \tau }{4r \sqrt{n}} \Bigg) . \nn
\#
With $\bx_i = (x_{i1} , \ldots, x_{id})^\T \in \RR^d$, note that
$$
	\bigg\| \frac{1}{n} \sn  G_i'  \bx_i \bigg\|_{\infty} = \max_{1\leq j\leq  d } \bigg| \frac{1}{n} \sn G'_i   x_{ij} \bigg|   .
$$
Since $G'_i   x_{ij}$ are sub-exponential/sub-gamma random variables, from Corollary~2.6 in \cite{BLM2013} we find that
$$
  \EE \bigg\| \frac{1}{n} \sn G_i'  \bx_i \bigg\|_{\infty} \lesssim  A_0 \max_{1\leq j\leq d}  \sigma_{jj}^{1/2}   \bigg(  \sqrt{\frac{\log d}{n}} + \frac{\log d}{n} \bigg) .
$$
Substituting this  into \eqref{Delta.concentration} and taking $x=\log d$, we obtain that with probability at least $1-d^{-1}$,
\#
	\frac{\cT(\bbeta )}{\| \bbeta  -\bbeta^* \|^2_{\bSigma , 2} } \geq \frac{1}{4}  ~\mbox{ uniformly over }  \bbeta \in \Theta_0(r) \cap \cC  \nn
\#
for all sufficiently large $n$ that scales as $ \kappa_l^{-1} (A_0 \tau /r)^2 \max_{1\leq j\leq d}  \sigma_{ jj}  \, s \log d$
up to an absolute constant. This proves \eqref{RSC.bound0}.
\end{proof}

Lemmas~\ref{lm:grad.l2} and \ref{lm:grad} provide concentration inequalities for $\| \bSigma^{-1/2} \nabla \cL_\tau(\bbeta^*) \|_2$ and $\| \nabla \cL_\tau(\bbeta^*) \|_{\infty}$, respetively.

\begin{lemma} \label{lm:grad.l2}
Assume Condition~\ref{ass:3.1} holds with $0<\delta \leq 1$.  Then with probability at least $1-2e^{-t}$,
\#
	 \big\| \bSigma^{-1/2} \nabla \cL_\tau (\bbeta^* ) \big\|_2 \leq 4\sqrt{2} A_0  \sqrt{\frac{v_\delta \tau^{ 1-\delta } (d+t)}{n}} + 2A_0 \tau \frac{d+t}{n} +  v_\delta \tau^{-\delta} .  \label{r0.def}
\#
\end{lemma}

\begin{proof}[Proof of \ref{lm:grad.l2}]
Assume without loss of generality that $t\geq \log 2$, or equivalently, $2e^{-t} \leq 1$; otherwise $2e^{-t} >  1$ so that the bound is trivial.
To bound $\| \bSigma^{-1/2} \nabla\cL_\tau(\bbeta^*)\|_2$, first define the centered random vector
$$
	\bxi^* =  \bSigma^{-1/2}  \big\{  \nabla \cL_\tau(\bbeta^* ) -   \nabla  \EE\cL_\tau(\bbeta^* ) \big\} =  - \frac{1}{n}\sn \big\{ \xi_i \wt \bx_i - \EE(  \xi_i \wt \bx_i  ) \big\} ,
$$
where $\xi_i = \psi_\tau(\varepsilon_i)$. To evaluate the $\ell_2$-norm, there exits a $1/2$-net $\mathcal{N}_{1/2}$ of the unit sphere $\mathbb{S}^{d-1}$ in $\RR^d$ with $| \mathcal{N}_{1/2} | \leq 5^d$ such that $\| \bxi^* \|_2 \leq 2 \max_{\bu \in \mathcal{N}_{1/2}} |\langle \bu,  \bxi^* \rangle |$. Under Condition~\ref{ass:3.1}, it holds for every $\bu \in \mathbb{S}^{d-1}$ that $\EE   | \langle \bu,   \wt \bx \rangle  |^k  \leq A_0^k \, k \Gamma(k/2)$ for all $k \geq 1$. By direct calculations,
\#
	\sn \EE(\xi_i  \langle \bu,   \wt \bx_i \rangle )^2 \leq   2A_0^2 \,\tau^{1-\delta}  \sn v_{i,1} = 2 A_0^2 \, n v_\delta   \tau^{1-\delta}  ,\nn \\
	\sn \EE |\xi_i  \langle \bu,   \wt \bx_i \rangle  |^k \leq \frac{k!}{2} (A_0 \tau /2)^{k-2}  2 A_0^2 \,  n v_\delta \tau^{1-\delta} ~\mbox{ for all } k \geq 3 . \nn
\#
It then follows from Bernstein's inequality that
\#
	\PP\bigg\{ | \langle \bu ,  \bxi^* \rangle | \geq  2 A_0 \sqrt{\frac{v_\delta \tau^{1-\delta} x}{n}} +  (A_0 /2) \frac{ \tau x}{  n} \bigg\} \leq 2 e^{-x} ~\mbox{ for any } x >0. \nn
\#
Taking the union bound over $\bu \in \mathcal{N}_{1/2}$, we obtain that with probability at least $1- 5^d \cdot 2 e^{-x}$,
\#
	\| \bxi^* \|_2 \leq 4A_0 \sqrt{\frac{v_\delta \tau^{1-\delta} x}{n}} +   A_0 \frac{ \tau x}{  n} . \label{xi*.concentration.rd}
\#
Next, for the deterministic part $\| \bSigma^{-1/2} \nabla \EE \cL_\tau(\bbeta^*) \|_2$, it is easy to see that
$$
 \big\| \bSigma^{-1/2} \nabla \EE \cL_\tau(\bbeta^*) \big\|_2 = \sup_{\bu \in \mathbb{S}^{d-1}} \frac{1}{n} \sn \EE |\xi_i \langle  \bu, \wt \bx_i \rangle |  \leq    v_\delta \tau^{-\delta} .
$$
Combining this and \eqref{xi*.concentration.rd} with $x=2(d+t)$, we reach the bound \eqref{r0.def} which holds with probability at least $1-2e^{-2t} \geq 1- e^{-t}$.
\end{proof}

\medskip
\begin{lemma}  \label{lm:grad}
Assume Condition~\ref{ass:3.1} holds with $0<\delta \leq 1$. Then with probability at least $1- 2d^{-1}$,
\#	
	\| \nabla \cL_\tau (\bbeta^* ) \|_{\infty} \leq       \max_{1\leq j\leq d} \sigma_{ jj}^{1/2} \Bigg(  2 \sqrt{2}  A_0    \sqrt{  \frac{v_\delta \tau^{1-\delta}  \log d}{n}} +  A_0  \frac{\tau \log d}{n}  + v_\delta \tau^{-\delta}  \Bigg) .  \nn
\#
\end{lemma}

\begin{proof}[Proof of Lemma~\ref{lm:grad}]
The proof is based on Bernstein's inequality and the union bound.
Define $\xi_i  =\psi_\tau(\varepsilon_i)$ for $i=1,\ldots, n$ such that $\nabla \cL_\tau(\bbeta^*) = -n^{-1} \sn \xi_i \bx_i $.
For every $1\leq j\leq d$, note that $| \EE (\xi_i x_{ij}) | = | \EE \{ \EE(\xi_i | x_{ij}) x_{ij} \}| \leq \sigma_{jj}^{1/2} v_\delta \tau^{-\delta}$.
Moreover, from the proof of Lemma~\ref{lm:grad.l2} we see that
\#
		\sn \EE (\xi_i x_{ij})^2 \leq  \sigma_{jj} n v_\delta \tau^{1-\delta } , \nn \\
		\sn \EE |\xi_i x_{ij}|^k \leq \frac{k!}{2} 2 A_0^2 \sigma_{jj} n v_\delta \tau^{1-\delta }    (A_0 \sigma_{jj}^{1/2} \tau/2 )^{k-2}    ~\mbox{ for } k\geq 3 . \nn
\#
By Bernstein's inequality, for any $x>0$ it holds
$$
	 \bigg| \frac{1}{n} \sn (\xi_i x_{ij} - \EE \xi_i x_{ij}) \bigg| \leq 2 A_0 \sigma_{jj}^{1/2} \sqrt{\frac{v_\delta \tau^{1-\delta } x}{n}} +  A_0 \sigma_{jj}^{1/2} \frac{\tau x}{2 n}
$$
with probability at least $1- 2e^{-x}$. By the union bound and taking $x=2\log d$ in the last display, we arrive at the stated result.
\end{proof}

\subsection{Proof of Proposition~\ref{prop:error}}

Define the error vector $\bDelta = \bbeta^* - \bbeta^*_\tau$ and function $h(\bbeta) = n^{-1} \sn \e \{ \ell_\tau (y_i - \langle \bx_i , \bbeta \rangle ) \}$, $\bbeta \in \RR^d$. By the optimality of $\bbeta^*_\tau$ and the mean value theorem, we have $\nabla h(\bbeta^*_\tau) = \textbf{0}$ and thus
\#
	\langle \bDelta, \nabla^2 h (\tilde{\bbeta}_1) \bDelta \rangle = \langle \nabla h(\bbeta^* ) - \nabla h(\bbeta_\tau^* ),  \bDelta \rangle  = \langle \nabla h(\bbeta^*) , \bDelta \rangle = -\frac{1}{n} \sn \EE \{ \psi_\tau(\varepsilon_i) \} \langle \bx_i , \bDelta \rangle , \label{basic.equality}
\#
where $\tilde{\bbeta}_1 = \lambda \bbeta^* + (1-\lambda) \bbeta^*_\tau$ for some $0\leq \lambda \leq 1$.

\medskip
\noindent
{\sc Case 1.} First we consider the case of $0<\delta <1$. Since $\EE ( \varepsilon_i ) =0$, we have
$
	- \EE \{ \psi_\tau(\varepsilon_i ) \} =   \EE \{ \varepsilon_i 1 (|\varepsilon_i | > \tau )   - \tau1 ( \varepsilon_i >\tau ) + \tau1 ( \varepsilon_i < -\tau )  \}
$
and therefore
\#
	|\EE \{ \psi_\tau(\varepsilon_i) \} | \leq \EE \big\{ \big( |\varepsilon_i | - \tau \big) 1\big( |\varepsilon_i | > \tau \big) \big\} \leq  v_{i, \delta}  \tau^{-\delta }  . \label{gradient.mean.bound}
\#

Taking $\tilde{\varepsilon}_i = y_i - \langle \bx_i, \tilde{\bbeta}_1 \rangle$, we see that
\#
	\nabla^2 h(\tilde{\bbeta}_1 )   = \Sb_n - \frac{1}{n} \sn \PP\big(|\tilde{\varepsilon}_i|  > \tau \big)  \bx_i \bx_i^\T . \label{Hessian}
\#
Note that
\$
	& \EE \{ \ell_\tau(\varepsilon_i ) \}  \\
	& \leq \EE \bigg\{  \frac{\tau^{1-\delta}}{2} |\varepsilon_i |^{1+\delta}  1\big( |\varepsilon_i | \leq \tau \big) + \bigg(  \tau^{1-\delta } |\varepsilon_i |^{1+\delta } - \frac{\tau^{2-\tau}}{2} |\varepsilon_i |^\delta \bigg) 1\big( |\varepsilon_i | >\tau \big)  \bigg\} \leq  v_{i, \delta} \tau^{1-\delta }.
\$
This, together with the convexity of $h$ implies that $h(\tilde{\bbeta}_1) \leq \lambda h(\bbeta^*) + (1-\lambda) h(\bbeta^*_\tau ) \leq  h(\bbeta^*) \leq v_\delta  \tau^{1-\delta} $, where $v_\delta = n^{-1} \sn v_{i,\delta}$. For the lower bound, note that $h(\bbeta) \geq  n^{-1}  \EE \{   ( \tau | y_i- \langle \bx_i , \bbeta \rangle | - \tau^2/2) \} 1 (|y_i - \langle \bx_i , \bbeta \rangle | >\tau)  \}$ for all $\bbeta\in \RR^d$. Putting these upper and lower bounds on $h(\tilde{\bbeta}_1)$ together yields
\$
	\frac{\tau}{n} \sn \EE |\tilde{\varepsilon}_i | 1\big(|\tilde{\varepsilon}_i | >\tau \big) \leq \frac{\tau^2}{2n} \sn \PP\big(|\tilde{\varepsilon}_i | >\tau \big) +  v_\delta \tau^{1-\delta},
\$
as a consequence of which $n^{-1} \sn \PP(|\tilde{\varepsilon}_i | > \tau) \leq  2 v_\delta \tau^{-1-\delta} $. Combining this with \eqref{Hessian}, we deduce that as long as $\tau >  (2 v_\delta  \wt M^2)^{1/(1+\delta)}$,
\#
 	 \bDelta^{{\rm T}} \,  \nabla^2 h(\tilde{\bbeta}_1 ) \bDelta & \geq \| \Sb_n^{1/2} \bDelta \|_2^2 - \frac{1}{n} \sn   \PP \big(|\tilde{\varepsilon}_i|  > \tau \big) \langle \bx_i , \bDelta \rangle^2 \nn \\
 	& \geq \| \Sb_n^{1/2} \bDelta \|_2^2  - 2 \| \Sb_n^{1/2} \bDelta \|_2^2 \max_{1\leq i\leq n} \| \Sb_n^{-1/2} \bx_i \|_2^2 \, v_\delta \tau^{-1-\delta}   \nn \\
 	& \geq \big( 1 -     2 v_\delta  \wt M^2  \tau^{-1-\delta } \big) \| \Sb_n^{1/2} \bDelta \|_2^2  . \nn
\#
This provides a lower bound for the left-hand side of \eqref{basic.equality}. On the other hand, using \eqref{gradient.mean.bound} and H\"older's inequality to bound the right-hand side of \eqref{basic.equality}, the claim \eqref{approxi.error} for $0<\delta
 <1$ follows immediately.

\medskip
\noindent
{\sc Case 2.} Next we assume $\delta\geq 1$ and note that $v_{i,1} = \EE (\varepsilon_i^2)$. In this case, we have $\EE \{ \ell_\tau(\varepsilon_i ) \} \leq \frac{1}{2} v_{i,1}$ and $|\EE \{ \psi_\tau(\varepsilon_i ) \} | \leq v_{i,\delta} \tau^{-\delta}$. Then, following the same arguments as above, it can be shown that as long as $\tau > v_1^{1/2} m_n $,
\#
     \big( 1- v_1 m_n^2 \tau^{-2}  \big) \| \Sb_n^{1/2} \bDelta \|_2^2   \leq  \langle \bDelta ,   \nabla^2 h(\tilde{\bbeta}_1 ) \bDelta \rangle  \leq   \| \Sb_n^{1/2} \bDelta \|_2 \, v_\delta  \tau^{-\delta} .
\#
This proves \eqref{approxi.error} for $\delta\geq 1$ and hence completes the proof. \qed

\subsection{Proof of Theorem~\ref{thm:ld}}

Without loss of generality, we assume $t\geq 1$ throughout the proof; otherwise, $3e^{-t} \geq 1$ and the stated result holds trivially. For simplicity, we write $\hat{\bbeta} = \hat{\bbeta}_\tau$.  Note that for any prespecified $r>0$, we can construct an intermediate estimator, denoted by $\hat\bbeta_{\tau, \eta} =\bbeta^*+ \eta (\hat\bbeta-\bbeta^*)$, such that $\| \hat\bbeta_{\tau, \eta} -\bbeta^* \|_2\leq r$. To see that, we take $\eta =1$ if $\| \hat\bbeta-\bbeta^* \|_2\leq r$; otherwise, we can always choose some $\eta \in(0,1)$ so that  $\|  \hat\bbeta_{\tau, \eta}  -\bbeta^*  \|_2=r$.
Applying Lemma \ref{lm01} gives
\#
 \big\langle \nabla\cL_\tau(\hat\bbeta_{\tau, \eta})-\nabla\cL_\tau(\bbeta^*), \hat\bbeta_{\tau, \eta} -\bbeta^* \big\rangle \leq \eta \big\langle \nabla\cL_\tau(\hat\bbeta)-\nabla\cL_\tau(\bbeta^*), \hat\bbeta-\bbeta^* \big\rangle , \label{localized.analysis}
\#
where $\nabla\cL_\tau(\hat\bbeta) = \mathbf{0}$ according to the Karush-Kuhn-Tucker condition. By the mean value theorem for vector-valued functions, we have
\$
 \nabla\cL_\tau(\hat\bbeta_{\tau, \eta})-\nabla\cL_\tau(\bbeta^*) = \int_0^1 \nabla^2 \cL_\tau(  (1-t) \bbeta^* + t \hat \bbeta_{\tau, \eta} ) \, dt \, (\hat \bbeta_{\tau, \eta} - \bbeta^* ).\nn
\$
If, there exists some $a_0>0$ such that
\#\label{pos}
 \min_{\bbeta \in \RR^d :\|  \bbeta   -\bbeta^*  \|_2 \leq r }  \lambda_{\min} \left( \nabla^2\cL_\tau( \bbeta )    \right) \geq a_0 ,
\#
then we have $a_0\| \hat\bbeta_{\tau, \eta} -\bbeta^*  \|_2^2 \leq \|   \nabla\cL_\tau(\bbeta^*)\|_2  \|  \hat\bbeta_{\tau, \eta} -\bbeta^*  \|_2$. Canceling the common factor on both sides yields
\begin{align}
\big\| \hat\bbeta_{\tau, \eta}  - \bbeta\big\|_2\leq a_0^{-1} \big\|   \nabla\cL_\tau(\bbeta^*) \big\|_2.  \label{pos1}
\end{align}

Define the random vector $\bxi^*=  \nabla\cL_\tau(\bbeta^*)$, which can be  written as
$$
	\bxi^* = - \frac{1}{n} \sum_{i=1}^n  \psi_\tau(\varepsilon_i)  \bx_i. %~\mbox{ with }~ \wt \bx_i = (\wt x_{i1}, \ldots, \wt x_{ip})^\T =  \Sb_n^{-1/2} \bx_i .
$$
By definition \eqref{psi.def}, $\psi_1(x)  =   \tau^{-1}  \psi_\tau(\tau x )$. We write $\Psi_j = \!n^{-1}\sum_{i=1}^n   ( x_{ij} / L)\psi_1( \varepsilon_i /\tau)$ for $j=1, \ldots, d$, such that $\| \bxi^* \|_2 \leq  d^{1/2} \| \bxi^* \|_\infty =  Ld^{1/2} \tau \max_{1\leq j\leq d} | \Psi_j |$. With $0<\delta\leq 1$, it is easy to see that the function $\psi_1(\cdot)$ satisfies
 \# \label{eq:tail}
 -\log (1-u+|u|^{1+\delta})  \leq \psi_1(u)\leq \log (1+u+|u|^{1+\delta})
 \#
for all $u\in \RR$. It follows that
\begin{align}
	  (  x_{ij} / L) \psi_1( \varepsilon_i /\tau) & \leq  (  x_{ij} / L) 1(    x_{ij}    \geq 0) \log (1+  \varepsilon_i /\tau + | \varepsilon_i / \tau |^{1+\delta}) \nn \\
& \qquad   -  (  x_{ij} / L) 1(   x_{ij}  < 0) \log (1 -  \varepsilon_i / \tau + | \varepsilon_i / \tau |^{1+\delta}) . \nn
\end{align}
This, together with the inequality $(1+u)^v \leq 1+ uv$ for $u \geq -1$ and $0<v\leq 1$, implies
\begin{align}
& \exp\{   (  x_{ij} / L) \psi_1(\varepsilon_i/\tau) \}  \nn \\
& \leq (1+\varepsilon_i/\tau+|\varepsilon_i/\tau|^{1+\delta})^{  ( x_{ij} / L) 1(  x_{ij} \geq 0)} + (1 - \varepsilon_i/\tau + |\varepsilon_i/\tau|^{1+\delta})^{ -   (  x_{ij} / L)1( x_{ij} < 0)} \nn \\
& \leq 1 +  ( \varepsilon_i  /\tau)  ( x_{ij} / L) + |\varepsilon_i/\tau |^{1+\delta}. \nn
\end{align}
Consequently, we have
\begin{align}
	&  \EE \{ \exp(n\Psi_j ) \} = \prod_{i=1}^n \e \exp\{   (  x_{ij} / L) \psi_1(\varepsilon_i/\tau)\} \leq \prod_{i=1}^n ( 1 + v_{i,\delta}  \tau^{-1-\delta}  ) \leq  \exp(  v_\delta n \tau^{-1-\delta} ), \nn
\end{align}
where we used the inequality $1+ u \leq e^u$ in the last step. For any $z\geq 0$, using Markov's inequality gives
 \$
 \PP  (  \Psi_j \geq v_\delta z  ) \leq  \exp( -  v_\delta n  z) \EE \{  \exp( n\Psi_j )  \} \leq   \exp\{   v_\delta n ( \tau^{-1-\delta }   -  z  ) \} .
 \$
As long as $\tau \geq (2/z)^{1/(1+\delta)}$, we have $\PP (  \Psi_j  \geq  v_\delta  z ) \leq e^{-   v_\delta  n  z/2 }$. On the other hand, it can be similarly shown that $\PP( -  \Psi_j \geq   v_\delta  z ) \leq  e^{-   v_\delta  n z/2 }$. For any $t>0$, taking $z=2t/( v_\delta n)$ in these two inequalities yields that as long as $\tau \geq   (v_\delta n/t)^{1/(1+\delta)}$,
\begin{align}
	& \PP\big(  \|  \bxi^* \|_2\geq 2  Ld^{1/2}  \tau  n^{-1} t   \big)  \nn \\
	&  \leq  \PP\big(  \|  \bxi^* \|_{\infty} \geq 2 L \tau  n^{-1} t   \big)   \leq \sum_{j=1}^d  \PP\big(  | \Psi_j | \geq  2  n^{-1} t \big)   \leq 2  d \exp(-t)  . \label{L2bound}
\end{align}

Taking $r  = \tau /(4\sqrt{2} M)$, it follows from Lemma~\ref{lm02.fix} and the definition of $\tau$ that with probability at least $1- e^{- t}$, \eqref{pos} holds with $a_0 = c_l/2$ provided
$$
	n \geq     \max\big(  8 M^4 c_l^{-2} , 2^{4+\delta} M^2 c_l^{-1} \big)  t .
$$
Combining \eqref{pos1} and \eqref{L2bound} implies that, with probability at least $1-(2d+1)e^{-  t}$,
$$
	\big\|  \hat{\bbeta}_{\tau, \eta}  - \bbeta^* \big\|_2 < 4 Lc_l^{-1}d^{1/2} \tau   n^{-1} t  .
$$
Provided $n\geq 16\sqrt{2} c_l^{-1} L M d^{1/2} t $, the intermediate estimator $\hat{\bbeta}_{\tau, \eta}$  will lie in the interior of the ball with radius $r$.
By our construction in the beginning of the proof, this enforces $\eta=1$ and thus $\hat{\bbeta} =\hat  \bbeta_{\tau, \eta}$.  \qed

\subsection{Proof of Theorem~\ref{thm:ld:mini}}

We start by defining a simple  class of distributions for the response variable $y$ as $\cP_{c,\gamma}=\big\{\PP_{c+}, \PP_{c-}\big\}$, where
\$
\PP^+_c(\{0\})=1-\gamma,~\PP^+_c(\{c\})=\gamma,~\textnormal{and}~\PP^-_c(\{0\})=1-\gamma,~\PP^-_c(\{-c\})=\gamma.
\$
Here, we suppress the dependence of $\PP^+_c$ and $\PP^-_c$  on  $\gamma$ for convenience. It follows that, for any $0<\delta\leq 1$, the $(1+\delta)$-th absolute central moment $v_\delta$ of $y$ with law either $\PP_c^+$ or $\PP_c^-$ is %equal to %$c^{1+\delta}\gamma$, implying that
\#\label{thm:min:ld:eq:1}
v_\delta=|c|^{1+\delta}\gamma(1-\gamma)\{ \gamma^\delta+(1-\gamma)^{\delta}\}.
\#
%We first derive a simple lower bound for the class of distributions for the response $y$:  $\cP_{c,\gamma}=\{\PP_{+}, \PP_-\}$, defined by
%\$
%\PP_{+}(\{0\})=\PP_-(\{0\})=1-\gamma, ~~~\PP_+(\{c\})=\PP_-(\{-c\})=\gamma,
%\$
%where $\gamma\in [0,1]$ and $c>0$. It then follows that, for any
% $0<\delta\leq 1$, the $(1+ \delta)$th absolute central moment of $y$ is
%\$
%v_\delta\equiv c^{1+\delta}\gamma(1-\gamma) \{ \gamma^{\delta}+(1-\gamma)^{\delta} \}.
%\$
For $i=1, \ldots, n$, let $(y_{1i}, y_{2i})$ be independent pairs of real-valued random variables satisfying
\begin{gather*}
\PP(y_{1i}=y_{2i}=0)=1-\gamma, ~\PP(y_{1i}=c_i,y_{2i}=-c_i)=\gamma,
~\textnormal{and}~y_{1i}\sim \PP^+_{c_i}, ~y_{2i}\sim \PP^-_{c_i}.
\end{gather*}
Let $\by_{k}=(y_{k1},\ldots, y_{kn})^\T$ for $k=1,2$, and $\xi \in (0,1/2].$ Taking $\gamma= \log  \{1/(2\xi) \}  / (2n)$ with $\xi\geq  e^{-n}/2$, we obtain $1-\gamma\geq 1/2$ and
\$
\PP\big(\by_{1}=\by_{2}={\bf 0} \big) = (1-\gamma)^n\geq \Big\{\exp\Big(\frac{-\gamma}{1-\gamma}\Big)\Big\}^n\geq 2\xi.
\$
By assumption, we know that there is an $n$-dimensional vector $\ub\in \{-1, +1\}^n$ with each coordinate taking $-1$ or $1$ such that $\frac{1}{n}\|\Xb^\T \ub\|_{\min}\geq \alpha$. Note that this assumption naturally holds for the mean model, where $\Xb=(1, \ldots, 1)^\T$ and $\alpha$ can be taken as $1$.  Now we take $\bc$, $\bbeta_1^*$ and $\bbeta_2^*$ such that $\bc=c\ub$ for a $c>0$, $\Xb\bbeta_1^*=\bc\gamma$ and $\bbeta_2^*=-\bbeta_1^*$, which indicates that
\begin{gather*}
\bbeta^*_1=\Big(\frac{1}{n}\Xb^\T \Xb\Big)^{-1}\frac{1}{n}\Xb^\T \ub, ~~\textnormal{and}~~\\
\big\|\bbeta^*_1\big\|_2\geq c\gamma\bigg\|\Big(\frac{1}{n}\Xb^\T \Xb\Big)^{-1}\frac{1}{n}\Xb^\T \ub\bigg\|_2\geq c\gamma  \frac{ d^{1/2} }{c_u} \|\Xb^\T \ub/n\|_{\min} \geq c\gamma  \frac{d^{1/2} \alpha}{c_u}   .
\end{gather*} %Take  $\bbeta_1^*$, $\bbeta_2^*$ and $\bc$ such that $\bbeta_1^*=\gamma {\mathbf 1}_d$, $\Xb\bbeta^*_1= \gamma\bc$  and $\bbeta_2^*\!=\!-\bbeta_1^*$.
Let $\widehat\bbeta_{k}(\by_{k})$ be any estimator possibly depending on $\xi$,  then the above calculation yields
\#\label{eq:thm3:1}
&\max\Big\{\PP\big(\big\|\hat\bbeta_{1}-\bbeta_1^*\big\|_2\geq c\gamma  c_u^{-1} d^{1/2} \alpha   \big), \PP\big(\big\|\hat\bbeta_{2}-\bbeta_2^*\big\|_2\geq c\gamma c_u^{-1} d^{1/2} \alpha    \big)\Big\}\notag\\
&\geq  \frac{1}{2} \PP\Big( \big\|\hat\bbeta_{1}-\bbeta_1^*\big\|_{2}\geq c\gamma  c_u^{-1} d^{1/2} \alpha   ~\textnormal{or}~\big\|\hat\bbeta_{2}-\bbeta_2^*\big\|_2\geq c\gamma  c_u^{-1} d^{1/2} \alpha \Big)\notag\\
&\geq \frac{1}{2}\PP\big(\hat\bbeta_{1}=\hat\bbeta_{2}\big)\geq \frac{1}{2}\PP\big(\by_{1}=\by_{2}\big)\geq \frac{1}{2}(1-\gamma)^n\geq \xi,
\#
where we suppress the dependence of $\widehat\bbeta_{k}$ on $\by_{k}$ for simplicity.
 Using the fact that $c\gamma\geq v_\delta^{1/(1+\delta)} (\gamma/2)^{\delta/(1+\delta)}$ further implies
% we note that
%\$
%1\geq  (1-\gamma )^{1/ (1+\delta)}\bigg\{\bigg(\frac{\gamma}{2}\bigg)^{\delta}+ \bigg(\frac{1-\gamma}{2}\bigg)^{\delta} \bigg\}^{1/ (1+\delta)},
%\$
%or equivalently,
%$
%c\gamma\geq v_\delta^{1/(1+\delta)}({\gamma}/{2})^{ \delta/ (1+\delta)}.
%$
%This further implies that
\$
& \PP\Bigg[  \big\|\hat\bbeta_{1}-\bbeta_1^*\big\|_2 \geq  v_\delta^{1/(1+\delta)} \frac{   d^{1/2} \alpha }{ c_u}\bigg\{ \frac{ \log(1/(2\xi))}{2n}\bigg\}^{ \delta/ (1+\delta )} \Bigg] \nn \\
& \qquad  \bigvee  \PP\Bigg[  \big\| \hat\bbeta_{2}-\bbeta_2^* \big\|_2\geq  v_\delta^{1/(1+\delta)}  \frac{  d^{1/2} \alpha }{ c_u} \bigg\{ \frac{ \log(1/(2\xi))}{2n}\bigg\}^{ \delta/ (1+\delta )} \Bigg]
%&%\geq \max\big\{\PP\big(\|\hat\bbeta_{1\delta}\!-\!\Xb\bbeta^*\|_2\geq c\gamma d^{1/2}   \big), \PP\big(\|\Xb\hat\bbeta_{2\delta}-\Xb\bbeta^*\|_2\geq c\gamma d^{1/2} \big)\big\}
\geq \xi.
\$
%where we used the inequality $\|\Xb\hat\bbeta_{1\delta}-\Xb\bbeta^*\|_2\leq \sqrt{c_u}\|\hat\bbeta_{1\delta}-\bbeta^*\|_2$ in the first step.
Now since $\cP_{c,\gamma}\subseteq \cP_{\delta}^{v_\delta}$, taking $\log\{ 1/(2\xi)\}={2 t}$ implies the result for the case where $\delta\in (0,1]$.   When $\delta>1$, the second moment exists, and therefore using the fact that $v_1\!<\!\infty$ completes the proof.        \qed

%%%definitions
\newcommand{\range}[1]{\textnormal{range}\left(#1\right)}
\newcommand{\rank}[1]{\textnormal{rank}\left(#1\right)}

%%%%

%\comment{Q: we now outline a different proof, which eliminates the extra assumption there.}
%\begin{proof}[Proof of Theorem 3.2]
%Consider the standard linear model
%\$
%y=X\beta+\varepsilon,
%\$
%where $\varepsilon=(\varepsilon_1,\ldots, \varepsilon_n)^\T$, $\EE\varepsilon_i=0$ and $\EE|\varepsilon|^{1+\delta}<\infty$.
%For a tolerance model $\delta>0$ to be chosen, consider the set
%\$
%\left\{\gamma\in \range{X\beta/\sqrt{n}}: \|\gamma\|_2\leq 4\delta \right\}.
%\$
%Let $\{\gamma^1,\ldots, \gamma^M\}$ be a $2\delta$-packing in the $\ell_2$-norm. By the volume ratio lemma  and the relationship between covering and packing, we know that we can find such a packing with $M\geq d\log 2$, since $X$ is nonsingular: $\rank X=d$.  Now we consider
%\$
%y^j=X\beta^j+\varepsilon^j=\sqrt{n}\gamma^j +\varepsilon^j.
%\$
%
%Without loss of generality, we assume that $X$ is orthonormal, that is $X^\T X/n=I $. In other words,  we can write
%\$
%\gamma=a_1e_1+\ldots +a_de_d.
%\$
%
%%\comment{Q: current point. }

%\end{proof}

\subsection{Proof of Theorem \ref{thm:hd}}\label{appendix:A.7}

We start with the proof of Lemma \ref{lemma:hd:lm1}.

\begin{proof}[Proof of Lemma \ref{lemma:hd:lm1}]
Let $\Hb_\tau=\nabla^2\cL_\tau(\bbeta),$ where we suppress the dependence on $\bbeta$. Then for any $( \bu  ,\bbeta  ) \in \mathcal{C}(k,\gamma, r)$, we have
\#\label{thm:hd:lm1:eq1}
 \langle \bu,  \Hb_\tau\bu \rangle  & =  \bu^\T \bigg\{\frac{1}{n}\sum_{i=1}^n \bx_i \bx_i^\T 1\big(|y_i- \langle  \bx , \bbeta_i \rangle  |\leq \tau\big)\bigg\}\bu   \nn \\
& \geq  \big\| \mathbf{S}_n^{1/2} \bu  \big \|_2^2  - \frac{1}{n} \sum_{i=1}^n \langle \bu, \bx_i\rangle^2  1\big(| \langle \bx_i,\bbeta  - \bbeta^* \rangle | \!\geq\! \tau/2 ) \!-\! \frac{1}{n} \sn \langle \bu, \bx_i\rangle^2 1\big( | \varepsilon_i |\! >\! \tau/2\big)   \nn \\
& \geq   \big\| \mathbf{S}_n^{1/2} \bu  \big\|_2^2 \!-\!  { \frac{2 r  }{\tau}   \max_{1\leq i\leq n} \| \bx_i \|_\infty\| \mathbf{S}_n^{1/2} \bu   \|_2^2  } \!-\!  \max_{1\leq i\leq n} \langle \bu, \bx_i\rangle^2 \, \frac{1}{n} \sn   1\big(| \varepsilon_i | \!>\! \tau/2 \big) .
\#
As $\|\bx_i\|_\infty\leq L$ for any $1\leq i\leq n$, we have
\$
|  \langle \bu, \bx_i\rangle  |  \leq \| \bx_i \|_\infty \| \bu  \|_1 \leq (1+\gamma) \| \bx_i \|_\infty \| \bu _J \|_1  \leq  L k^{1/2} (1 +\gamma ) .
\$
 Moreover, for any $t\geq 0$, applying Hoeffding's inequality yields that, with probability at least $1- e^{- t}$,
\$
  \frac{1}{n} \sn   1\big( | \varepsilon_i | > \tau/2 \big) \leq  \bigg( \frac{2}{\tau} \bigg)^{1+\delta}  \frac{1}{ n}   \sum_{i=1}^n v_{i,\delta}  +  \sqrt{\frac{t}{2n}} = \bigg( \frac{2}{\tau} \bigg)^{1+\delta}    v_{\delta} + \sqrt{\frac{t}{2n}} .
\$
Putting the above calculations together, we obtain  %under the condition $\max_{1\leq i\leq n} \|\bx_i \|_\infty \leq 1$,
\$
	\langle \bu,  \Hb_\tau\bu \rangle
&\geq  \big\| \mathbf{S}_n^{1/2} \bu   \big\|_2^2 -  2  \tau^{-1} rL   \big\| \mathbf{S}_n^{1/2} \bu  \big\|_2^2 -  k(1+\gamma)^2  L^2\big( 2^{1+\delta} v_\delta \tau^{ - 1-\delta } +   \sqrt{t/2} \, n^{-1/2} \big).
\$
Consequently, as long as $\tau \geq  8 L r$, the following inequality
\#
	\langle \bu,  \Hb_\tau\bu \rangle   \geq \frac{3}{4}  \kappa_l - k(1+\gamma)^2 L^2 \big(  2^{1+\delta} v_\delta \tau^{- 1-\delta } +  \sqrt{t/2} \, n^{-1/2}  \big)   \geq \frac{1}{2} \kappa_l , \label{lre.Ltau}
\#
holds uniformly over $( \bu  ,\bbeta  ) \in \mathcal{C}(k,\gamma,r)$ with probability at least $1-e^{-t}$, where the last inequality in \eqref{lre.Ltau} holds whenever $\tau \gtrsim (1+\gamma)^{2/(1+\delta)}\kappa_l^{-1/(1+\delta)} ( L^2k v_\delta )^{1/(1+\delta)} $ and $n\gtrsim (1+\gamma)^4 \kappa_l^{-2}L^4k^2t$. On the other side, it can be easily shown that  $\langle \bu,  \Hb_\tau\bu \rangle \leq \kappa_u$. This completes the proof of the lemma.
\end{proof}

The following lemma is taken from \cite{fan2015tac} with slight modification, which shows that the solution $\widehat\bbeta = \hat{\bbeta}_{\tau, \lambda}$ falls in a $\ell_1$-cone.

\begin{lemma}[\sf $\ell_1$-cone Property]\label{lemma:1cone}
For any  $\cE$ such that $\cS\subseteq \cE$,  if $\|\nabla\cL_\tau(\bttc)\|_\infty\leq \lambda/2$, then
$\|(\hbt-\bttc)_{\cE^c}\|_1\leq %\frac{\|\blam\|_\infty+\|\nabla\cL(\bttc)\|_\infty+\varepsilon}{\|\blam_{\cE}\|_{\min}-(\|\nabla\cL(\bttc)\|_\infty+\varepsilon)}
3\|{(\hbt-\bttc)}_{\cE}\|_1.
$
\end{lemma}
Now we are ready to prove the theorem.
\begin{proof}[Proof of Theorem \ref{thm:hd}]
It suffices to prove the statement for $\delta\in (0, 1]$.
We start by constructing an intermediate estimator $\hat\bbeta_{  \eta} =\bbeta^*+ \eta (\hat\bbeta-\bbeta^*)$ such that $\|\hat\bbeta_{ \eta} -\bbeta^*\|_1\leq r$ for some $r>0$ to be specified. We take $\eta =1$ if $\|\hat\bbeta-\bbeta^*\|_1\leq r$, and choose $\eta \in(0,1)$ so that  $\|\hat\bbeta_{ \eta}  -\bbeta^*\|_1=r$ otherwise. Lemma \ref{lemma:1cone}, $\widehat\bbeta_{  \eta}$ also falls in a $\ell_1$-cone:
\#\label{thm:3.7:eq:1}
\|(\widehat\bbeta_{ \eta}-\bbeta^*)_{\cS^c}\|_1\leq 3\|(\widehat\bbeta_{  \eta}-\bbeta^*)_{\cS}\|_1.
\#
Under Condition~\ref{con:re}, it follows from Lemma \ref{lemma:hd:lm1} that with probability at least $1-e^{-t}$,
%Using \eqref{lre.Ltau}, and exploiting the localized analysis proposed in \cite{fan2015tac}, we replace $\tbs$ there with $\widehat\bbeta$ and obtain
\$
 \frac{ \kappa_l }{2} \|\widehat\bbeta_{  \eta}-\bbeta^*\|_2^2\leq \big\langle\nabla\cL_\tau(\widehat \bbeta_{  \eta})-\nabla\cL_\tau(\bbeta^*),\widehat\bbeta_{\eta}-\bbeta^* \big\rangle
\$
as long as $\tau \gtrsim \max\{ (L^2k v_\delta )^{1/(1+\delta)} , L r\}$ and $n\gtrsim L^4k^2t$. Applying Lemma \ref{lm01} and following the same calculations as in Lemma \red{B.7} of \cite{fan2015tac}, we obtain
\$
  \frac{ \kappa_l }{2} \|\widehat\bbeta_{  \eta}-\bbeta^*\|_2^2\leq \big\{ s^{1/2} \lambda +\|\nabla\cL_\tau(\bbeta^*)_\cS\|_2\big\} \|(\widehat\bbeta_{ \eta}-\bbeta^*)_\cS\|_2,
\$
which, combined with $\|\nabla \cL_\tau(\bbeta^*)_\cS \|_\infty\leq  \lambda/2$,   implies that
\#\label{thm:hd:1}
\|\widehat\bbeta_{  \eta}-\bbeta^*\|_2\leq 3 \kappa_l^{-1}  s^{1/2} \lambda . %\big\{ \lambda\vee\|\nabla\cL_\tau(\bbeta^*)_\cS\|_\infty\big\}
\#
Inequalities in \eqref{thm:3.7:eq:1} imply that $\|\widehat\bbeta_{ \eta}-\bbeta^*\|_1\leq 4\|(\widehat\bbeta_{ \eta}-\bbeta^*)_\cS\|_1\leq 4 s^{1/2} \|\widehat\bbeta_{\eta}-\bbeta^*\|_2\leq 12\kappa_l^{-1}  s \lambda  <r$. By the construction of $\widehat\bbeta_{ \eta}$, we conclude that $\widehat\bbeta_{\eta}=\widehat\bbeta$, and thus the stated result holds.
It remains to bound the probability that event $ \{\|\nabla\cL_\tau(\bbeta^*)_\cS\|_\infty\leq \lambda/2  \}$ occurs. Recall the gradient of $\cL_\tau$ evaluated at $\bbeta^*$, i.e. $\nabla\cL_\tau(\bbeta^*)= -  n^{-1} \sum_{i=1}^n  \psi_\tau(\varepsilon_i)\bx_i$.
Following the same argument used in the proof of Theorem \ref{thm:ld}, we take $\tau = \tau_0 (  n / t)^{1/(1+\delta)}$ for some $\tau_0\geq \nu_\delta$ and reach
$\PP \{ \|    \nabla\cL_\tau(\bbeta^*)_{\mathcal{S}} \|_{\infty}  \geq  2L \tau  n^{-1} t   \} \leq 2  s e^{- t}$. This, together with \eqref{thm:hd:1}, proves \eqref{hd.bound.general}.

Finally, taking $t = (1+c)\log d$ for some $c>0$ yields that with probability at least $1- (2s+1)d^{-1-c}$,
$  \|\nabla\cL_\tau(\bbeta^*)_{\cS} \|_{\infty}  \leq     2L\tau_0   \{ (1+c)(\log d) / n\}^{\delta/(1+\delta )}$.
As implied by Condition~\ref{con:re} with $k=2s$, we have $2s+1 \leq d$ and thus \eqref{hd.bound.special} follows immediately.
\end{proof}

\subsection{Proof of Theorem \ref{thm:hd:mini} }
The proof of this theorem follows the similar argument to that of Theorem \ref{thm:ld:mini}.
It suffices to prove the result for $\delta\in(0,1]$. Similar to the proof of Theorem \ref{thm:ld:mini},
We start by defining a simple  class of distributions for the response variable $y$ as $\cP_{c,\gamma}= \{\PP_{c+}, \PP_{c-} \}$, where
\$
\PP^+_c(\{0\})=1-\gamma,~\PP^+_c(\{c\})=\gamma,~\textnormal{and}~\PP^-_c(\{0\})=1-\gamma,~\PP^-_c(\{-c\})=\gamma.
\$
Here, we suppress the dependence of $\PP^+_c$ and $\PP^-_c$  on  $\gamma$ for convenience. It follows that, for any $0<\delta\leq 1$, the $(1+\delta)$-th absolute central moment $v_\delta$ of $y$ with law either $\PP_c^+$ or $\PP_c^-$ is
\#\label{thm:min:hd:eq:1}
v_\delta=|c|^{1+\delta}\gamma(1-\gamma) \{ \gamma^\delta+(1-\gamma)^{\delta} \}.
\#
%Let $\cH_+(s)\!=\!\big\{\zb \in  \{ 0, 1\}^d : \|\zb\|_0\!=\!s\big\}$ and $\cH_{-}(s)\!=\!\{-\zb:\zb\!\in\!\cH_+(s)\}$.
%Using Lemma \ref{lemma:packing}, there exists  subsets $\tilde\cH_+,~\tilde\cH_-$, each with cardinality $M\geq  1\vee\exp[2s\log \{(d-2s)/s\} ] $, such that  $\rho(\zb, \zb')\geq s$, for all $\zb,\zb'\in \tilde\cH_+\cup\tilde\cH_-$. Here $\rho$ is the Hamming distance defined as in Lemma \ref{lemma:packing}.
We define  the following $s$-sparse sign-ball  $\cU_n$ as
\$
\cU_n=\big\{\ub: \ub\in \{-1,1\}^n \big\}.
\$
By assumption, there exist $\ub\in \cU_n$ and $\cA$ with $|\cA|=s$ such that $\|\Xb_\cA^\T \ub\|_{\min}/n\geq \alpha$. %Take $\bbeta_1^*=\eta \ub$ and $\bbeta_2^*=-\eta\ub$.
Take $\bbeta^*_1,\,\bbeta^*_2$ supported on $\cA$ and $\bc\in\RR^n$ such that $\bc=c\ub$ for a $c>0$, $\Xb\bbeta_1^*=\bc\gamma$ and $\bbeta_2^*=-\bbeta_1^*$.
Let $\PP^+$ be the distribution of $\by_1=\Xb\bbeta^*_1+\bvarepsilon$ and $\PP_-$ that of  $\by_2=\Xb\bbeta^*_2+\bvarepsilon$. Clearly, we have %where $y_{1i}$ follows $\PP^+_{c_i}$ with $c_i=cu_i$, and $y_{2i}\sim\PP^-_{c_{i}}$. Clearly, we have
\$
\EE( \varepsilon_i ) =0~~\textnormal{and}~~\EE( |\varepsilon_i|^{1+\delta} )=c^{1+\delta}\gamma(1-\gamma) \{ \gamma^\delta+(1-\gamma)^{\delta} \}.
\$
%Let $\BB_0(s)\!\equiv\! \big\{\bbeta \in \RR^d : \|\bbeta\|_0=s \big\}$ be the $s$-sparse set in $d$-dimensional space.
%Let $\bc=(c_1,\ldots, c_n)^\T$, which implies that $\Xb\bbeta_1^*=\bc\gamma$ and $\Xb\bbeta_2^*=-\bc\gamma$.
Let $\cA$ be the support of $\bbeta_1^*$. Then, we have
\begin{gather*}
(\bbeta_1^*)_\cA=c\gamma\Big(\frac{1}{n}\Xb_\cA^\T \Xb_\cA\Big)^{-1}\frac{1}{n}\Xb_\cA^\T \ub,~\textnormal{and}\\
\|\bbeta_1^*\|_2\geq c\gamma  \, \kappa_u^{-1} s^{1/2} \|\Xb_\cA^\T \ub/n\|_{\min}  \geq c\gamma \,\kappa_u^{-1} s^{1/2} \alpha .
\end{gather*}
Let $\widehat\bbeta_k(\by_k)$ be any $s$-sparse estimator. With the above setup, we have
\#\label{thm:min:hd:eq:2}
&\max\Big\{\PP\big(\big\|\hat\bbeta_{1}-\bbeta_1^*\big\|_2\geq c\gamma  \, \kappa_u^{-1} s^{1/2} \alpha  \big),\, \PP\big(\big\|\hat\bbeta_{2}-\bbeta_2^*\big\|_2\geq c\gamma \, \kappa_u^{-1} s^{1/2} \alpha  \big)\Big\}\notag\\
&\geq  \frac{1}{2} \PP\Big( \big\|\hat\bbeta_{1}-\bbeta_1^*\big\|^2_{2}\geq c\gamma \, \kappa_u^{-1} s^{1/2} \alpha ~\textnormal{or}~\big\|\hat\bbeta_{2}-\bbeta_2^*\big\|_2\geq c\gamma \, \kappa_u^{-1}  s^{1/2} \alpha \Big)\notag\\
&\geq \frac{1}{2}\PP\big(\hat\bbeta_{1}=\hat\bbeta_{2}\big)\geq \frac{1}{2}\PP\big(\by_{1}=\by_{2}={\bf 0}\big)%\geq \frac{1}{2}(1-\gamma)^n\geq t,
\#
where  we suppress the dependence of $\widehat\bbeta_{k}$ on $\by_{k}$ for simplicity. For the last quantity in the displayed inequality above, taking $\gamma= \log \{1/(2t) \}  / (2n)$ with $t\geq  e^{-n}/2$, we obtain $1-\gamma\geq 1/2$ and
\$
\PP\big(\by_{1}=\by_{2}={\bf 0} \big) = (1-\gamma)^n\geq \Big\{\exp\Big(\frac{-\gamma}{1-\gamma}\Big)\Big\}^n\geq 2t.
\$
Using the fact that
%$|c_i|\gamma\geq v_{i, \delta}^{1/(1+\delta)} (\gamma/2)^{\delta/(1+\delta)}$, we obtain
%\$
%1\geq  (1-\gamma )^{1/ (1+\delta)}\bigg\{\bigg(\frac{\gamma}{2}\bigg)^{\delta}+ \bigg(\frac{1-\gamma}{2}\bigg)^{\delta} \bigg\}^{1/ (1+\delta)},
%\$
%or equivalently,
$
c\gamma\geq v_\delta^{1/(1+\delta)}({\gamma}/{2})^{ \delta/ (1+\delta)},
$
this further implies
\$
& \PP\Bigg[  \big\|\hat\bbeta_{1}-\bbeta_1^*\big\|_2 \geq  v_\delta^{1/(1+\delta)}  \kappa_u^{-1} \alpha s^{1/2} \bigg\{ \frac{ \log\{1/(2t)\}}{2n}\bigg\}^{ \delta/ (1+\delta )} \Bigg] \nn \\
& \qquad  \bigvee  \PP\Bigg[  \big\| \hat\bbeta_{2}-\bbeta_2^* \big\|_2\geq  v_\delta^{1/(1+\delta)}  \kappa_u^{-1} \alpha s^{1/2}  \bigg\{ \frac{ \log\{1/(2t)\}}{2n}\bigg\}^{ \delta/ (1+\delta )} \Bigg]
%&%\geq \max\big\{\PP\big(\|\hat\bbeta_{1\delta}\!-\!\Xb\bbeta^*\|_2\geq c\gamma\sqrt{d} \, \big), \PP\big(\|\Xb\hat\bbeta_{2\delta}-\Xb\bbeta^*\|_2\geq c\gamma\sqrt{d} \,\big)\big\}
\geq t.
\$
%where we used the inequality $\|\Xb\hat\bbeta_{1\delta}-\Xb\bbeta^*\|_2\leq \sqrt{c_u}\|\hat\bbeta_{1\delta}-\bbeta^*\|_2$ in the first step.
Now since $\cP_{c,\gamma}\subseteq \cP_{\delta}^{v_\delta}$, taking $t=d^{-A}/2$ implies the result for the case where $\delta\in (0,1]$.   When $\delta>1$, the second moment exists. Thus using $v_1 < \infty$ completes the proof.     \qed

\subsection{Proof of Theorem~\ref{robust.hdreg.dev.ineq}}

The proof is almost identical to that of Theorem \ref{thm:hd}. We only need to derive a probability bound for the event $\{ \|  \bxi^*_{\cS} \|_\infty\leq \lambda/2 \}$ under the assumed scaling and moment conditions, where $\bxi^* := \nabla \cL^{\varpi}_\tau(\bbeta^*)$.

%First, we deal with the case $\delta \in (0,1)$ using arguments similar to those employed in the proof of Theorem~\ref{thm:ld}.
Recall that $\bx^{\varpi}_i = ( x^{\varpi}_{i1} , \ldots,  x^{\varpi}_{id})^\T$ with $ x_{ij}^{\varpi} = \psi_\varpi(x_{ij})$ for $i=1,\ldots ,n$ and $j=1,\ldots, d$. Define $\bz_i = (z_{i1},\ldots, z_{id})^\T =\bx_i-  \bx_i^{\varpi}$, where $ z_{ij}=\{  x_{ij} -\varpi\sgn(x_{ij}) \}1(|x_{ij}|>\varpi)$. Moreover, write $\bz_{i\cS}= (z_{ij}1(j\in \cS))\in \RR^d$ and $\epsilon_i=\varepsilon_i+ \langle \bz_i,  \bbeta^* \rangle$. In this notation, we have $\bxi^* =  - n^{-1}\sn \psi_\tau(\epsilon_i )  {\bx}^{\varpi}_i$. From the identity $\EE\{ \psi_\tau(\epsilon_i)  {x}^{\varpi}_{ij} \} = \EE \{ \langle \bz_i , \bbeta^* \rangle {x}^{\varpi}_{ij}  \}   - \EE \{ \epsilon_i - \tau \sgn(\epsilon_i)\}   {x}^{\varpi}_{ij} 1(|\epsilon_i | > \tau)$, we see that
\#
	&  | \EE\{ \psi_\tau ( \epsilon_i ) {x}^{\varpi}_{ij} \}  | \nn \\
    & \leq     M_4 \| \bbeta^* \|_1    \varpi^{-2} + \tau^{-2} \EE(  | \epsilon_i |^3 |  {x}^{\varpi}_{ij} | ) \nn \\
	& \leq  M_4 \| \bbeta^* \|_1    \varpi^{-2} + 4  \tau^{-2}  \big\{ \EE ( |\varepsilon_i |^3 | {x}^{\varpi}_{ij} | )  + \| \bbeta^* \|_2^3 \, \EE(  \| \bz_i \|_2^3 | {x}^{\varpi}_{ij} | ) \big\} \nn \\
	& \leq  M_4 \| \bbeta^* \|_2 \, s^{1/2}   \varpi^{-2} + 4  \tau^{-2}  \big\{ v_2  M_2^{1/2}  +  M_4 \| \bbeta^* \|_2^3 \,  s^{3/2} \big\} . \nn
\#
Then it holds
\#
	 \|  \EE ( \bxi^* ) \|_\infty     \leq    M_4 \| \bbeta^* \|_2 \, s^{1/2}   \varpi^{-2} + 4  \tau^{-2}  \big\{ v_2  M_2^{1/2}  +  M_4 \| \bbeta^* \|_2^3 \,  s^{3/2} \big\}. \label{mean.xi*.bound}
\#
For each $j$ fixed, note that
\#
	\sn \EE \{  {x}^{\varpi}_{ij}  \psi_\tau(\epsilon_i) \}^2 \leq \sn \EE \{   x_{ij}^2  (\varepsilon_i^2 + \langle \bz_i, \bbeta^* \rangle^2 ) \} \leq  n\big(\sigma^2 M_2  +  M_4 \|\bbeta^* \|_2^2 \, s   \big)  ,\nn\\
 \mbox{ and }~ \sn	\EE |  {x}^{\varpi}_{ij}  \psi_\tau(\epsilon_i) |^k \leq \frac{k!}{2} (\varpi \tau /2)^{k-2}  n\big(\sigma^2 M_2  +  M_4 \|\bbeta^* \|_2^2 \, s   \big)~\mbox{ for all } k\geq 3.\nn
\#
Applying Bernstein's inequality gives
\#
	 \bigg| \frac{1}{n} \sn \big[  {x}^{\varpi}_{ij} \psi_\tau (\epsilon_i ) -   \EE \{  {x}^{\varpi}_{ij} \psi_\tau (\epsilon_i )  \} \big] \bigg| \leq    \big(\sigma^2 M_2  +  M_4 \|\bbeta^* \|_2^2 \, s   \big)^{1/2} \sqrt{\frac{2t}{n}}  +  \varpi \tau\frac{ t }{2n} \nn
\#
with probability at least $1- 2e^{-t}$. Taking the union bound over $j \in \cS$, we obtain that, with probability at least $1- 2s e^{-t}$,
\#
 \| \bxi^*_{\cS} -\EE(\bxi^*_{\cS})  \|_{\infty} \leq \big(2\sigma^2 M_2  +  2M_4 \|\bbeta^* \|_2^2 \, s   \big)^{1/2} \sqrt{\frac{t}{n}}  +  \varpi \tau\frac{ t }{2n} .\nn
\#
This, together with \eqref{mean.xi*.bound}, implies that $\PP\{\mathcal{E}(\tau, \varpi, \lambda)\} \geq 1 - 2s e^{-t}$ provided
\#
 \lambda & \geq  2 M_4 \| \bbeta^* \|_2 \, s^{1/2}   \varpi^{-2} + 8  \big\{ v_2  M_2^{1/2}  +  M_4 \| \bbeta^* \|_2^3 \,  s^{3/2} \big\}  \tau^{-2}  \nn \\
 & \quad + 2\big(2\sigma^2 M_2  +  2M_4 \|\bbeta^* \|_2^2 \, s   \big)^{1/2} \sqrt{\frac{t}{n}}  +  \varpi \tau\frac{ t }{n}. \nn
\#
This is the stated result. \qed

\subsection{Proof of Theorem~\ref{thm:A1}}

To begin with, define the parameter set $\Theta_0(r)    =  \{ \bbeta \in \RR^d :    \|  \bbeta - \bbeta ^*  \|_{ \bSigma,2} \leq r \}$ for some $r>0$ to be specified, and let $\hat \bbeta_{\tau, \eta} \in \Theta_0(r)$ be the intermediate estimator introduced in the proof of Theorem~\ref{thm:ld}.

\noindent
{\sc Proof of \eqref{CI.type2}}. In view of \eqref{localized.analysis} and \eqref{pos1}, lying in the heart of the arguments is to derive  deviation inequalities for $\| \bSigma^{-1/2} \nabla\cL_\tau(\bbeta^*)\|_2$ under the moment condition that $v_\delta <\infty$ for some $0<\delta \leq 1$, and to establish the restricted strong convexity for the Huber loss $\cL_\tau$, i.e. there exists some $\kappa>0$ such that
$$
	 \langle \nabla\cL_\tau( \bbeta )-\nabla\cL_\tau(\bbeta^*), \bbeta  -\bbeta^* \rangle \geq \kappa \| \bbeta - \bbeta^* \|_2^2
$$
holds uniformly over $\bbeta $ in a neighborhood of $\bbeta^*$.

First, from \eqref{r0.def} in Lemma~\ref{lm:grad.l2} we see that
\#
 \big\| \bSigma^{-1/2}  \nabla \cL_\tau(\bbeta^* ) \big\|_2 < r_0 :=    4 \sqrt{2} A_0 v_\delta^{1/2} \tau^{(1-\delta)/2} \sqrt{\frac{d+t}{n}}  + 2 A_0 \tau \frac{d+t}{n} +   \frac{v_\delta }{\tau^\delta} \nn
\#
with probability at least $1- e^{-t}$. Next, since $\hat \bbeta_{\tau, \eta} \in \Theta_0(r)$ and according to Lemma~\ref{lm02}, we take $r = \tau/(4A_1^2)$ such that under the scaling \eqref{RSC.scaling},
$$
 \big\langle \nabla\cL_\tau( \hat \bbeta_{\tau, \eta} )-\nabla\cL_\tau(\bbeta^*),\hat \bbeta_{\tau, \eta}  -\bbeta^* \big\rangle \geq \frac{1}{4}  \big\|  \hat \bbeta_{\tau, \eta} - \bbeta \big\|_{ \bSigma,2}^2
$$
with probability at least $1- e^{-t}$. Together, the last two displays and \eqref{localized.analysis} imply that with probability at least $1- 2e^{-t}$,
$\|   \hat{\bbeta}_{\tau, \eta} - \bbeta^*   \|_{ \bSigma,2} \leq 4r_0 < r$ provided $n \geq C_1   (d+t)$, where $C_1 >0$ is a constant depending only on $ A_0$. Following the same arguments as we used in the proof of Theorem~\ref{thm:ld}, this proves \eqref{CI.type2}.

\medskip
\noindent
{\sc Proof of \eqref{BR}}.
From the preceding proof, we see that
\#
	\PP\big\{   \hat{\bbeta} \in \Theta_0(r_1) \big\} \geq 1- 2 e^{-t}   \label{local.ball.rd}
\#
as long as $ n\geq C_1  (d+t)$, where $r_1=4r_0$. Moreover, define random processes $\bzeta(\bbeta) = \cL_\tau(\bbeta) - \EE \cL_\tau(\bbeta)$ and
\#
	\bB(\bbeta) = \bSigma^{-1/2} \big\{ \nabla \cL_\tau(\bbeta) - \nabla \cL_\tau(\bbeta^*) \big\} -  \bSigma^{1/2} (\bbeta - \bbeta^*).  \label{def.B}
\#
To bound $\| \bB(\hat \bbeta_\tau )  \|_2 = \| \bSigma^{1/2} (\hat \bbeta_\tau - \bbeta^*) +  \bSigma^{-1/2} \nabla \cL_\tau(\bbeta^*) \|_2$, the key is to bound the supremum of the empirical process $\{ \bB(\bbeta): \bbeta \in \Theta_0(r) \}$. To that end, we deal with $\bB(\bbeta) - \EE\{\bB(\bbeta ) \}$ and $\EE \{ \bB(\bbeta) \}$ separately, starting with the latter. By the mean value theorem,
\#
	\EE\{\bB(\bbeta ) \} & =  \bSigma^{-1/2} \big\{ \nabla \EE \cL_\tau(\bbeta) - \nabla \EE \cL_\tau(\bbeta^*) \big\}  -  \bSigma^{1/2} (\bbeta - \bbeta^*) \nn \\
	& = \big\{  \bSigma^{-1/2}  \nabla^2 \EE \cL_\tau(\wt \bbeta)  \bSigma^{-1/2}  - \Ib_d \big\} \bSigma^{1/2} (\bbeta - \bbeta^*) , \nn
\#
where $\wt \bbeta $ is a convex combination of $\bbeta $ and $\bbeta^*$. Therefore,
\#
	\sup_{\bbeta \in \Theta_0(r)} \big\|  \EE\{\bB(\bbeta ) \}  \big\|_2  \leq r\times \sup_{\bbeta \in \Theta_0(r)} \big\|  \bSigma^{-1/2}  \nabla^2 \EE \cL_\tau( \bbeta)  \bSigma^{-1/2}  - \Ib_d \big\| . \nn
\#
For $\bbeta\in   \Theta_0(r)$ and $\bu  \in \mathbb{S}^{d-1}$, write $\bdelta = \bSigma^{1/2}(\bbeta - \bbeta^*)$ such that $\| \bdelta \|_2 \leq r$. Let $A_1>0$ be the constant in Lemma~\ref{lm02} that scales as $A_0$.
It follows that
\begin{align}
	& \big| \bu^\T \big\{  \bSigma^{-1/2}  \nabla \EE \cL_\tau(\bbeta) \bSigma^{-1/2} - \Ib_{d} \big\} \bu  \big|  =   \frac{1}{n} \sn \e\big\{ 1\big( | y_i - \langle \bx_i , \bbeta \rangle | > \tau \big) \langle \bu,  \wt \bx_i \rangle^2 \big\}    \nn \\
	& \leq  \frac{1}{n\tau^2 } \sn  \big\{    v_{i,1}  +\e \langle \bdelta ,  \wt \bx_i \rangle^2  \langle \bu,  \wt \bx_i \rangle^2 \big\}     \leq  v_1 \tau^{-2}  +   A_1^4  \tau^{-2}\| \bdelta \|_2^2  \leq   v_1 \tau^{-2} +  A_1^4   r^2  \tau^{-2} , \nn
\end{align}
which, further implies
\#
	\sup_{\bbeta \in \Theta_0(r)} \big\| \EE \{ \bB(\bbeta) \} \big\|_2 \leq v_1 \tau^{-2} +  A_1^4  r^2  \tau^{-2} . \label{mean.B.bound.rd}
\#

Next, we consider $\bB(\bbeta) - \EE\{\bB(\bbeta ) \} =  \bSigma^{-1/2} \{ \nabla \bzeta(\bbeta) - \nabla \bzeta(\bbeta^*) \}$. With $\bdelta =  \bSigma^{1/2}(\bbeta -\bbeta^*)$, define a new process $\overline{\bB}(\bdelta) = \bB(\bbeta) - \EE \{ \bB(\bbeta ) \}$, satisfying $\overline{\bB}(\textbf{0}) = \textbf{0}$ and $\EE\{ \overline{\bB}(\bdelta)\} = \textbf{0}$. Note that, for every $\bu , \bv \in \mathbb{S}^{d-1}$ and $\lambda \in \RR$,
\#
	& \EE \exp\big\{  \lambda \sqrt{n} \, \bu^\T \nabla_{\bdelta} \overline{\bB}(\bdelta) \bv  \big\} \nn \\
	& \leq \prod_{i=1}^n  \bigg( 1 +  \frac{\lambda^2 }{ n } \e \Big[ \big\{
   \langle \bu,  \wt \bx_i \rangle^2  \langle \bv ,  \wt \bx_i \rangle^2   + \big( \e |   \langle \bu,  \wt \bx \rangle    \langle \bv,  \wt \bx  \rangle | \big)^2 \big\}  e^{    \frac{|\lambda|}{\sqrt{n}}  ( | \langle \bu,  \wt \bx_i \rangle \langle \bv,  \wt \bx_i \rangle | + \e  |  \langle \bu,  \wt \bx \rangle \langle \bv ,  \wt \bx \rangle  |  )  }  \Big]  \bigg) \nn \\
	& \leq  \prod_{i=1}^n  \bigg\{ 1 +    e^{\frac{| \lambda | }{\sqrt{n}}}  \frac{\lambda^2}{n}   \e  \big( e^{\frac{ | \lambda| }{\sqrt{n}} | \langle \bu,  \wt \bx_i \rangle \langle \bv,  \wt \bx_i \rangle | } \big) +    e^{\frac{| \lambda |}{\sqrt{n}}}  \frac{\lambda^2}{n} \e  \big(   \langle \bu,  \wt \bx_i \rangle^2 \langle \bv,  \wt \bx_i \rangle^2  e^{\frac{ | \lambda | }{\sqrt{n}} | \langle \bu,  \wt \bx_i \rangle  \langle \bv,  \wt \bx_i \rangle | } \big)  \bigg\}  \nn \\
	& \leq  \prod_{i=1}^n \bigg\{ 1 +  e^{\frac{ | \lambda| }{\sqrt{n}}}  \frac{\lambda^2}{n}    \max_{ \bw \in \mathbb{S}^{d-1} } \e \big(  e^{ \frac{ | \lambda| }{\sqrt{n}} \langle \bw,  \wt \bx  \rangle^2  } \big) +  e^{\frac{| \lambda| }{\sqrt{n}}}  \frac{\lambda^2}{n}    \max_{ \bw \in \mathbb{S}^{d-1}  } \e \big( \langle \bw,  \wt \bx  \rangle^4 e^{ \frac{| \lambda | }{\sqrt{n}} \langle \bw,  \wt \bx  \rangle^2  } \big)  \bigg\}  \nn \\
	& \leq \exp\bigg\{  e^{\frac{ | \lambda| }{\sqrt{n}}}   \lambda^2   \max_{ \bw \in \mathbb{S}^{d-1} } \e \big(  e^{ \frac{ | \lambda | }{\sqrt{n}} \langle \bw,  \wt \bx \rangle^2  } \big) +   e^{\frac{ | \lambda| }{\sqrt{n}}}  \lambda^2   \max_{ \bw \in \mathbb{S}^{d-1}  } \e \big(  \langle \bw,  \wt \bx \rangle^4 e^{ \frac{ | \lambda | }{\sqrt{n}} \langle \bw ,  \wt \bx  \rangle^2  } \big) \bigg\}. \nn
\#
Under Condition~\ref{ass:3.1}, there exist constants $C_2, C_3 >0$ depending only on $A_0$ such that, for any $ |\lambda | \leq \sqrt{n/C_2}$,
$$
\sup_{\bu,\bv \in \mathbb{S}^{d-1} } \EE \exp\big\{  \lambda \sqrt{n} \, \bu^\T \nabla_{\bdelta} \overline{\bB}(\bdelta) \bv  \big\} \leq \exp(C_3^2\lambda^2/2).
$$
With the above preparations and applying Theorem~A.3 in \cite{spokoiny2013bernstein}, we reach
\#
	\PP\Bigg\{  \sup_{\bbeta \in \Theta_0(r)} \| \bB(\bbeta) - \EE\{ \bB(\bbeta) \} \|_2 \geq 6 C_3 (8d+2t)^{1/2} r \Bigg\} \leq e^{-t} \nn
\#
as long as $n \geq C_2 (8d+2t)$. Together with \eqref{mean.B.bound.rd}, this yields
\begin{align}
	& \sup_{\bbeta \in \Theta_0(r_1)} \big\|    \bSigma^{1/2}(\bbeta - \bbeta^*) - \bSigma^{-1/2} \big\{ \nabla \cL_\tau(\bbeta) - \nabla \cL_\tau(\bbeta^*) \big\} \big\|_2  \nn \\
	& \qquad \qquad \qquad \qquad  \leq   v_1  \tau^{-2}  r_1 +  A_1^4     \tau^{-2}  r_1^3  + 6 C_3 (  8d + 2t )^{1/2} n^{-1/2} r_1  \nn
\end{align}
with probability at least $1-e^{-t}$. Combine this bound with \eqref{local.ball.rd} to obtain the stated result \eqref{BR}. \qed

\subsection{Proof of Proposition~\ref{Prop2}}

Since $\e(\varepsilon)=0$, we have $\e \{ \psi_\tau( \varepsilon ) \}  = - \e \{ ( \varepsilon - \tau ) 1 ( \varepsilon > \tau) \} + \e \{ ( - \varepsilon  - \tau ) 1(\varepsilon< -\tau) \}$. Thus, for any $2\leq q \leq 2+  \kappa$, $|\e \psi_\tau( \varepsilon ) | \leq \e \{ | \varepsilon | - \tau ) 1(|  \varepsilon |>\tau) \} \leq  \tau^{1- q } \,\e( | \varepsilon |^q )$. In particular, taking $q$ to be 2 and $2+\kappa$ proves the first conclusion. Next,  note that $\e \{ \psi_\tau^2(  \varepsilon ) \}  = \e ( \varepsilon^2)  - \{  \e \varepsilon^2 1(| \varepsilon  |>\tau) - \tau^2 \mathbb{P}(| \varepsilon |>\tau)\}$. Letting $\eta = | \varepsilon |$, we deduce that
\begin{align}
   & \e\{ \eta^2 1(\eta > \tau ) \}    = 2 \e \int_0^\infty  1(\eta>y) 1(\eta>\tau) y \, dy \nn \\
 & = 2 \mathbb{P}(\eta > \tau ) \int_0^\tau y\,dy + 2\int_\tau^\infty y \mathbb{P}(\eta>y) \, dy  = \tau^2 \mathbb{P}(\eta>\tau) + 2 \int_\tau^\infty y \mathbb{P}(\eta > y) \, dy. \nn
\end{align}
By Markov's inequality, $ \int_\tau^\infty y \mathbb{P}(\eta > y) \, dy \leq \e (\eta^{2+\kappa } )  \int_\tau^\infty y^{-1-\kappa } \, dy = \kappa^{-1}\tau^{-\kappa }\, \e ( \eta^{2+ \kappa })$. Putting the above calculations together proves the second inequality. \qed

\subsection{Proof of Theorem~\ref{hd.huber}}

For simplicity, we write $\hat \bbeta = \hat \bbeta_{\tau, \lambda}$ and assume without loss of generality that $0<\delta\leq 1$.
As in the proof of Theorem~\ref{thm:A1}, we construct an intermediate estimator $\wt  \bbeta_\eta = \bbeta^* + \eta (\hat \bbeta  - \bbeta^*)$ satisfying $\| \wt  \bbeta_\eta - \bbeta^* \|_{\bSigma, 2} \leq r$ for some $r>0$ to be specified. We take $\eta = 1$ if $\| \hat \bbeta  - \bbeta^* \|_{\bSigma, 2} \leq r$; otherwise if $\| \hat \bbeta  - \bbeta^* \|_{\bSigma, 2} > r$, there exists $\eta \in (0,1)$ such that $\| \wt \bbeta_\eta - \bbeta^* \|_{\bSigma, 2} = r$. Lemma~\ref{lm01} demonstrates that
\#
	\langle   \nabla \cL_\tau( \wt  \bbeta_\eta  )  - \nabla \cL_\tau(\bbeta^*) ,  \wt \bbeta_\eta - \bbeta^* \rangle \leq  \eta \langle   \nabla \cL_\tau( \hat \bbeta  )  - \nabla \cL_\tau(\bbeta^*) , \hat \bbeta - \bbeta^* \rangle . \label{localized.bound}
\#

Next, let $\cS \subseteq \{ 1,\ldots, d\}$ be the support of $\bbeta^*$ and  define the $\ell_1$-cone $\cC \subseteq \RR^{d}$:
$$
	\cC = \big\{ \bbeta  \in \RR^d  :  \| (\bbeta - \bbeta^*)_{\cS^{{\rm c}}} \|_1 \leq 3\| (\bbeta - \bbeta^*)_{\cS} \|_1      \big\}.
$$
We claim that
\#
	\hat  \bbeta \in \cC  ~\mbox{ on the event }~ \{   \lambda \geq 2 \| \nabla \cL_\tau(\bbeta^*) \|_{\infty}  \}, \label{cone}
\#
from which it follows
\# \label{cone.l1.bound}
	\| \hat \bdelta \|_1 =  \| \hat{ \bdelta}_{\cS} \|_1 + \|  \hat  \bdelta_{\cS^{{\rm c}}} \|_1 \leq  4 \| \hat  \bdelta_{\cS} \|_1 \leq  4\sqrt{s} \,\| \hat \bdelta \|_2,
\#
where $ \hat  \bdelta  :=  \hat  \bbeta - \bbeta^*$. To prove \eqref{cone}, first, from the optimality of $\hat{\bbeta}$ we see that
\#
	\cL_\tau(\hat \bbeta) - \cL_\tau(\bbeta^*) \leq \lambda \big( \| \bbeta^* \|_1 - \| \hat \bbeta \|_1 \big).  \label{loss.diff.1}
\#
By direct calculation, we have
\#
	\| \hat \bbeta \|_1 -  \| \bbeta^* \|_1 &  \geq \| \bbeta^*_{\cS}  +  \hat \bdelta _{\cS^{{\rm c}}} \|_1  - \| \bbeta^*_{\cS^{{\rm c}}} \|_1 - \| \hat\bdelta _{\cS} \|_1 - \big(  \| \bbeta^*_{\cS} \|_1 + \| \bbeta^*_{\cS^{{\rm c}}} \|_1 \big) \nn \\
	& \geq  \| \hat \bdelta _{\cS^{{\rm c}}} \|_1  - \| \hat \bdelta _{\cS} \|_1 . \nn
\#
Under the scaling $\lambda \geq 2 \| \nabla \cL_\tau(\bbeta^*) \|_{\infty}$, it follows from the convexity of $\cL_\tau$ and Cauchy-Schwarz inequality that
\#
	\cL_\tau(\hat \bbeta) - \cL_\tau(\bbeta^*) &  \geq \langle \nabla \cL_\tau(\bbeta^*), \hat \bdelta \rangle \geq - \| \nabla \cL_\tau(\bbeta^*) \|_\infty \| \hat \bdelta \|_1 \nn\\
	& \geq -\frac{\lambda}{2} \big(  \| \hat  \bdelta _{\cS^{{\rm c}}} \|_1  + \| \hat  \bdelta _{\cS} \|_1 \big). \label{loss.diff.2}
\#
Together, \eqref{loss.diff.1} and \eqref{loss.diff.2} imply $0\leq \frac{\lambda}{2}   (  3 \| \hat \bdelta_{\cS} \|_1 - \| \hat \bdelta_{\cS^{{\rm c}}} \|_1  )$ and thus $\hat \bbeta \in \cC$.

By necessary conditions of extrema in the convex optimization problem \eqref{regularized.huber},
\#
	  \langle \nabla \cL_\tau( \hat \bbeta )   + \lambda   \hat  \bz ,   \hat  \bbeta - \bbeta^* \rangle \leq 0 , \nn
\#
where $\hat  \bz \in \partial \| \hat \bbeta \|_1$ satisfies $\langle \hat  \bz , \bbeta^* - \hat \bbeta \rangle \leq \| \bbeta^* \|_1 - \| \hat \bbeta \|_1$. Under the scaling $\lambda \geq 2 \| \nabla \cL_\tau(\bbeta^*) \|_{\infty}$, it holds
\#
	&  \langle \nabla \cL_\tau( \hat \bbeta) - \nabla \cL_\tau(\bbeta^*) ,  \hat \bbeta - \bbeta^* \rangle  \leq  \lambda\big(  \| \bbeta^* \|_1 -   \|  \hat \bbeta \|_1 \big) +  \frac{\lambda }{2} \| \hat \bbeta - \bbeta^* \|_1  \nn \\
& \leq  \lambda \big(  \|  \hat \bdelta_{\cS} \|_1 -  \| \hat  \bdelta_{\cS^{{\rm c}}} \|_1   \big)  +  \frac{\lambda }{2} \| \hat \bbeta - \bbeta^* \|_1  \leq \frac{\lambda }{2} \big( 3 \|  \hat  \bdelta_{\cS} \|_1 -   \| \hat   \bdelta_{\cS^{{\rm c}}} \|_1  \big)  . \nn
\#
Together with \eqref{localized.bound}, this implies
\#
\langle   \nabla \cL_\tau( \wt  \bbeta_\eta  )  - \nabla \cL_\tau(\bbeta^*) , \wt \bbeta_\eta - \bbeta^* \rangle \leq \frac{ 1 }{2} \lambda \eta \big(  3 \| \hat \bdelta_{\cS} \|_1 - \| \hat  \bdelta_{\cS^{{\rm c}}} \|_1 \big)  .
 \label{basic.inequality}
\#
Moreover, we introduce $\wt \bdelta_{\eta} = \wt \bbeta_\eta - \bbeta^*$ and note that $\wt \bdelta_\eta = \eta \hat\bdelta$. By \eqref{cone}, we also have $\wt \bbeta_\eta \in \cC$ under the assumed scaling.

Let $\Omega_r$ be the event on which \eqref{RSC.bound0} holds. Then $\PP(\Omega_r^{{\rm c}}) \leq d^{-1}$ under the scaling \eqref{sample.size.scaling.1} and it holds on $\Omega_r \cap  \{   \lambda \geq 2 \| \nabla \cL_\tau(\bbeta^*) \|_{\infty}  \}$ that
\#
\langle \nabla \cL_\tau(\wt \bbeta_\eta ) - \nabla \cL_\tau(\bbeta^*) ,   \wt \bbeta_\eta  - \bbeta^* \rangle \geq  \frac{1}{4} \| \wt \bdelta_\eta \|_{\bSigma, 2}^2  \geq  \frac{1}{4} \kappa_l^{1/2} \| \wt \bdelta_\eta  \|_2 \| \wt \bdelta_\eta  \|_{\bSigma , 2} .  \nn
\#
Substituting this lower bound into \eqref{basic.inequality} yields
\#
	 \frac{1}{4}  \kappa_l^{1/2}   \| \wt \bdelta_\eta  \|_2 \| \wt \bdelta_\eta  \|_{\bSigma ,2} \leq     \frac{3}{2} \lambda \eta    \| \hat \bdelta_{\cS} \|_1   \leq \frac{3}{2} \lambda s^{1/2} \|  \eta \hat \bdelta \|_2 =  \frac{3}{2} \lambda  s^{1/2} \|   \wt \bdelta_{\eta} \|_2 . \nn
\#
Canceling $\| \wt \bdelta_\eta  \|_2$ on both sides delivers
\#
\| \wt \bdelta_\eta  \|_{\bSigma, 2} \leq  6 \kappa_l^{-1/2} s^{1/2} \lambda     ~~\mbox{ and }~~ \| \wt \bdelta_\eta  \|_1 \leq  24  \kappa_l^{-1} s \lambda  \label{intermediate.bounds}
\#
under the scaling $ \lambda \geq 2 \| \nabla \cL_\tau(\bbeta^*) \|_{\infty}$ and \eqref{sample.size.scaling.1} .

It remains to calibrate the parameters $\tau, \lambda$ and $r$. First, applying Lemma~\ref{lm:grad} with $\tau = \tau_0 (n/\log d)^{1/(1+\delta)}$, we see that
\#
	 \| \nabla \cL_\tau(\bbeta^*) \|_\infty \leq  c_1   \max_{1 \leq j\leq d} \sigma_{jj}^{1/2}    \tau_0  \bigg( \frac{\log d}{n} \bigg)^{\delta/(1+\delta)} \nn
\#
with probability at least $1- 2d^{-1}$, where $c_1 = (2\sqrt{2}+1)A_0 + 1$. We therefore choose $\lambda =c_2   \max_{1\leq j\leq d} \sigma_{ jj}^{1/2}  \tau_0  \{ (\log d)/n \}^{\delta/(1+\delta)}$ for some constant $c_2 \geq 2c_1$, such that $\lambda \geq 2 \| \nabla \cL_\tau(\bbeta^*) \|_\infty$ with probability at least $1-2 d^{-1}$. Next, according to \eqref{sample.size.scaling.1}, the restricted strong convexity \eqref{RSC.bound0} holds with $r \asymp \kappa_l^{-1/2} A_0 \max_{1\leq j\leq d} \sigma_{jj}^{1/2}  \tau \sqrt{(\log d)/n}$. Putting the above calculations together, we conclude that
\#
	\big\| \wt   \bbeta_\eta - \bbeta^* \big\|_{\bSigma, 2} \leq 6c_2  \kappa_l^{-1/2}     \max_{1 \leq j\leq d} \sigma_{jj}^{1/2}  \tau_0 \, s^{1/2}  \bigg( \frac{\log d}{n} \bigg)^{\delta/(1+\delta)} < r   \label{rescaled.error.bound}
\#
with probability at least $1-3 d^{-1}$, assuming the scaling $n\gtrsim \kappa_l^{-1}  A_0^2 A_1^4    \max_{1\leq j\leq d} \sigma_{jj} \,s \log d$. By the construction of $\wt  \bbeta_\eta$, with the same probability we must have $\eta =1$ and therefore $\hat \bbeta = \wt \bbeta_\eta$. The stated result \eqref{l1huber.bounds} then follows from \eqref{intermediate.bounds}.  \qed

\subsection{Proof of Corollary~\ref{hd.huber.corr}}
Recall that $\bx_1,\ldots, \bx_n$ are i.i.d. random vectors from a sub-Gaussian vector $\bx = ( x_1, \ldots, x_d)^\T$ with $\EE(\bx \bx^\T)=\bSigma$.
Let $\bPsi = \Xb \bSigma^{-1/2}$ be an $n\times d$ matrix whose rows are independent isotropic sub-Gaussian random vectors.
Since $\kappa_l = \lambda_{\min}(\bSigma )>0$, Definition~1 in \cite{RZ2013} holds with $s_0=s$, $k_0=3$, $A= \bSigma^{1/2}$ and $K(s_0, k_0, A) = \kappa_l^{-1/2}$. Taking $\delta=1$ in Theorem~16 of \cite{RZ2013} we obtain that, with probability at least $1 -2 d^{-1}$,
$$
	 \frac{1}{\sqrt{n}} \frac{\| \Xb (  \bbeta - \bbeta^*) \|_2 }{\| \bbeta - \bbeta^* \|_{\bSigma,2}}  =  \frac{1}{\sqrt{n}} \frac{\| \bPsi \bSigma^{1/2} (\bbeta - \bbeta^*) \|_2 }{\| \bSigma^{1/2}  ( \bbeta - \bbeta^* ) \|_{2}}  \leq 2
$$
for all $\bbeta \in \cC$ as long as $n \gtrsim \kappa_l^{-1} A_0^4 \max_{0\leq j\leq d} \sigma_{jj}  \,  s \log d$. This, together with \eqref{cone} and \eqref{rescaled.error.bound}, proves  \eqref{prediction.error}. \qed


\begin{thebibliography}{11}

\bibitem[{Alquier, Cottett and Lecu\'e(2017)}]{ACL2017}
{\sc Alquier, P.}, {\sc Cottet, V.} and {\sc Lecu\'e, G.} (2017).
Estimation bounds and sharp oracle inequalities of regularized procedures with Lipschitz loss functions.
Preprint.
Available at \href{https://arxiv.org/abs/1702.01402}{arXiv:1702.01402}.

\bibitem[Belloni and Chernozhukov(2011)]{BC2011}
	{\sc Belloni, A.} and {\sc Chernozhukov, V.} (2011).
	$\ell_1$-penalized quantile regression in high-dimensional sparse models.
	\textit{The Annals of Statistics}, \textbf{39} 82--130.
	
	
\bibitem[{Bellec et~al.(2018)Bellec, Lecu{\'e} and Tsybakov}]{bellec2016slope}
\textsc{Bellec, P.\,C.}, \textsc{Lecu{\'e}, G.} and \textsc{Tsybakov, A.\,B.}
  (2018).
\newblock Slope meets Lasso: {I}mproved oracle bounds and optimality.
\newblock \textit{The Annals of Statistics}, \textbf{46} 3603--3642.

\bibitem[Bickel, Ritov and Tsybakov(2009)]{bickel2009simultaneous}
\textsc{Bickel, P.\,J.}, \textsc{Ritov, Y.} and \textsc{Tsybakov, A.\,B.} (2009).
\newblock Simultaneous analysis of {L}asso and {D}antzig selector.
\newblock \textit{The Annals of Statistics}, \textbf{37} 1705--1732.


\bibitem[Bogdan et~al.(2015)]{bogdan2015slope}
\textsc{Bogdan, M.}, \textsc{van den~Berg, E.}, \textsc{Sabatti, C.},
  \textsc{Su, W.} and \textsc{Cand{\`e}s, E.~J.} (2015).
\newblock SLOPE--Adaptive variable selection via convex optimization.
\newblock \textit{The Annals of Applied Statistics}, \textbf{9} 1103--1140.

\bibitem[Brownlees, Joly and Lugosi(2015)]{brownlees2015empirical}
\textsc{Brownlees, C.}, \textsc{Joly, E.} and \textsc{Lugosi, G.} (2015).
\newblock Empirical risk minimization for heavy-tailed losses.
\newblock \textit{The Annals of Statistics}, \textbf{43} 2507--2536.

\bibitem[B\"uhlmann and van de Geer(2011)]{BvdG2011}
	 {\sc B\"uhlmann, P.} and {\sc van de Geer, S.} (2011).
	{\em Statistics for High-Dimensional Data: Methods, Theory and Applications.} Springer, Heidelberg.

\bibitem[{Catoni(2012)}]{catoni2012challenging}
\textsc{Catoni, O.} (2012).
\newblock Challenging the empirical mean and empirical variance: {A} deviation
  study. \textit{Annales de I'Institut Henri Poincar\'e - Probabilit\'es et Statistiques}, \textbf{48} 1148--1185.

\bibitem[{Catoni(2016)}]{catoni2016pac}
\textsc{Catoni, O.} (2016).
{PAC-Bayesian bounds for the Gram matrix and least squares regression with a random design}. Preprint.
Available at \href{https://arxiv.org/abs/1603.05229}{arXiv:1603.05229}.

\bibitem[Chen, Gao and Ren(2018)]{CGR2017}
\textsc{Chen, M.}, \textsc{Gao, C.} and \textsc{Ren, Z.} (2018).
Robust covariance and scatter matrix estimation under Huber's contamination model.
{\em The Annals of Statistics}, {\bf 46}, 1932--1960.


\bibitem[Cont(2001)]{C2001}
	{\sc Cont, R.} (2001).
	Empirical properties of asset returns: Stylized facts and statistical issues.
	{\it Quantitative Finance}, {\bf 1}, 223--236.

\bibitem[{Delaigle, Hall and Jin(2011)Delaigle, Hall and Jin}]{DHJ2011}
\textsc{Delaigle, A.}, \textsc{Hall, P.} and \textsc{Jin, J.} (2011).
\newblock {Robustness and accuracy of methods for high dimensional data
  analysis based on Student's $t$-statistic}.
\newblock \textit{Journal of the Royal Statistical Society, {\rm Series B}},
  \textbf{73} 283--301.

\bibitem[Devroye et~al.(2016)]{devroye2015sub}
\textsc{Devroye, L.}, \textsc{Lerasle, M.}, \textsc{Lugosi, G.} and
  \textsc{Oliveira, R.\,I.} (2016).
\newblock Sub-{G}aussian mean estimators.
\newblock \textit{The Annals of Statistics}, \textbf{44} 2695--2725.

\bibitem[Efron et~al.(2004)]{lars2004}
\textsc{Efron, B.}, \textsc{Hastie, T.}, \textsc{Johnstone, I.} and
  \textsc{Tibshirani, R.} (2004).
\newblock {Least angle regression}.
\newblock \textit{The Annals of Statistics}, \textbf{32} 407--499.

\bibitem[Eklund, Nichols and Knutsson(2016)]{eklund2016cluster}
\textsc{Eklund, A.}, \textsc{Nichols, T.}  and \textsc{Knutsson, H.} (2016).
	{Cluster failure: Why fMRI inferences for spatial extent have inflated false-positive rates}.
	\textit{Proceedings of the National Academy of Sciences}, \textbf{113} 7900--7905.

\bibitem[Fan, Fan and Barut(2014)]{fan2014adaptive}
\textsc{Fan, J.}, \textsc{Fan, Y.} and \textsc{Barut, E.} (2014).
\newblock Adaptive robust variable selection.
\newblock \textit{The Annals of Statistics}, {\bf 42} 324--351.

\bibitem[Fan, Li and Wang(2017)]{fan2016estimation}
\textsc{Fan, J.}, \textsc{Li, Q.} and \textsc{Wang, Y.} (2017).
\newblock Estimation of high dimensional mean regression in the absence of symmetry and light tail assumptions.
\newblock \textit{Journal of the Royal Statistical Society, {\rm Series B}}, {\bf 79} 247--265.

\bibitem[Fan and Li(2001)]{fan2001variable}
\textsc{Fan, J.} and \textsc{Li, R.} (2001).
\newblock Variable selection via nonconcave penalized likelihood and its oracle properties.
\newblock \textit{Journal of the American Statistical Association}, \textbf{96} 1348--1360.

\bibitem[Fan et~al.(2018)]{fan2015tac}
\textsc{Fan, J.}, \textsc{Liu, H.}, \textsc{Sun, Q.} and \textsc{Zhang, T.}
  (2018).
\newblock {I-LAMM} for sparse learning: Simultaneous control of algorithmic
  complexity and statistical error. {\em The Annals of Statistics},  \textbf{96} 1348--1360.

\bibitem[Fan et~al.(2016)]{fan2016robust}
\textsc{Fan, J.}, \textsc{Wang, W.} and \textsc{Zhu, Z.}
  (2016).
\newblock A shrinkage principle for heavy-tailed data:
High-dimensional robust low-rank matrix recovery. Available at \href{https://arxiv.org/abs/1603.08315}{arXiv:1603.08315}.


\bibitem[Friedline et~al.(1998)]{friedline1998differential}
{\sc Friedline, J.\,A., Garrett, S.\,H., Somji, S., Todd, J.\,H.} and {\sc Sens, D.\,A.} (1998). Differential expression of the MT-1E gene in estrogen-receptor-positive and-negative human breast cancer cell lines.
{\it The American Journal of Pathology}, \textbf{152} 23--27.

\bibitem[{Giulini(2017)}]{giulini2016robust}
\textsc{Giulini, I.} (2017).
\newblock Robust PCA and pairs of projections in a Hilbert space.
  {\it Electronic Journal of Statistics}, {\bf 11} 3903--3926.

\bibitem[Hastie, Tibshirani and Wainwright(2015)]{HTW2015}
	 {\sc Hastie, T., Tibshirani, R.} and {\sc Wainwright, M.\,J.} (2015).
	{\em Statistical Learning with Sparsity: The Lasso and Generalizations.} CRC Press.
	
\bibitem[{He and Shao(1996)}]{HeShao1996}
\textsc{He, X.} and \textsc{Shao, Q.-M.} (1996).
\newblock {A general Bahadur representation of $M$-estimators and its
  application to linear regression with nonstochastic designs}.
\newblock \textit{The Annals of Statistics}, \textbf{24} 2608--2630.

\bibitem[{He and Shao(2000)}]{HeShao2000}
\textsc{He, X.} and \textsc{Shao, Q.-M.} (2000).
\newblock {On parameters of increasing dimensions}.
\newblock \textit{Journal of Multivariate Analysis}, \textbf{73} 120--135.

%\bibitem[{Holland and Ikeda(2017)}]{HI2017}
%\textsc{Holland, M.\,J.} and {\sc Ikeda, K.} (2017).
%Robust regression using biased objectives.
%{\it Machine Learning}, \textbf{106} 1643--1679.

\bibitem[{Huber(1964)}]{Huber1964}
\textsc{Huber, P.\,J.} (1964).
\newblock {Robust estimation of a location parameter.}
\newblock \textit{The Annals of Mathematical Statistics}, \textbf{35} 73--101.

\bibitem[{Huber(1973)}]{Huber1973}
\textsc{Huber, P.\,J.} (1973).
\newblock {Robust regression: Asymptotics, conjectures and Monte Carlo}.
\newblock \textit{The Annals of Statistics}, \textbf{1} 799--821.


%\bibitem[{Huber and Ronchetti(2009)}]{huber2011robust}
%\textsc{Huber, P.\,J.}  and {\sc Ronchetti, E.\,M.} (2009).
%\newblock {\it Robust Statistics}.
%\newblock {Wiley}.


\bibitem[{Koenker(2005)}]{K2005}
\textsc{Koenker, R.} (2005).
\newblock {\it Quantile Regression}.
\newblock Cambridge University Press, New York.

%@incollection{huber2011robust,
%  title={Robust statistics},
%  author={Huber, Peter J},
%  booktitle={International Encyclopedia of Statistical Science},
%  pages={1248--1251},
%  year={2011},
%  publisher={Springer}
%}


\bibitem[{Kretzschmar(2000)}]{kretzschmar2000transforming}
{\sc Kretzschmar, M.} (2000)
Transforming growth factor-$\beta$ and breast cancer: transforming growth factor-$\beta$/Smad signaling defects and cancer.
{\it Breast Cancer Research}, \textbf{2} 107--115.

\bibitem[Landi et~al.(2014)]{landi2014genome}
{\sc Landi, A., Vermeire, J., Iannucci, V., Vanderstraeten, H., Naessens, E., Bentahir, M.} and {\sc Verhasselt, B.} (2014).
Genome-wide shRNA screening identifies host factors involved  in early endocytic events for HIV-1-induced CD4 down-regulation.
{\it Retrovirology}, \textbf{11} 118--129.


\bibitem[{Lepski(1991)}]{Lepski1991}
\textsc{Lepski, O.\,V.} (1991).
Asymptotically minimax adaptive estimation. I. Upper bounds. Optimally adaptive estimates.
{\it IEEE Transactions on Information Theory}, {\bf 36} 682--697.

\bibitem[Liu(1990)]{Liu1990}
\textsc{Liu, R.\,Y.} (1990).
\newblock On a notion of data depth based on random simplices.
\newblock \textit{The Annals of Statistics}, \textbf{18} 405--414.

\bibitem[Liu, Parelius, and Singh(1999)]{LPS1999}
\textsc{Liu, R.\,Y., Parelius, J.\,M.} and {\sc Singh, K.} (1999).
\newblock Multivariate analysis by data depth: Descriptive statistics, graphics and inference, (with discussion and a rejoinder by Liu and Singh).
\newblock \textit{The Annals of Statistics}, \textbf{27} 783--858.

\bibitem[{Loh and Wainwright(2015)}]{loh2015regularized}
\textsc{Loh, P.} and \textsc{Wainwright, M.\,J.} (2015).
Regularized $M$-estimators with nonconvexity: Statistical and algorithmic theory for local optima.
\textit{Journal of Machine Learning Research}, \textbf{16} 559--616.

\bibitem[{Mammen(1989)}]{Mammen1989}
\textsc{Mammen, E.} (1989).
	Asymptotics with increasing dimension for robust regression with
  applications to the bootstrap.
	\textit{The Annals of Statistics}, \textbf{17} 382--400.

\bibitem[{Minsker(2018)}]{minsker2016sub}
\textsc{Minsker, S.} (2018).
	Sub-Gaussian estimators of the mean of a random matrix with heavy-tailed entries.
	\textit{The Annals of Statistics}, \textbf{46} 2871--2903.

\bibitem[Mizera(2002)]{M2002}
\textsc{Mizera, I.} (2002).
\newblock On depth and deep points: A calculus.
\newblock \textit{The Annals of Statistics}, \textbf{30} 1681--1736.

\bibitem[Mizera and M\"uller(2004)]{MM2004}
\textsc{Mizera, I.} and \textsc{M\"uller, C.\,H.} (2004).
\newblock Location-scale depth.
\newblock \textit{Journal of the American Statistical Association}, \textbf{99} 949--966.

\bibitem[{Nakata et~al.(2004)}]{nakata2004serum}
{\sc Nakata, B.}, {\sc Takashima, T.}, {\sc Ogawa, Y.}, {\sc Ishikawa, T.} and {\sc Hirakawa, K.} (2004).
Serum {CYFRA} 21-1 (cytokeratin-19 fragments) is a useful tumour marker for detecting disease relapse and assessing treatment efficacy in breast cancer.
{\it British Journal of Cancer}, \textbf{91} 873--878.

\bibitem[Oh et~al.(2005)]{oh2005transcriptome}
{\sc Oh, J.\,H., Yang, J.\,O., Hahn, Y., Kim, M.\,R., Byun, S.\,S., Jeon, Y.\,J., Kim, J.\,M., Song, K.\,S., Noh, S.\,M., Kim, S.} and {\sc Yoo, H.\,S.} (2005).
Transcriptome analysis of human gastric cancer.
{\it Mammalian Genome}, \textbf{16} 942--954.


\bibitem[{Portnoy(1985)}]{Portnoy1985}
\textsc{Portnoy, S.} (1985).
\newblock {Asymptotic behavior of $M$ estimators of $p$ regression parameters
  when $p^2/n$ is large; II. Normal approximation}.
\newblock \textit{The Annals of Statistics}, \textbf{13} 1403--1417.

\bibitem[Purdom and Holmes(2005)]{PH2005}
	{\sc Purdom, E.} and {\sc Holmes, S.\,P.} (2005).
	Error distribution for gene expression data.
	{\it Statistical Applications in Genetics and Molecular Biology}, {\bf 4}: 16.


\bibitem[{Ross et~al.(2000)}]{ross2000systematic}
{\sc Ross, D.\,T.}, {\sc Scherf, U.}, {\sc Eisen, M.\,B.}, {\sc Perou, C.\,M.}, {\sc Rees, C.}, {\sc Spellman, P.}, {\sc Iyer, W.}, {\sc Jeffrey, S.\,S.},  {\sc Van de Rijn, M.},  {\sc Pergamenschikov, A.}, {\sc Lee, J.\,C.\,F.}, {\sc Lashkari, D.}, {\sc Shalon, D.}, {\sc Myers, T.\,G.}, {\sc Weinstein, J.\,N.}, {\sc Botstein, D.} and {\sc Brown, P.\,O.}  (2000).
    Systematic variation in gene expression patterns in human cancer cell lines.
    {\it Nature Genetics}, \textbf{24}, 227--235.

\bibitem[{Shangguan et~al.(2012)}]{shangguan2012inhibition}
	{\sc Shangguan, L.}, {\sc Ti, X.}, {\sc Krause, U.}, {\sc Hai, B.}, {\sc Zhao, Y.}, {\sc Yang, Z.} and  {\sc Liu, F.} (2012).
	Inhibition of {TGF}-$\beta$/Smad signaling by {BAMBI} blocks differentiation of human mesenchymal stem cells to carcinoma-associated fibroblasts and abolishes their protumor effects.
	{\it Stem Sells}, \textbf{30} 2810--2819.
	
\bibitem[{Shankavaram et~al.(2007)}]{shankavaram2007transcript}
	{\sc Shankavaram, U.\,T.}, {\sc Reinhold, W.\,C.}, {\sc Nishizuka, S.}, {\sc Major, S.}, {\sc Morita, D.}, {\sc Chary, K.\,K.}, {\sc Reimers, M.\,A.}, {\sc Scherf, U.} {\sc Kahn, A.}, {\sc Dolginow, D.}, {\sc Cossman, J.},  {\sc Kaldjian, E.\,P.},  {\sc Scudiero, D.\,A.}, {\sc Petricoin, E.}, {\sc Liotta,  L.}, {\sc Lee, J.\,K.}  and  {\sc Weinstein, J.\,N.} (2007).
	Transcript and protein expression profiles of the NCI-60 cancer cell panel: An integromic microarray study.
	{\it Molecular Cancer Therapeutics}, \textbf{40} 2877--2909. 	
	
\bibitem[Shehata et~al.(2008)]{shehata2008nonredundant}
{\sc Shehata, M., Bi\`{e}che, I., Boutros, R., Weidenhofer, J., Fanayan, S., Spalding, L., Zeps, N., Byth, K., Bright, R.\,K., Lidereau, R.} and {\sc Byrne, J.\,A.}  (2008).
Nonredundant functions for tumor protein D52-like proteins support specific targeting of TPD52.
{\it Clinical Cancer Research}, \textbf{14} 5050--5060.

	
\bibitem[{Tibshirani(1996)}]{tibs1996regression}
\textsc{Tibshirani, R.} (1996).
\newblock {Regression shrinkage and selection via the lasso}.
\newblock \textit{Journal of the Royal Statistical Society, {\rm Series B}}, \textbf{58} 267--288.

\bibitem[{Tukey(1975)}]{T1975}
\textsc{Tukey, J.\,W.} (1975).
\newblock Mathematics and the picturing of data.
\newblock In \textit{Proceedings of the International Congress of Mathematicians}, {\bf 2} 523--531.

\bibitem[{Wainwright(2009)}]{wainwright2009ieee}
\textsc{Wainwright, M.\,J.} (2009).
\newblock {Sharp thresholds for high-dimensional and noisy sparsity recovery using $\ell_1$-constrained quadratic programming (Lasso)}.
\newblock \textit{IEEE Transactions on Information Theory}, \textbf{55} 2183--2202.


\bibitem[Wang(2013)]{W2013}
{\sc Wang, L.} (2013).
The $L_1$ penalized LAD estimator for high dimensional linear regression.
{\it Journal of Multivariate Analysis}, \textbf{120} 135--151.


\bibitem[Wang, Peng and Li(2015)]{Lan2015}
\textsc{Wang, L.}, \textsc{Peng, B.} and \textsc{Li, R.} (2015).
\newblock {A high-dimensional nonparametric multivariate test for mean vector}.
\newblock \textit{Journal of the American Statistical Association}, \textbf{110} 1658--1669.


\bibitem[Wang, Wu and Li(2012)]{wang2012quantile}
\textsc{Wang, L.}, \textsc{Wu, Y.} and \textsc{Li, R.} (2012).
\newblock {Quantile regression for analyzing heterogeneity in ultra-high dimension}.
\newblock \textit{Journal of the American Statistical Association}, \textbf{107} 214--222.



\bibitem[Wu et~al.(2008)]{wu2008overlapping}
{\sc Wu, Y., Siadaty, M.\,S., Berens, M.\,E., Hampton, G.\,M.} and {\sc Theodorescu, D.} (2008).
Overlapping gene expression profiles of cell migration and tumor invasion in human bladder cancer identify metallothionein E1 and nicotinamide N-methyltransferase as novel regulators of cell migration.
{\it Oncogene}, \textbf{27} 6679--6689.

\bibitem[Yohai and Maronna(1979)]{YohaiMaronna1979}
\textsc{Yohai, V.\,J.} and \textsc{Maronna, R.\,A.} (1979).
\newblock {Asymptotic behavior of $M$-estimators for the linear model}.
\newblock \textit{The Annals of Statistics}, \textbf{7} 258--268.

\bibitem[Zheng, Peng and He(2015)]{ZPH2015}
\textsc{Zheng, Q.}, \textsc{Peng, L.} and \textsc{He, X.} (2015).
\newblock Globally adaptive quantile regression with ultra-high dimensional data.
\newblock \textit{The Annals of Statistics}, \textbf{43} 2225--2258.

\bibitem[Zhou et~al.(2017)]{zhou2017silencing}
{\sc Zhou, T.}, {\sc Li, Y.}, {\sc Yang, L.}, {\sc Liu, L.}, {\sc Ju, Y.} and {\sc Li, C.} (2017).
Silencing of ANXA3 expression by RNA interference inhibits the proliferation and invasion of breast cancer cells.
{\it Oncology Reports}, \textbf{37} 388-398.

\bibitem[Zuo and Serfling(2000)]{ZS2000}
\textsc{Zuo, Y.} and \textsc{Serfling, R.} (2000).
\newblock General notions of statistical depth function.
\newblock \textit{The Annals of Statistics}, \textbf{28} 461--482.

\end{thebibliography}

\begin{thebibliography}{11}


 \bibitem[Belloni and Chernozhukov(2011)]{BC2011}
	{\sc Belloni, A.} and {\sc Chernozhukov, V.} (2011).
	$\ell_1$-penalized quantile regression in high-dimensional sparse models.
	{\it The Annals of Statistics}, {\bf 39} 82--130.

 \bibitem[Bickel, Ritov and Tsybakov(2009)]{BRT2009}
	{\sc Bickel, P.\,J., Ritov, Y.} and {\sc Tsybakov, A.\,B.} (2009).
	Simultaneous analysis of lasso and Dantzig selector.
	{\it The Annals of Statistics}, {\bf 37} 1705--1732.


\bibitem[{Boucheron, Lugosi and Massart(2013)}]{BLM2013}
{\sc Boucheron, S., Lugosi, G.} and {\sc Massart, P.} (2013).
{\it Concentration Inequalities: A Nonasymptotic Theory of Independence.}
Oxford University Press, Oxford.

\bibitem[{Bousquet(2003)}]{B2003}
{\sc Bousquet, O.} (2003).
Concentration inequalities for sub-additive functions using the entropy method.
{\it In Stochastic Inequalities and Applications. Progress in Probability} {\bf 56} 213--247. Birkh\"auser, Basel.

\bibitem[Fan et~al.(2018)]{fan2015tac}
\textsc{Fan, J.}, \textsc{Liu, H.}, \textsc{Sun, Q.} and \textsc{Zhang, T.}
  (2018).
\newblock {I-LAMM} for sparse learning: Simultaneous control of algorithmic
  complexity and statistical error. {\em The Annals of Statistics},  \textbf{96} 1348--1360.

  \bibitem[{Ledoux and Talagrand(1991)}]{LT1991}
{\sc Ledoux, M.} and {\sc Talagrand, M.} (1991).
{\it Probability in Banach Spaces: Isoperimetry and Processes}.
Springer-Verlag, Berlin.

  \bibitem[{Lepski(1991)}]{Lepski1991}
\textsc{Lepski, O.\,V.} (1991).
Asymptotically minimax adaptive estimation. I. Upper bounds. Optimally adaptive estimates.
{\it IEEE Transactions on Information Theory}, {\bf 36} 682--697.


\bibitem[{Loh and Wainwright(2015)}]{LW2015}
{\sc Loh, P.-L.} and {\sc Wainwright, M.\,J.} (2015).
Regularized $M$-estimators with nonconvexity: Statistical and algorithmic theory for local optima.
{\it Journal of Machine Learning Research}, {\bf 16} 559--616.

 \bibitem[Negahban et al.(2012)]{NRWY2012}
	{\sc Negahban, S.\,N., Ravikumar, P., Wainwright, M.\,J.} and {\sc Yu, B.} (2012).
	A unified framework for high-dimensional analysis of $M$-estimators with decomposable regularizers.
	{\it Statistical Science}, {\bf 27} 538--557.
	
\bibitem[{Rudelson and Zhou(2013)}]{RZ2013}
{\sc Rudelson, M.} and {\sc Zhou, S.} (2013).
Reconstruction from anisotropic random measurements.
{\it IEEE Transactions on Information Theory}, {\bf 59} 3434--3447.

\bibitem[Spokoiny(2013)]{spokoiny2013bernstein}
	{\sc Spokoiny, V.} (2013).
	Bernstein--von Mises theorem for growing parameter dimension. Preprint.
	Available at \href{https://arxiv.org/abs/1302.3430}{arXiv:1302.3430.}


 \bibitem[van de Geer(2008)]{vdG2008}
	{\sc van de Geer, S.\,A.} (2008).
	High-dimensional generalized linear models and the lasso.
	{\it The Annals of Statistics}, {\bf 36} 614--645.

\bibitem[Wainwright(2009)]{W2009}
	{\sc Wainwright, M.\,J.} (2009).
	Sharp thresholds for high-dimensional and noisy sparsity recovery using $\ell_1$-constrained quadratic programming (Lasso).
	{\it IEEE Transactions on Information Theory}, {\bf 55} 2183--2202.
	
\end{thebibliography}
\end{document}